\documentclass[10pt]{amsart}
\usepackage{fullpage,url,amssymb,amsmath,amsfonts,amsthm,mathrsfs}
\usepackage[all]{xy} \SelectTips{cm}{10}

\usepackage{mathpazo}

\usepackage{color}
\definecolor{webcolor}{rgb}{0.8,0,0.2}
\definecolor{webbrown}{rgb}{.6,0,0}
\usepackage[
        colorlinks,
        linkcolor=webbrown,  filecolor=webcolor,  citecolor=webbrown, 
        backref,
        pdfauthor={David Zywina}, 
]{hyperref}
\usepackage[alphabetic,backrefs,lite]{amsrefs} 

\numberwithin{equation}{section}

\renewcommand{\AA}{\mathbb A}

\newcommand{\CC}{\mathbb C}

\newcommand{\FF}{\mathbb F}
\newcommand{\GG}{\mathbb G}

\newcommand{\PP}{\mathbb P}
\newcommand{\QQ}{\mathbb Q}
\newcommand{\RR}{\mathbb R}

\newcommand{\ZZ}{\mathbb Z} 
\newcommand{\Zhat}{\widehat\ZZ}

\newcommand{\OO}{\mathcal O}
\newcommand{\calG}{\mathcal G}  \newcommand{\calF}{\mathcal F}

\newcommand{\calB}{\mathcal B}
\newcommand{\calS}{\mathcal S}

\newcommand{\calA}{\mathcal A}
\newcommand{\calC}{\mathcal C}

\newcommand{\calZ}{\mathcal Z}
\newcommand{\calL}{\mathcal L}
\newcommand{\calD}{\mathcal D}
\newcommand{\calW}{\mathcal W}

\newcommand{\calM}{\mathcal M}

\newcommand{\calU}{\mathcal U}
\newcommand{\calV}{\mathcal V}
\newcommand{\calY}{\mathcal Y}

\newcommand{\p}{\mathfrak p}

\newcommand{\scrS}{\mathscr S}

\newcommand{\Gr}{\operatorname{Gr}}

\newcommand{\norm}[1]{ \left|\!\left| #1 \right|\!\right|  }

\def\cyc{{\operatorname{cyc}}}
\def\ab{{\operatorname{ab}}}
\def\un{{\operatorname{un}}}

\def\conn{{\operatorname{conn}}}
\def\rank{{\operatorname{rank}}}
\def\Res{{\operatorname{Res}}}

\def\Spec{\operatorname{Spec}} 
\def\Gal{\operatorname{Gal}}
 
\def \GL {\operatorname{GL}}  

\def \GSp {\operatorname{GSp}}  

\def \SL {\operatorname{SL}}

\def \Sp {\operatorname{Sp}}
\def\Aut{\operatorname{Aut}} 
\def\End{\operatorname{End}}
\def\Frob{\operatorname{Frob}}

\def\sc{{\operatorname{sc}}}
\def\ad{{\operatorname{ad}}}
\newcommand{\q}{\mathfrak q}

\newcommand{\defi}[1]{\textsf{#1}} 

\newcommand\blank[1]{}

\def\bbar#1{\setbox0=\hbox{$#1$}\dimen0=.2\ht0 \kern\dimen0 
\overline{\kern-\dimen0 #1}}
\newcommand{\Qbar}{{\overline{\mathbb Q}}} 
\newcommand{\Kbar}{{\bbar{K}}}

\newcommand{\FFbar}{\overline{\FF}} 

\newtheorem{thm}{Theorem}[section]
\newtheorem{lemma}[thm]{Lemma}
\newtheorem{cor}[thm]{Corollary}
\newtheorem{prop}[thm]{Proposition}

\theoremstyle{definition}
\newtheorem{definition}[thm]{Definition}

\theoremstyle{remark}
\newtheorem{remark}[thm]{Remark}
\newtheorem{example}[thm]{Example}

\newenvironment{romanenum}{\hfill \begin{enumerate} }{\end{enumerate}}
\newenvironment{alphenum}{\hfill \begin{enumerate} }{\end{enumerate}}

\begin{document}

\title[Families of abelian varieties and large Galois images]{Families of abelian varieties and large Galois images}
\author{David Zywina}
\address{Department of Mathematics, Cornell University, Ithaca, NY 14853, USA}
\email{zywina@math.cornell.edu}
\urladdr{http://www.math.cornell.edu/~zywina}


\begin{abstract}
Associated to an abelian variety $A$ of dimension $g$ over a number field $K$ is a Galois representation $\rho_A\colon \Gal(\Kbar/K)\to \GL_{2g}(\Zhat)$.  The representation $\rho_A$ encodes the Galois action on the torsion points of $A$ and its image is an interesting invariant of $A$ that contains a lot of arithmetic information.  We consider abelian varieties over $K$ parametrized by the $K$-points of a nonempty  open subvariety $U\subseteq \PP^n_K$.  We show that away from a set of density $0$, the image of $\rho_A$ will be very large; more precisely, it will have uniformly bounded index in a group obtained from the family of abelian varieties.  This generalizes earlier results which assumed that the family of abelian varieties had ``big monodromy''.   We also give a version for a family of abelian varieties with a more general base.
\end{abstract}

\maketitle

\section{Introduction}

Fix an abelian scheme $\pi\colon A\to U$ of relative dimension $g\geq 1$, where $U$ is a non-empty open subvariety of $\PP^n_K$ with $K$ a number field and $n\geq 1$.  Choose an algebraic closure $\Kbar$ of $K$ and define the absolute Galois group $\Gal_K:=\Gal(\Kbar/K)$.  
 
Take any point $u\in U(K)$.  The fiber of $\pi$ over $u$ is an abelian variety $A_u$ over $K$ of dimension $g$.    For each positive integer $m$, let $A_u[m]$ be the $m$-torsion subgroup of $A(\Kbar)$.   The group $A_u[m]$ is a free $\ZZ/m\ZZ$-module of rank $m$ and has a natural $\Gal_K$-action.   This Galois action can  expressed in terms of a representation
\[
\bbar\rho_{A,m} \colon \Gal_K \to \Aut_{\ZZ/m\ZZ}(A[m]).
\]
Taking the inverse limit over all $m$, ordered by divisibility, we obtain a single representation
\[
\rho_{A_u} \colon \Gal_K \to \Aut(\varprojlim A_u[m]) \cong \GL_{2g}(\Zhat)
\]
that encodes the Galois action on the torsion of $A_u$, where $\Zhat$ is the profinite completion of $\ZZ$.  We are interested in describing how large the image of $\rho_{A_u}$ can be as we vary the point $u\in U(K)$.   \\

We first observe that the abelian scheme $A$ imposes a constraint on the image of $\rho_{A_u}$.  Let $\pi_1(U,\bbar\eta)$ be the \'etale fundamental group of $U$, where $\bbar\eta$ is a fixed geometric generic point of $U$.   For each positive integer $m$, let $A[m]$ be the $m$-torsion subscheme of $A$.  The morphism $A[m]\to U$ can be viewed as locally constant sheaf of $\ZZ/m\ZZ$-modules on $U$ that is free of rank $2g$; it thus corresponds to a representation $\bbar\rho_{A,m} \colon \pi_1(U,\bbar\eta) \to \Aut_{\ZZ/m\ZZ}(A[m]_{\bbar\eta})$,
where the group $A[m]_{\bbar\eta}$ is the fiber of $A[m]$ above ${\bbar\eta}$.  Taking the inverse limit over all $m$, ordered by divisibility, we obtain a single representation
\[
\rho_A \colon \pi_1(U,\bbar\eta) \to \Aut(\varprojlim A[m]_{\bbar\eta}) \cong \GL_{2g}(\Zhat).
\]
Specialization at a point $u\in U(K)$ induces a homomorphism $u_*\colon \Gal_K \to \pi_1(U,\bbar\eta)$, uniquely defined up to conjugacy.  Composing $u_*$ with $\rho_A$, we obtain a representation $\Gal_K \to \GL_{2g}(\Zhat)$ that agrees with $\rho_{A_u}$ up to an inner automorphism of $\GL_{2g}(\Zhat)$.   So we may identity $\rho_{A_u}$ with the specialization of $\rho_A$ at $u$.  In particular, we can view $\rho_{A_u}(\Gal_K)$ as a subgroup of $\rho_A(\pi_1(U,\bbar\eta))$ that is uniquely defined up to conjugation.      Suppressing the base point, we have 
\[
\rho_{A_u}(\Gal_K)\subseteq \rho_A(\pi_1(U))
\] 
for all $u\in U(K)$.  Our main result is that the index $[\rho_A(\pi_1(U)):\rho_{A_u}(\Gal_K)]$ is finite and bounded as we vary over ``most'' $u\in U(K)$.   Our notion of ``most'' will be that of density.   Let $H$ be the absolute multiplicative height function on $\PP^n(K)$. The \defi{density} of a set  $B\subseteq \PP^n(K)$ is the value
\[
\lim_{x\to + \infty} \frac{|\{ u \in B : H(u)\leq x\}|}{ |\{ u \in \PP^n(K) : H(u)\leq x\}|}
\] 
if the limit exists.  For example, $U(K)$ has density $1$.   Our main theorem is the following:

\begin{thm} \label{T:MAIN}
Fix an abelian scheme $\pi\colon A\to U$ of positive relative dimension, where $U$ is non-empty open subvariety of $\PP^n_K$ for a number field $K$ and $n\geq 1$.  Then there is a constant $C$  such that 
\[
\big[\rho_A(\pi_1(U)): \rho_{A_u}(\Gal_K)\big] \leq C
\]
holds for all $u\in U(K)$ in a set of density $1$.  
\end{thm}

Theorem~\ref{T:MAIN} shows that, up to bounded index, the image of the specializations are usually as large as possible when the geometric constraints are taken into account.  This can thus be viewed as a variant of Hilbert's irreducibility theorem (and effective versions of Hilbert's irreducibility theorem will be a key component of our proof).  This is useful since in practice, $\rho_A(\pi_1(U))$ is easier to compute that the images of the representations $\rho_{A_u}$ (one reason is that there are geometric and topological approaches to computing the normal subgroup $\rho_{A}(\pi_1(U_{\Kbar}))$).  

However, note that our theorem is not a formal consequence of Hilbert's irreducibility theorem since $\rho_A(\pi_1(U))$ is not finitely generated when viewed as a topological group with the profinite topology.   Moreover, the constant $C$ cannot alway be taken to be $1$.   As an example, take $K=\QQ$ and consider any abelian scheme $A\to U:=\AA_\QQ^1-\{0,1728\}$ of relative dimension $1$ such that each fiber $A_u$ is an elliptic curve with $j$-invariant $u$.  In this case, we have $\rho_A(\pi_1(U))=\GL_2(\Zhat)$.   The theorem cannot hold with $C=1$ since from Serre we know that $\rho_{E}(\Gal_\QQ)\neq \GL_2(\Zhat)$ for all elliptic curves $E$ over $\QQ$, cf.~Proposition~22 of \cite{MR387283}.   For this example, the theorem will hold with $C=2$.

There are several special cases of Theorem~\ref{T:MAIN} occurring in the literature and we will recall some in \S\ref{SS:earlier}.   These earlier results have a strong constraint on the image of $\rho_A$; more precisely, they assume that $\rho_A(\pi_1(U))$ is an open subgroup of $\GSp_{2g}(\Zhat)$.   The main novelty of Theorem~\ref{T:MAIN} is the lack of restrictions on our abelian scheme $A\to U$.   Since we have less control on the image of $\rho_A$, the group theory involved is much more complicated; for example, the $\ell$-adic monodromy groups need not be connected and their derived subgroups need not be simply connected.

The constant $C$ in Theorem~\ref{T:MAIN} that occurs in our proof will be given in \S\ref{SS:the constant C}.   We have not tried to determine the optimal $C$.

\subsection{General base}

Fix a number field $K$.  Let $\pi\colon A\to X$ be an abelian scheme of relative dimension $g\geq 1$, where $X$ is a smooth and geometrically integral variety defined over $K$ of dimension $n\geq 1$.   As before, we can define a representation
\[
\rho_A \colon \pi_1(X)\to \GL_{2g}(\Zhat).
\]

Take any closed point $x$ of $X$.  The residue field $k(x)$ of $x$ is a finite extension of $K$.   The fiber of $A$ over $x$ is an abelian variety $A_x$ over $k(x)$.   Associated to $A_x$, we have a representation $\rho_{A_x} \colon \Gal_{k(x)} \to \GL_{2g}(\Zhat)$ whose image we may again view as a subgroup of $\rho_A(\pi_1(X))$.  The following theorem says that there are infinitely many closed points $x$ of $X$ of bounded degree such that $\rho_{A_x}(\Gal_{k(x)})$ is large.

\begin{thm} \label{T:general base}
There are constants $d$ and $C$ such that there are infinitely many closed points $x$ of $X$ satisfying $[k(x):K] \leq d$ and $[\rho_A(\pi_1(X)): \rho_{A_x}(\Gal_{k(x)})] \leq C$.
\end{thm}

The theorem can fail if we insist that $d=1$; for example, consider the case where $X$ is a curve of genus at least $2$ and hence $X(K)$ is finite.

We will deduce Theorem~\ref{T:general base} from Theorem~\ref{T:MAIN}.   The idea is to  find, after possibly shrinking $X$,  an \'etale map $X\to U$, where $U$ is open in $\PP^n_k$.  We then apply Theorem~\ref{T:MAIN} to the restriction of scalars of $A$ from $X$ to $U$.

\subsection{The constant $C$} \label{SS:the constant C}

We now give a brief description of the constant $C$ from Theorem~\ref{T:MAIN} that occurs in our proof, see Remark~\ref{R:specific constant}.  In particular, observe that the constant $C$ can be computed directly from $H:=\rho_{A}(\pi_1(U))$ and its normal subgroup $H_g:=\rho_{A}(\pi_1(U_\Kbar))$.

Take any prime $\ell$ and let $p_\ell\colon \GL_{2g}(\Zhat)\to \GL_{2g}(\ZZ_\ell)$ be the $\ell$-adic projection.   Let $G_\ell$ be the Zariski closure in $\GL_{2g,\QQ_\ell}$ of $p_\ell(\rho_A(\pi_1(U)))$; it is an algebraic group over $\QQ_\ell$ whose neutral component we denote by $G_\ell^\circ$.   Let $M$ be the kernel of the homomorphism
\[
H \xrightarrow{p_\ell} G_\ell(\QQ_\ell) \to G_\ell(\QQ_\ell)/G_\ell^\circ(\QQ_\ell).
\]
We will show later that $M$ does not depend on the choice of $\ell$.   The commutator subgroup $M'$ of $M$ is normal in $H$.  We will see that the image of the quotient map
\[
H_g \to H/M'
\]
is finite and its cardinality is our constant $C$. 

\begin{example}
As a special case, consider when $H=\GSp_{2g}(\Zhat)$ and $H_g=\Sp_{2g}(\Zhat)$; these are the largest that both $H$ and $H_g$ could possibly be, up to conjugation in $\GL_{2g}(\Zhat)$,  when the abelian scheme $A$ is principally polarized.   We have $M=H$ since $G_\ell=\GSp_{2g,\QQ_\ell}$ is connected.   Therefore, $C$ is the cardinality of the group $\Sp_{2g}(\Zhat)/\GSp_{2g}(\Zhat)'$.  One can show that $C=1$ if $g\geq 3$ and $C=2$ if $g=1$ or $2$.  
\end{example}

\subsection{Some earlier results} \label{SS:earlier}

We now discuss special cases of Theorem~\ref{T:MAIN} that have already been proved and some related results.   Note that some of the results use integral points instead of rational points.  In the results mentioned, the term \emph{most} will refer to a suitable notion of density; the reader is encouraged to look at the corresponding articles for the precise definitions. \\

We first discuss the fundamental case $g=1$, i.e., $A$ is an elliptic curve.   For a non-CM elliptic curve $E/\QQ$, \emph{Serre's open image theorem} says that $\rho_E(\Gal_\QQ)$ is an open subgroup of $\GL_2(\Zhat)$, cf.~\cite{MR387283}.  In particular, $\bbar\rho_{E,\ell}(\Gal_K)=\GL_2(\ZZ/\ell\ZZ)$ for all primes $\ell\geq c_E$, where $c_E$ is a constant depending on $E$.   As a consequence of Serre's theorem, if $A$ is non-isotrivial then $\rho_A(\pi_1(U))$ is an open subgroup of $\GL_2(\Zhat)$.\\

The following are all with respect to the family $A\to U:=\Spec \QQ[a,b, (4a^3+27b^2)^{-1}]$ of elliptic curves given by the Weierstrass equation $y^3=x^3+ax+b$.  In this case, we have $\rho_A(\pi_1(U))=\GL_2(\Zhat)$ and $\rho_A(\pi_1(U_{\Qbar}))=\SL_2(\Zhat)$.

Duke \cite{MR1485897} proved that for ``most'' elliptic curves $E/\QQ$, we have $\bbar\rho_{E,\ell}(\Gal_K)=\GL_2(\ZZ/\ell\ZZ)$ for \emph{all} primes.  Building on this, Jones \cite{MR2563740} showed that for ``most'' elliptic curves $E/\QQ$, $\rho_E(\Gal_\QQ)$ is an index $2$ subgroup of $\GL_2(\Zhat)$ (as already noted, $\rho_E$ is never surjective for an elliptic curve over $\QQ$).   Similar results for one-parameter families of elliptic curves over $\QQ$ can be found in \cite{MR2837018}.

For a number field $K\neq \QQ$, Zywina showed that for ``most'' elliptic curves $E/K$, we have $\rho_E(\Gal_K) = \{ B\in \GL_2(\Zhat) : \det(B) \in \chi_{\cyc}(\Gal_K) \}$, where $\chi_\cyc\colon \Gal_\QQ \to \Zhat^\times$ is the cyclotomic character.   In particular, if $K\neq \QQ$ contains no non-trivial abelian extension of $\QQ$, then there is an elliptic curve $E/K$ with $\rho_E(\Gal_K)=\GL_2(\Zhat)$.  
 Greicius \cite{MR2778661} had previously constructed an explicit elliptic curve $E$ over a number field with $\rho_E$ surjective.\\
 
We now consider $g\geq 2$ and assume further that $A$ is principally polarized.   After possibly conjugating $\rho_E$ by an element in $\GL_{2g}(\Zhat)$ we may assume that $\rho_A(\pi_1(U))$ is a subgroup of $\GSp_{2g}(\Zhat)$.  Under the ``big monodromy'' assumption that $\rho_A(\pi_1(U))$ is an open subgroup of $\GSp_{2g}(\Zhat)$,  Landesman, Swaminathan, Tao and Xu proved Theorem~\ref{T:MAIN} with an optimal $C$.   Earlier, Wallace \cite{MR3263960} had proved a variant of this with $g=2$;   also see Remark~1.3 in \cite{REU}.  The case $g=1$ had been proved in  \cite{1011.6465}.

When $K=\QQ$, one can use the Kronecker--Weber theorem to show that $\rho_{A_u}(\Gal_\QQ) \cap \GSp_{2g}(\Zhat)$ agrees with the commutator subgroup of $\rho_{A_u}(\Gal_\QQ)$; this and the inclusion $\rho_{A_u}(\Gal_K) \subseteq \rho_A(\pi_1(U))$ are the only constraints on the image of $\rho_{A_u}$ for ``most'' $u\in U(K)$.   In the general setting of Theorem~\ref{T:MAIN}, there may be additional constraints on the image of all the representations $\rho_{A_u}$.\\

In the setting of Theorem~\ref{T:MAIN}, define the set
\[
S:=\{u\in U(K): \rho_{A_u}(\Gal_K) \text{ is \emph{not} an open subgroup of } \rho_A(\pi_1(U)) \}.
\]
An immediate consequence of Theorem~\ref{T:MAIN} is that the set $S$ density $0$.   Moreover, Cadoret proved that the set $S$ is \emph{thin} in $U(K)$, cf.~Theorem~1.2 and \S1.1 of \cite{MR3455865}.   Recall that every thin subset of $U(K)$ has density $0$, cf.~\S13.1 of \cite{MR1757192}.

\subsection{Overview}
We now give a brief overview of the proof of Theorem~\ref{T:MAIN}.

For a rational prime $\ell$, let $\rho_{A,\ell} \colon \pi_1(U)\to \GL_{2g}(\ZZ_\ell)$ be the representation obtained by composing $\rho_A$ with the natural projection $\GL_{2g}(\Zhat)\to \GL_{2g}(\ZZ_\ell)$.   We define $\calG_{A,\ell}$ be the $\ZZ_\ell$-group subscheme of $\GL_{2g,\ZZ_\ell}$ obtained by taking the Zariski closure of $\rho_{A,\ell}(\pi_1(U))$.  These will agree with later definitions of $\rho_{A,\ell}$ and $\calG_{A,\ell}$ after choosing an appropriate $\ZZ_\ell$-basis for the $\ell$-adic Tate module $T_\ell(A)$.      The generic fiber $G_{A,\ell}$ of $\calG_{A,\ell}$ is an algebraic group over $\QQ_\ell$ that we call the \emph{$\ell$-adic monodromy groups} of $A$.   In \S\ref{S:monodromy in families},  we recall several properties of $G_{A,\ell}$.  

   We have an easy inclusion $\bbar\rho_{A,\ell}(\pi_1(U))\subseteq \calG_{A,\ell}(\FF_\ell)$.     Theorem~\ref{T:geometric gp theory}, which is a generalization of Serre's open image theorem, implies that 
   \[
   [\calG_{A,\ell}(\FF_\ell): \bbar\rho_{A,\ell}(\pi_1(U))] \leq C
\]
for a constant $C$ that does not depend on $\ell$.  There is a constant $b_A$ such that the neutral component of the algebraic group $(\calG_{A,\ell})_{\FF_\ell}$ over $\FF_\ell$ is reductive for all $\ell \geq b_A$.  For $\ell\geq b_A$, let $H_\ell$ is the derived subgroup of the neutral component of the group $(\calG_{A,\ell})_{\FF_\ell}$ and let $S_\ell$ be the commutator subgroup of $H_\ell(\FF_\ell)$ (in the notation of \S\ref{S:big ell-adic images}, $S_\ell=\calS_{A,\ell}(\FF_\ell)'$).  

    After possibly increasing $b_A$, we will observe that
   \[
   S_\ell \subseteq \bbar\rho_{A,\ell}(\pi_1(U))
   \]    
   holds for all $\ell\geq b_A$.     In the special case of the examples in \S\ref{SS:earlier} where $\rho_A(\pi_1(U))$ is an open subgroup of $\GSp_{2g}(\Zhat)$, we find that $\calG_{A,\ell}=\GSp_{2g,\ZZ_\ell}$ and $S_\ell =\Sp_{2g}(\FF_\ell)$ for all sufficiently large $\ell$.\\
  
Fix a prime $\ell\geq b_A$ and a point $u\in U(K)$, specialization gives an inclusion $\bbar\rho_{A_u,\ell}(\Gal_K) \subseteq \bbar\rho_{A,\ell}(\pi_1(U))$ uniquely determined up to conjugation.  Since $S_\ell$ is a normal subgroup of $\calG_{A,\ell}(\FF_\ell)$, and hence also of $\bbar\rho_{A,\ell}(\pi_1(U))$, it makes sense to ask whether or not $S_\ell$ is a subgroup of $\bbar\rho_{A_u,\ell}(\Gal_K)$.  Define the set 
 \[
B:=\{ u\in U(K) : \bbar\rho_{A_u,\ell}(\Gal_K)\not\supseteq S_\ell \text{ for some prime }\ell\geq b_A\}.
\]
One of the main tasks of this paper is to show that $B$ has density $0$; equivalently, that for ``most'' $u\in U(K)$, we have $\bbar\rho_{A,\ell}(\Gal_K) \supseteq S_\ell$ for \emph{all} $\ell\geq b_A$.   In \S\ref{S:main reduction},  we prove that if the set $B$ has density $0$, then Theorem~\ref{T:MAIN} will hold.   This will require some information about the groups $\rho_{A,\ell}(\pi_1(U_{\Kbar}))$ which we study in \S\ref{S:geometric monodromy}.   \\

We will prove Theorem~\ref{T:MAIN} in \S\ref{S:main proof}.    Take any real number $x\geq 2$.  Let $B(x)$ be the set of $u \in B$ satisfying $H(u)\leq x$.  To prove that $B$ has density $0$, we need to show that $|B(x)|=o(x^{[K:\QQ](n+1)})$ as $x\to \infty$.  For each $\ell\geq b_A$, let $B_\ell(x)$ be the set of $u\in U(K)$ with $H(u)\leq x$ satisfying $\bbar\rho_{A_u,\ell}(\Gal_K)\not\supseteq S_\ell$.  Using an effective open image theorem for abelian varieties, we will show that
\begin{align} \label{E: B intro}
B(x) \subseteq R(x) \cup T(x)  \cup \bigcup_{b_A\leq \ell \leq c(\log x)^\gamma} B_\ell(x).
\end{align}
for some constant $c$, where the sets $R(x)$ and $T(x)$ are defined at the beginning of \S\ref{S:main proof}.   The important aspect of the inclusion (\ref{E: B intro}) is that the right hand side involves only a bounded number of primes $\ell$ while the definition of $B$ requires considering all primes $\ell\geq b_A$.   In \S\ref{SS:bounding R(x)} and \S\ref{SS:bounding T(x)}, we show that $|R(x)| = o(x^{[K:\QQ](n+1)})$ and $|T(x)| = o(x^{[K:\QQ](n+1)})$.   So from (\ref{E: B intro}), we have
\begin{align} \label{E: |B| intro}
|B(x)| \leq \sum_{b_A\leq \ell \leq c(\log x)^\gamma} |B_\ell(x)|  + o(x^{[K:\QQ](n+1)}).
\end{align}
So we need to find bounds for $|B_\ell(x)|$.  The Hilbert Irreducibility Theorem (HIT) implies that $|B_\ell(x)| = o(x^{[K:\QQ](n+1)})$ as $x\to \infty$.  However, to use (\ref{E: |B| intro}) to find meaningful bounds for $|B(x)|$, we need to find better bounds for $|B_\ell(x)|$ with an  explicit dependency on $\ell$.\\

In \S\ref{S:explicit HIT}, we prove an effective version of HIT using the large sieve.  In \S\ref{SS:special HIT}, we use it to give a more specialized version of HIT that is relevant to our application.  To obtain the explicit bounds we require, we will need some group theoretic input which is discussed in \S\ref{S:derangements}.    For each prime $\ell\geq b_A$ and  $x\geq 2$, Theorem~\ref{T:almost there mod ell} says that 
\[
|B_\ell(x)| = O\Big( (\ell+1)^{3g(2g+1)/2} \cdot x^{[K:\QQ](n + 1/2)} \log x + (\ell+1)^{(6n+15/2)g(2g+1)} \Big),
\]   
where the implicit constant depends only on $A$.  Combining this with (\ref{E: |B| intro}) gives $|B(x)|=o(x^{[K:\QQ](n+1)})$ and hence $B$ has density $0$.

Finally in \S\ref{S:general base proof}, we prove Theorem~\ref{T:general base} by reducing to Theorem~\ref{T:MAIN}.

\subsection{Notation} \label{SS:notation}

Consider a topological group $G$.    The \defi{commutator subgroup} of $G$ is the closed subgroup $G'$ generated by the set of commutators of $G$.   We say that $G$ is \defi{perfect} if $G'=G$.   Profinite groups, and in particular finite groups, will always be considered with their profinite topology. 

For a scheme $X$ over a commutative ring $R$ and a commutative $R$-algebra $S$, we denote by $X_S$ the base extension of $X$ by $\Spec S$.  
  Let $M$ be a free module of finite rank over a commutative ring $R$.  Denote by $\GL_M$ the $R$-scheme such that $\GL_M(S)= \Aut_S(M\otimes_R S)$ for any commutative $R$-algebra $S$ with the obvious functoriality. 

For an algebraic group $G$ over a field $F$, we denote by $G^\circ$ the neutral component of $G$, i.e., the connected component of the identity of $G$.  Note that $G^\circ$ is an algebraic subgroup of $G$. 

For two real quantities $f$ and $g$, the expression $f \ll_{\alpha_1,\ldots,\alpha_n} g$ means that the inequality $|f| \leq C |g|$ holds for some positive constant $C$ depending only on $\alpha_1,\ldots,\alpha_n$.  In particular, $f\ll g$ means that the implicit constant $C$ is absolute.  We denote by $O_{\alpha_1,\ldots,\alpha_n}(g)$ a quantity $f$ satisfying $f \ll_{\alpha_1,\ldots,\alpha_n} g$.  For two real valued functions $f$ and $g$ of a real variable $x$, with $g(x)$ non-zero for all sufficiently large $x$, the expression $f=o(g)$ means that $f(x)/g(x)\to 0$ as $x\to +\infty$.

For a number field $K$, we denote by $\OO_K$ the ring of integers of $K$.   For a non-zero prime ideal $\p$ of $\OO_K$, we define its residue field $\FF_\p:=\OO_K/\p$.   For a representation $\rho\colon  \Gal_K \to G$ unramified at a prime $\p$, we will view $\rho(\Frob_\p)$ as either a conjugacy class of $G$ or as an element of $G$ that is uniquely defined up to conjugacy.  Throughout, $\ell$ will always denote a rational prime.

When talking about prime ideals of a number field $K$, density will always refer to \emph{natural density}.  From context, there should be no confusion with the notion of density of subsets of $\PP^n(K)$.

\subsection{Acknowledgements}

Many thanks to David Zureick--Brown;  this article was originally intended to be part of a joint work with him.   Many parts of the original project have been greatly expanded on by his REU students in \cite{REU} and \cite{REU2}.   In particular, an examination of their papers will hopefully make up for the lack of examples in this article.

\section{$\ell$-adic monodromy groups}  \label{S:monodromy in families}

Fix a number field $K$ and an abelian scheme $\pi\colon A\to U$ of relative dimension $g\geq 1$, where $U$ is a non-empty open subvariety of $\PP^n_K$ for some integer $n\geq 0$.  

Note that by including the case $n=0$, the following notation and definitions will also hold for an abelian variety of dimension $g\geq 1$ defined over the number field $K$ (when $n=0$, we have $U=\PP_K^n=\Spec K$ and we can identify $\pi_1(U)$ with $\Gal_K$).\\

We now extend our notation from the introduction.  For an integer $m\geq 2$, let $T_m(A)$ be the inverse limit of the groups $A[m^e]_{\bbar\eta}$ over all $e\geq 1$, where the transition homomorphisms $A[m^{e+1}]_{\bbar\eta}\to A[m^e]_{\bbar\eta}$ are multiplication by $m$.  The group $T_m(A)$ is a free $\ZZ_m$-module of rank $2g$, where $\ZZ_m:=\varprojlim_e\ZZ/m^e\ZZ=\prod_{\ell|m} \ZZ_\ell$.  The representations $\bbar\rho_{A,m^e}$ combine to give a continuous representation
\[
\rho_{A,m}\colon \pi_1(U) \to \Aut_{\ZZ_m}(T_m(A)).
\]

Take any prime $\ell$.   Define $V_\ell(A):= T_\ell(A) \otimes_{\ZZ_\ell} \QQ_\ell$; it is a $\QQ_\ell$-vector space of dimension $2g$.  With notation as in \S\ref{SS:notation},  we have an algebraic group $\GL_{V_\ell(A)}$ defined over $\QQ_\ell$.  We can view $\Aut_{\ZZ_\ell}(T_\ell(A))$, and hence also $\rho_{A,\ell}(\pi_1(U))$, as a subgroup of $\Aut_{\QQ_\ell}(V_\ell(A))=\GL_{V_\ell(A)}(\QQ_\ell)$.

\subsection{$\ell$-adic monodromy groups}

For a prime $\ell$, we have $\rho_{A,\ell}(\pi_1(U)) \subseteq \GL_{V_\ell(A)}(\QQ_\ell)$.  To study the image of $\rho_{A,\ell}$, we will  study a related algebraic group defined over $\QQ_\ell$.

\begin{definition}
The \defi{$\ell$-adic monodromy group} of $A$, which we denote by $G_{A,\ell}$, is the Zariski closure of $\rho_{A,\ell}(\pi_1(U))$ in $\GL_{V_\ell(A)}$; it is an algebraic group defined over $\QQ_\ell$. 
\end{definition}

For any $m\geq 2$ and $u\in U(K)$,  we can view $\rho_{A_u,m}(\Gal_K)$ as a closed subgroup of $\rho_{A,m}(\pi_1(U))$ that is uniquely determined up to conjugation.   With $m=\ell$, we can thus identify $G_{A_u,\ell}$ with a closed algebraic subgroup of $G_{A,\ell}$ uniquely defined up to conjugation by an element in $G_{A,\ell}(\QQ_\ell)$.

\begin{lemma} \label{L:HIT frattini}
Assume that $n\geq 1$.
\begin{romanenum}
\item \label{L:HIT frattini i}
For each integer $m\geq 2$, we have $\rho_{A_u,m}(\Gal_K)=\rho_{A,m}(\pi_1(U))$ for all $u\in U(K)$ away from a set of density $0$.  
\item \label{L:HIT frattini ii}
For each prime $\ell$, we have $G_{A_u,\ell}=G_{A,\ell}$ for all $u\in U(K)$ away from a set of density $0$.
\end{romanenum}
\end{lemma}
\begin{proof}
Part (\ref{L:HIT frattini ii}) is an easy consequence of (\ref{L:HIT frattini i}) with $m=\ell$.  

We now prove (\ref{L:HIT frattini i}).  Define $H:=\rho_{A,m}(\pi_1(U))$.  Let $\Phi(H)$ be the Frattini subgroup of $H$, i.e., the intersection of the maximal closed and proper subgroups of $H$.      The kernel of $H \to \Aut_{\ZZ/m\ZZ}(T_m(A)/m T_m(A))$ is an open subgroup of $H$ that is a product of finitely generated pro-$\ell$ groups with $\ell|m$.  From the proposition of \cite{MR1757192}*{\S10.5}, we find that $\Phi(H)$ is an open, and hence finite index, subgroup of $H$.  

So there is an integer $e\geq 1$ such that $\rho_{A_u,m}(\Gal_K)=H$ if and only if $\bbar\rho_{A_u,m^e}(\Gal_K)=\bbar\rho_{A,m^e}(\pi_1(U))$.   The lemma follows since Hilbert's irreducibility theorem implies that $\bbar\rho_{A_u,m^e}(\Gal_K)=\bbar\rho_{A,m^e}(\pi_1(U))$ holds for all $u \in U(K)$ away from a set of density $0$.
\end{proof}

Note that a priori the implicit set of density $0$ in Lemma~\ref{L:HIT frattini}(\ref{L:HIT frattini ii}) depends on $\ell$.  The following proposition, which we prove in \S\ref{SS:monodromy independence},  removes this dependence on $\ell$.

\begin{prop} \label{P:monodromy independence}
Assume that $n\geq 1$.  The set of $u\in U(K)$ for which $G_{A_u,\ell}=G_{A,\ell}$ holds for all primes $\ell$ has density $1$.
\end{prop}

\subsection{Neutral component} 
 \label{SS:neutral component}
 
Let $G_{A,\ell}^\circ$ be the neutral component of $G_{A,\ell}$, i.e., the connected component of $G_{A,\ell}$ containing the identity.  Note that $G_{A,\ell}^\circ$ is an algebraic subgroup of $G_{A,\ell}$.  Let
\[
\gamma_{A,\ell} \colon \pi_1(U) \to G_{A,\ell}(\QQ_\ell)/G_{A,\ell}^\circ(\QQ_\ell)
\]
be the surjective homomorphism obtained by composing $\rho_{A,\ell}$ with the obvious quotient map.    For any $u\in U(K)$ satisfying $G_{A_u,\ell}=G_{A,\ell}$, the specialization of $\gamma_{A,\ell}$ at $u$ gives the homomorphism $\gamma_{A_u,\ell}\colon \Gal_K \to G_{A_u,\ell}(\QQ_\ell)/G_{A_u,\ell}^\circ(\QQ_\ell)$.

\begin{lemma} \label{L:connected independence families}
\begin{romanenum}
\item \label{L:connected independence families i}
The kernel of $\gamma_{A,\ell}$ is independent of $\ell$. 
\item \label{L:connected independence families ii}
Suppose $n\geq 1$.    Then there is a set $S \subseteq U(K)$ with density $1$ such that  the specialization of 
$\gamma_{A,\ell}$ at $u$ is surjective for all $\ell$ and $u\in S$.
\end{romanenum}
\end{lemma}
\begin{proof}
If $A$ is an abelian variety over $K$, then part (\ref{L:connected independence families i}) was proved by Serre \cite{MR1730973}*{133}; see also \cite{MR1441234}.   We may now assume that $n\geq 1$.

Now suppose that there are primes $\ell$ and $\ell'$ such that $\ker \gamma_{A,\ell} \neq \ker \gamma_{A,\ell'}$.   By Lemma~\ref{L:HIT frattini}(ii) and Hilbert's irreducibility theorem, there is a point $u\in U(K)$ such that $G_{A_u,\ell}=G_{A,\ell}$, $G_{A_u,\ell'}=G_{A,\ell'}$, and such that the kernel of the specializations of $\gamma_{A,\ell}$ and $\gamma_{A,\ell'}$ at $u$ give different subgroups of $\Gal_K$.  So $\ker \gamma_{A_u,\ell}\neq \ker \gamma_{A_u,\ell'}$ which contradicts the case of (\ref{L:connected independence families i}) we have already proved.   Therefore, $\ker \gamma_{A,\ell} = \ker \gamma_{A,\ell'}$ for any primes $\ell$ and $\ell'$.

We now prove (\ref{L:connected independence families ii}).
Let $S$ be the set of $u\in U(k)$ for which the specialization of $\gamma_{A,2}$ at $u$ is surjective.  The set $S$ has density $1$ by Hilbert's irreducibility theorem.   Take any point $u\in S$.  Since each $\gamma_{A,\ell}$ is surjective and $\ker \gamma_{A,\ell}$ is independent of $\ell$, we find that the specialization of $\gamma_{A,\ell}$ at $u$ is surjective for one prime $\ell$ if and only if it is surjective for all primes $\ell$.   Therefore, the specialization of $\gamma_{A,\ell}$ at $u$ is surjective for all $\ell$  by our definition of $S$.
\end{proof}

For an abelian variety $A$ defined over $K$,  we denote by $K_A^\conn$ the subfield of $\Kbar$ fixed by the kernel of the homomorphism 
\begin{align} \label{E:conn hom}
\gamma_{A,\ell}\colon \Gal_K \xrightarrow{\rho_{A,\ell}} G_{A,\ell}(\QQ_\ell)\to G_{A,\ell}(\QQ_\ell)/G_{A,\ell}^\circ(\QQ_\ell).  
\end{align}
Equivalently, $K_{A}^\conn$ is the smallest extension of $K$ in $\Kbar$ that satisfies $\rho_{A,\ell}(\Gal_{K_A^\conn}) \subseteq G_{A,\ell}^\circ(\QQ_\ell)$.   By Lemma~\ref{L:connected independence families}(\ref{L:connected independence families i}), the number field $K_A^\conn$ is independent of $\ell$.

\begin{prop} \label{P:G specialization}
\begin{romanenum}
\item \label{P:G specialization i}
The group $G_{A,\ell}^\circ$ is reductive.
\item \label{P:G specialization ii}
The rank of the reductive group $G_{A,\ell}^\circ$ is independent of $\ell$.
\item \label{P:G specialization iii}
Let $R_\ell$ be the commutant of $G_{A,\ell}^\circ$ in $\End_{\QQ_\ell}(V_\ell(A))$.  The dimension of $R_\ell$ as a $\QQ_\ell$-vector space is independent of $\ell$.  
\item \label{P:G specialization iv}  Suppose that $n\geq 1$.  
The set of $u\in U(K)$ for which $G_{A_u,\ell}^\circ=G_{A,\ell}^\circ$ for all primes $\ell$ has density $1$.
\end{romanenum}
\end{prop}
\begin{proof}
We first consider the case where $A$ is an abelian variety defined over a number field $K$.  After replacing $A/K$ by its base extension to $K_A^\conn$, we may assume without loss of generality that $G_{A,\ell}$ is connected for all $\ell$. From Faltings, cf.~\cite{MR861971}, we know that:
\begin{alphenum}
\item \label{I:Tate conjecture a}
The $\QQ_\ell[\Gal_K]$-module $V_\ell(A)$ is semisimple.
\item  \label{I:Tate conjecture b}
The natural map $\End(A) \otimes_\ZZ \QQ_\ell \hookrightarrow \End_{\QQ_\ell[\Gal_K]}(V_\ell(A))$ is an isomorphism.
\end{alphenum}
From (\ref{I:Tate conjecture a}), we deduce that $G_{A,\ell}$ is reductive.   From (\ref{I:Tate conjecture b}), the commutant $R_\ell$ of $G_{A,\ell}$ in $\End_{\QQ_\ell}(V_\ell(A))$ is isomorphic to $\End(A)\otimes_\ZZ \QQ_\ell$.    In particular, the dimension of $R_\ell$ over $\QQ_\ell$ is independent of $\ell$ and the center of $R_\ell$ is semisimple for all sufficiently large $\ell$.  Serre proved part (\ref{P:G specialization i}) in \cite{MR1730973}*{133}; this also follows from Theorem~1.2 of \cite{MR1441234} since the dimension of the groups $H_{v,\ell}$ that occur there do not depend on $\ell$.

It remains to consider the case where $n\geq 1$.  For a fixed $\ell$, Lemma~\ref{L:HIT frattini}(\ref{L:HIT frattini ii}) implies that $G_{A_u,\ell}=G_{A,\ell}$ for some $u\in U(K)$.   In particular, $G_{A_u,\ell}^\circ=G_{A,\ell}^\circ$.   The group $G_{A,\ell}^\circ=G_{A_u,\ell}^\circ$ is thus reductive from  part (\ref{P:G specialization i}) in the case of an abelian variety defined over a number field.  This proves (\ref{P:G specialization i}).

Take any two distinct primes $\ell$ and $\ell'$.    By Lemma~\ref{L:HIT frattini}(\ref{L:HIT frattini ii}), we have $G_{A_u,\ell}=G_{A,\ell}$ and $G_{A_u,\ell'}=G_{A,\ell'}$ for some $u\in U(K)$. In particular, $G_{A_u,\ell}^\circ=G_{A,\ell}^\circ$ and $G_{A_u,\ell'}^\circ=G_{A,\ell'}^\circ$.  By part (\ref{P:G specialization ii}) in the case of an abelian variety defined over a number field, the ranks of $G_{A_u,\ell}^\circ=G_{A,\ell}^\circ$ and $G_{A_u,\ell'}^\circ=G_{A,\ell'}^\circ$ are equal.  Since $\ell$ and $\ell'$ are arbitrary primes, we deduce that the rank of $G_{A,\ell}^\circ$ does not depend on $\ell$.  This proves (\ref{P:G specialization ii}).

 Note that specialization gives an isomorphism $V_\ell(A_u) = V_\ell(A)$ for which the actions of the groups $G_{A_u,\ell} \subseteq G_{A,\ell}$ are compatible.  Denote by $R_\ell$ the commutant of $G_{A,\ell}^\circ$ in $\End_{\QQ_\ell}(V_\ell(A))$.   For $u\in U(K)$, denote by $R_{u,\ell}$ the commutant of $G_{A_u,\ell}^\circ$ in $\End_{\QQ_\ell}(V_\ell(A))$. 
 
  Take any two distinct primes $\ell$ and $\ell'$.    As above, we have $G_{A_u,\ell}^\circ=G_{A,\ell}^\circ$ and $G_{A_u,\ell'}^\circ=G_{A,\ell'}^\circ$ for some $u\in U(K)$.   In particular, $R_{u,\ell}=R_\ell$ and $R_{u,\ell'}=R_{\ell'}$.  By part (\ref{P:G specialization iii}) in the case of an abelian variety defined over a number field, we deduce that $\dim_{\QQ_\ell} R_\ell$ is independent of $\ell$.  This proves (\ref{P:G specialization iii}).\\

It remains to prove that part (\ref{P:G specialization iv}) holds; suppose $n\geq 1$.    For each prime $\ell$, let $S_\ell$ be the set of $u\in U(K)$ for which $G_{A_u,\ell}^\circ=G_{A,\ell}^\circ$; it has density $1$ by Lemma~\ref{L:HIT frattini}(\ref{L:HIT frattini ii}).

Take any $u\in S_2$ and prime $\ell$.    We have an inclusion of groups $G_{A_u,\ell}^\circ \subseteq G_{A,\ell}^\circ$ and they are reductive by part (\ref{P:G specialization i}).   The inclusion implies that $R_{u,\ell}\subseteq R_\ell$.    We have $\dim_{\QQ_2} R_{u,2}=\dim_{\QQ_2} R_{2}$ since $u\in S_2$, and hence $\dim_{\QQ_\ell} R_{u,\ell}=\dim_{\QQ_\ell} R_{\ell}$ by part (\ref{P:G specialization iii}).  Therefore, $R_{u,\ell}=R_\ell$.  The groups $G_{A_u,2}^\circ$ and $G_{A,2}^\circ$ have the same rank since $u\in S_2$, and hence $G_{A_u,\ell}^\circ$ and $G_{A,\ell}^\circ$ have the same rank by part (\ref{P:G specialization ii}).

As noted above, we have $R_{u,\ell} \cong \End(A_u)\otimes_\ZZ \QQ_\ell$ and hence the $\QQ_\ell$-algebra $R_\ell=R_{u,\ell}$, and it center, are semisimple for all $\ell$ greater than some constant $b\geq 2$.

Applying Lemma~\ref{L:reductive inclusion} below, we deduce that $G_{A_u,\ell}^\circ= G_{A,\ell}^\circ$ for all primes $\ell\geq b$.    Therefore, $G_{A_u,\ell}^\circ= G_{A,\ell}^\circ$ holds for all primes $\ell$ with $u\in S:=\bigcap_{\ell\leq b} S_\ell$.   The set $S$ has density $1$ since it is a finite intersection of density $1$ sets by Lemma~\ref{L:HIT frattini}(\ref{L:HIT frattini ii}).  Part (\ref{P:G specialization iv}) now follows.
\end{proof}

\begin{lemma}  \cite{MR1944805}*{Lemma~7} \label{L:reductive inclusion}
Let $F$ be a perfect field whose characteristic is $0$ or at least $5$.   Let $G_1 \subseteq G_2$ be reductive groups defined over $F$ that have the same rank.   Suppose we have a faithful linear representation $G_2 \hookrightarrow \GL_V$, where $V$ is a finite dimension $F$-vector space, such that the centers of the commutants of $G_1$ and $G_2$ in $\End_F(V)$ are the same $F$-algebra $R$.  Suppose further that the commutative $F$-algebra $R$ is semisimple.   Then $G_1=G_2$. 
 \qed
\end{lemma}

\subsection{Proof of Proposition~\ref{P:monodromy independence}}\label{SS:monodromy independence} 

Let $S$ be the set of $u \in U(k)$ such that $G_{A_u,\ell}^\circ=G_{A,\ell}^\circ$ for all primes $\ell$ and such that the specialization of $\gamma_{\ell}$ at $u$ is surjective for all primes $\ell$.  The set $S$ has density $1$ by Proposition~\ref{P:G specialization}(\ref{P:G specialization iv}) and Lemma~\ref{L:connected independence families}(\ref{L:connected independence families ii}).  

Now take any $u\in S$ and any prime $\ell$.  Specialization gives an inclusion $G_{A_u,\ell} \subseteq G_{A,\ell}$ and we have $G_{A_u,\ell}^\circ=G_{A,\ell}^\circ$  since $u\in S$.    The group $G_{A,\ell}(\QQ_\ell)$ is Zariski dense in $G_{A,\ell}$ since $G_{A,\ell}$ is defined as the Zariski closure of a subgroup of $\GL_{V_\ell(A)}(\QQ_\ell)$.   So to prove that $G_{A_u,\ell} = G_{A,\ell}$, it suffices to show that the natural injective homomorphism $\varphi\colon G_{A_u,\ell}(\QQ_\ell)/G_{A_u,\ell}^\circ(\QQ_\ell) \hookrightarrow G_{A,\ell}(\QQ_\ell)/G_{A,\ell}^\circ(\QQ_\ell)$ is surjective.  If $\varphi$  was not surjective, then the specialization of $\gamma_\ell$ at $u$ would not be surjective which is impossible since $u\in S$.   Therefore, $G_{A_u,\ell} = G_{A,\ell}$.  The proposition follows since $S$ has density 1 and $\ell$ was arbitrary.

\section{Big $\ell$-adic images} \label{S:big ell-adic images}

Fix an abelian scheme $\pi\colon A\to U$ of relative dimension $g\geq 1$, where $U$ is a non-empty open subvariety of $\PP^n_K$ with $K$ a number field and $n\geq 0$.    As noted in \S\ref{S:monodromy in families}, this includes the case that $A$ is an abelian variety defined over $K$.

\subsection{More $\ell$-adic monodromy groups} \label{SS:more groups}

Similar to our definition of $G_{A,\ell}$, we now define an $\ell$-adic monodromy that is a group scheme over $\ZZ_\ell$.   With notation as in \S\ref{SS:notation},  we have an algebraic group scheme $\GL_{T_\ell(A)}$ defined over $\ZZ_\ell$.  Note that the generic fiber of $\GL_{T_\ell(A)}$ is $\GL_{V_\ell(A)}$.  

We define $\calG_{A,\ell}$ to be the Zariski closure of $\rho_{A,\ell}(\pi_1(U))$ in $\GL_{T_\ell(A)}$; it is a group scheme defined over $\ZZ_\ell$.  The group schemes $G_{A,\ell}$ and $\calG_{A,\ell}$ determine each other.  More precisely, $G_{A,\ell}$ is the generic fiber of $\calG_{A,\ell}$ and $\calG_{A,\ell}$ is the Zariski closure of $G_{A,\ell}$ in $\GL_{T_\ell(A)}$.   

Let $\calG_{A,\ell}^\circ$ be the $\ZZ_\ell$-group subscheme of $\calG_{A,\ell}$ that is the Zariski closure of $G_{A,\ell}^\circ$.   

Let $\calS_{A,\ell}$ be the $\ZZ_\ell$-group subscheme of $\calG^\circ_{A,\ell}$ that is the Zariski closure of the derived subgroup of $G_{A,\ell}^\circ$.

\subsection{An open image theorem}

The following theorem says that the image of $\rho_{A,\ell}$ is ``large'' for all sufficiently large $\ell$.  More precisely, $\rho_{A,\ell}(\pi_1(U))$ contains $\calS_{A,\ell}(\ZZ_\ell)'$  and its index in $\calG_{A,\ell}(\ZZ_\ell)$ is uniformly bounded for all sufficiently large $\ell$.  The theorem also describes several important properties of the $\ZZ_\ell$-group schemes $\calG^\circ_{A,\ell}$ and $\calS_{A,\ell}$.  

\begin{thm} \label{T:geometric gp theory}
There is a constant $b_A$, depending only on $A$, such that the following hold for all primes $\ell \geq b_A$:
\begin{romanenum}
\item  \label{T:geometric gp theory i}
The $\ZZ_\ell$-group scheme $\calG_{A,\ell}^\circ$ is reductive and $\calS_{A,\ell}$ is semisimple.
\item \label{T:geometric gp theory ii}
We have $\rho_{A,\ell}(\pi_1(U)) \supseteq \calS_{A,\ell}(\ZZ_\ell)'$ and $\bbar\rho_{A,\ell}(\pi_1(U)) \supseteq \calS_{A,\ell}(\FF_\ell)'$.
\item
\label{T:geometric gp theory iii}
We have $[\calG_{A,\ell}(\ZZ_\ell): \rho_{A,\ell}(\pi_1(U))]\ll_A 1$.   
\item \label{T:geometric gp theory iv}
The groups $\calS_{A,\ell}(\ZZ_\ell)'$ and $\calS_{A,\ell}(\FF_\ell)'$ are perfect and all of their finite simple quotients are of Lie type in characteristic $\ell$.  We have $\calG_{A,\ell}^\circ(\ZZ_\ell)'=\calS_{A,\ell}(\ZZ_\ell)'$.
\item \label{T:geometric gp theory v}
The cardinality of $\calS_{A,\ell}(\ZZ_\ell)/\calS_{A,\ell}(\ZZ_\ell)'$ is finite and can be bounded in terms of $g$.
\item \label{T:geometric gp theory vi}
Suppose that $H$ is a closed subgroup of $\calS_{A,\ell}(\ZZ_\ell)$, in the $\ell$-adic topology, whose image modulo $\ell$ contains $\calS_{A,\ell}(\FF_\ell)'$.  Then $H\supseteq \calS_{A,\ell}(\ZZ_\ell)'$. 
\end{romanenum}
\end{thm}

We will prove Theorem~\ref{T:geometric gp theory} in \S\ref{SS:main open image proof} by using the following lemma to reduce to the case of an abelian variety over a number field. 

\begin{lemma} \label{L:geometric gp theory reduction}
If Theorem~\ref{T:geometric gp theory} holds for abelian varieties over any number field $K$, then it holds in general.
\end{lemma}
\begin{proof}
We can assume that $n\geq 1$.   By Proposition~\ref{P:monodromy independence}, there is a $u\in U(K)$ such that $G_{A_u,\ell}=G_{A,\ell}$ for all $\ell$.  This implies that $\calG_{A_u,\ell}=\calG_{A,\ell}$ and $\calS_{A_u,\ell}=\calS_{A,\ell}$ hold for all $\ell$.  Specialization by $u$ gives inclusions $\rho_{A_u,\ell}(\Gal_K) \subseteq \rho_{A,\ell}(\pi_1(U)) \subseteq \calG_{A,\ell}(\ZZ_\ell)$.   It is now clear that Theorem~\ref{T:geometric gp theory} for the abelian variety $A_u/K$ implies that Theorem~\ref{T:geometric gp theory} holds for $A$ with $b_A:=b_{A_u}$.
\end{proof}

\subsection{An effective open image theorem} \label{SS:effective open image}

In this section, we assume that $A$ is an abelian variety of dimension $g\geq 1$ defined over a number field $K$.    We will state a version of Theorem~\ref{T:geometric gp theory} due to the author that gives a value of $b_A$ in terms of various invariants of $A$.  Before stating the results, we need to recall some quantities.\\

The algebraic group $G_{A,\ell}^\circ$ is reductive and its rank is independent of $\ell$, cf.~Proposition~\ref{P:G specialization}(\ref{P:G specialization ii}).  Denote the common rank of the groups $G_{A,\ell}^\circ$ by $r$.

Let $\p$ be any non-zero prime ideal of $\OO_K$ for which $A$ has good reduction.   Denote by $P_{A,\p}(x)$ the \defi{Frobenius polynomial} of $A$ at $\p$; it is a monic degree $2g$ polynomial with integer coefficients.   For a prime $\ell$ satisfying $\p\nmid \ell$, the representation $\rho_{A,\ell}$ is unramified at $\p$ and we have
\[
P_{A,\p}(x)= \det(xI - \rho_{A,\ell}(\Frob_\p)).
\]
Let  $\Phi_{A,\p}$ be the subgroup of $\CC^\times$ generated by the roots of $P_{A,\p}(x)$.  

We denote by $h(A)$ the (logarithmic absolute) Faltings height of $A$ obtained after base extending to any finite extension of $K$ over which $A$ has semistable reduction, see \S5 of \cite{MR861978}.   In particular, note that $h(A_L)=h(A)$ for any finite extension $L/K$.

\begin{thm} \label{T:main new}
Let $A$ be an abelian variety of dimension $g\geq 1$ defined over a number field $K$. Let $\q$ be a non-zero prime ideal of $\OO_K$ for which $A$ has good reduction and $\Phi_{A,\q}$ is a free abelian group of rank $r$.  Then there are positive constants $c$ and $\gamma$, depending only on $g$, such that Theorem~\ref{T:geometric gp theory} holds with
 \begin{align*} \label{E:effective bA}
 b_A= c \cdot (\max\{[K:\QQ],h(A), N(\q)\})^\gamma.
 \end{align*}
\end{thm}
\begin{proof}
Take any prime $\ell \geq b_A$.  Parts (\ref{T:geometric gp theory i}) and (\ref{T:geometric gp theory iii}) follow from parts (c) and (a), respectively, of Theorem~1.2 in \cite{Zywina-EffectiveOpenImage}.   Theorem~1.2(d) in \cite{Zywina-EffectiveOpenImage} implies that $\rho_{A,\ell}(\Gal_K) \supseteq \calS_{A,\ell}(\ZZ_\ell)'$.  Reducing modulo $\ell$ and using that $\calS_{A,\ell}$ is smooth, we find that $\bbar\rho_{A,\ell}(\Gal_K) \supseteq \calS_{A,\ell}(\FF_\ell)'$; this proves (\ref{T:geometric gp theory ii}).   Parts (\ref{T:geometric gp theory iv}) and (\ref{T:geometric gp theory v}) of Theorem~\ref{T:geometric gp theory} 
are shown to hold in the proof of Theorem~1.2 in \cite{Zywina-EffectiveOpenImage}, cf.~Proposition~4.25 of \cite{Zywina-EffectiveOpenImage}.   Part (\ref{T:geometric gp theory vi}) of Theorem~\ref{T:geometric gp theory} 
is also shown to hold in the proof of Theorem~1.2 in \cite{Zywina-EffectiveOpenImage}, cf.~Lemmas~4.23 and 4.24 of \cite{Zywina-EffectiveOpenImage}.
\end{proof}

\begin{remark} \label{R:q exists}
From Lemma~2.7 of \cite{Zywina-EffectiveOpenImage}, the set of non-zero prime ideals $\p$ of $\OO_K$ for which $A$ has good reduction and $\Phi_{A,\p}$ is a free abelian group of rank $r$ has density $1/[K_A^\conn:K]$.   In particular, there do exists prime ideals $\q$ as in the statement of Theorem~\ref{T:main new}.
\end{remark}

\subsection{Proof of Theorem~\ref{T:geometric gp theory}} \label{SS:main open image proof}
The theorem follows immediately from Lemma~\ref{L:geometric gp theory reduction} and Theorem~\ref{T:main new} (with Remark~\ref{R:q exists} to show that the assumptions are not vacuous).

\section{Geometric monodromy} \label{S:geometric monodromy}

 Fix an abelian scheme $\pi\colon A\to U$ of relative dimension $g\geq 1$, where $K$ is a number field and $U$ is a non-empty open subvariety of $\PP^n_K$ for some $n\geq 1$.     Fix notation as in \S\ref{S:monodromy in families} and \S\ref{S:big ell-adic images}.  Take a constant $b_A$ as in Theorem~\ref{T:geometric gp theory}.

In this section, we will prove the following constraints on the images of the representations $\rho_{A,\ell}$ when restricted to the geometric fundamental group $\pi_1(U_{\Kbar})$; there is no harm in suppressing base points below since $\pi_1(U_{\Kbar})$ is a normal subgroup of $\pi_1(U)$. 

\begin{prop} \label{P:geometric monodromy}
There is an open subgroup $H$ of $\pi_1(U_{\Kbar})$ such that the following hold:
\begin{alphenum}
\item \label{P:geometric monodromy a}
$\rho_{A,\ell}(H)$ lies in the group of $\QQ_\ell$-points of the derived subgroup of $G_{A,\ell}^\circ$ for all primes $\ell$,
\item \label{P:geometric monodromy b}
$\rho_{A,\ell}(H) \subseteq \calS_{A,\ell}(\ZZ_\ell)'$ for all primes $\ell \geq b_A$.
\end{alphenum}
\end{prop}

\subsection{The algebraic monodromy group}

Fix a field embedding $\Kbar \subseteq \CC$.    Let $\calA \to U_\CC$ be the fiber of $A\to U$ over $U_\CC$.   Associated to $\pi\colon \calA(\CC)\to U(\CC)$, we define the local system $\calF:=R^1 \pi_* \ZZ$ of $\ZZ$-modules on $U(\CC)$, where we are viewing $U(\CC)$ with its familiar analytic topology.    For each $u\in U(\CC)$, the fiber $\calF_u$ of $\calF$ at $u$ is the cohomology group $H^1(A_u(\CC),\ZZ)$.    Fix a point $u_0 \in U(\CC)$ and define $\Lambda:= H^1(A_{u_0}(\CC),\ZZ)$; it is a free abelian group of rank $2g$.   The local system $\calF$ gives rise to a monodromy representation
\[
\varrho\colon \pi_1^{\operatorname{top}}(U(\CC),u_0) \to \Aut_{\ZZ}(\Lambda) \cong \GL_{2g}(\ZZ),
\]
where we are now using the usual topological fundamental group.    Let $\calG$ be the $\ZZ$-group subscheme of $\GL_{2g,\ZZ}$ obtained by taking the Zariski closure of the image of $\varrho$; we call it the \defi{algebraic monodromy group} of the abelian scheme $A$.

\begin{lemma} \label{L:Deligne semisimplicity}
The neutral component $(\calG_\QQ)^\circ$ of the linear algebraic group $\calG_\QQ$ is semisimple.
\end{lemma}
\begin{proof}
It suffices to prove that the neutral component of $\calG_\CC$ is semisimple.  From Deligne \cite{MR0498551}*{Corollaire~4.2.9}, the neutral component of $\calG_\CC$ is semisimple; this uses that $U_\CC$ is smooth and connected and that $\pi\colon \calA \to U_\CC$ is smooth and proper.  
\end{proof}

\begin{lemma} \label{L:connected semisimple closure}
There is an open and normal subgroup $H$ of $\pi_1(U_\Kbar)$ such that the Zariski closure of $\rho_{A,\ell}(H)$ in $\GL_{V_\ell(A)}$ is connected and semisimple for all $\ell$.
\end{lemma}
\begin{proof}
The homomorphism $\pi_1(U_\CC) \to \pi_1(U_{\Kbar})$ induced by the embedding $\Kbar\subseteq \CC$ is an isomorphism, so it suffices to prove the lemma with $\Kbar$ replaced by $\CC$.  Recall that there is a natural isomorphism between the profinite completion of $\pi_1^{\operatorname{top}}(U(\CC),u_0)$ and $\pi_1(U_\CC)$ (only uniquely defined up to an inner automorphism since we are suppressing base points in our \'etale fundamental groups).

Take any integer $e\geq 1$.   Let $\bbar\rho_{\calA,\ell^e} \colon \pi_1(U_\CC)\to \GL_{2g}(\ZZ/\ell^e\ZZ)$ be the representation arising from the locally constant sheaf $\calA[\ell^e]$ of $\ZZ/\ell^e\ZZ$-modules on $U_\CC$.    By choosing compatible bases, these representation combine to give a single representation $\rho_{\calA,\ell} \colon \pi_1(U_\CC)\to \GL_{2g}(\ZZ_\ell)$.     Let $S$ be the Zariski closure of $\rho_{\calA,\ell}(\pi_1(U_\CC))$ in $\GL_{2g,\QQ_\ell}$.  Since $\calA$ is the fiber of $A$ above $U_\CC$, it suffices to prove the lemma with $\rho_{A,\ell}$ replaced by $\rho_{\calA,\ell}$.

Note that $\calA[\ell^e]$ gives a local system of $\ZZ/\ell^e\ZZ$-modules on $U(\CC)$ that is dual to $R^1\pi_*(\ZZ/\ell^e\ZZ)$, where $\pi\colon \calA(\CC)\to U(\CC)$; the $\ell^e$-torsion of $\calA_u(\CC)$ for a point $u\in U(\CC)$ is isomorphic to the homology group $H_1(\calA_u(\CC),\ZZ/\ell^e\ZZ)$ which is dual to the corresponding cohomology group.  Since $R^1\pi_*(\ZZ/\ell^e\ZZ)$ is isomorphic to $\calF/\ell^e\calF$, we find that the representation $\bbar\rho_{\calA,\ell^e}\colon \pi_1(U_\CC)\to \GL_{2g}(\ZZ/\ell^e\ZZ)$ is isomorphic to the dual of the representation obtained by taking the profinite completion of 
\[
\pi_1^{\operatorname{top}}(U(\CC),u_0) \xrightarrow{\varrho} \GL_{2g}(\ZZ)\to \GL_{2g}(\ZZ/\ell^e\ZZ).
\]  
Therefore, $\rho_{\calA,\ell}$ is isomorphic to the dual of the representation obtained from $\pi_1^{\operatorname{top}}(U(\CC),u_0) \xrightarrow{\varrho} \GL_{2g}(\ZZ)\subseteq \GL_{2g}(\ZZ_\ell)$ by taking the profinite completion and extending using continuity.   

 Let $H_0$ be the kernel of $\pi_1^{\operatorname{top}}(U(\CC),u_0) \xrightarrow{\varrho} \calG(\QQ)/\calG_\QQ^\circ(\QQ)$.  The Zariski closure of $\varrho(H_0)$ in $\GL_{2g,\QQ_\ell}$ is $\calG_{\QQ_\ell}^\circ$ which is connected and semisimple by Lemma~\ref{L:Deligne semisimplicity}.  The lemma will then hold by taking $H$ to be the profinite completion of $H_0$. 
\end{proof}

\subsection{Proof of Proposition~\ref{P:geometric monodromy}}

By Lemma~\ref{L:connected semisimple closure}, there is an open and normal subgroup $H$ of $\pi_1(U_\Kbar)$ such that the Zariski closure $S_\ell$ of $\rho_{A,\ell}(H)$ in $\GL_{V_\ell(A)}$ is connected and semisimple for all $\ell$.   We have $S_\ell \subseteq G_{A,\ell}$ and hence $S_\ell \subseteq G_{A,\ell}^\circ$ since $S_\ell$ is connected.   Since $G_{A,\ell}^\circ$ is reductive, we find that $S_\ell$ is contained in the derived subgroup of $G_{A,\ell}^\circ$.  This proves (\ref{P:geometric monodromy a}).

Now take any prime $\ell\geq b_A$.  By part (\ref{P:geometric monodromy a}) and $\rho_{A,\ell}(\pi_1(U))\subseteq \calG_{A,\ell}(\ZZ_\ell)$, we have $\rho_{A,\ell}(H)\subseteq \calS_{A,\ell}(\ZZ_\ell)$.   Define the homomorphism
\[
\beta_\ell \colon H \to \calS_{A,\ell}(\ZZ_\ell)\to B_\ell,
\]
where $B_\ell:=\calS_{A,\ell}(\ZZ_\ell)/\calS_{A,\ell}(\ZZ_\ell)'$.  By Theorem~\ref{T:geometric gp theory}(\ref{T:geometric gp theory v}), the group $B_\ell$ is abelian and its cardinality divides a positive integer $Q$ that depends only on $g$.  In particular, $Q$ is independent of $\ell$.   So to prove (\ref{P:geometric monodromy a}), it suffices to show that $H$ has only finitely many abelian quotients whose cardinality divides $Q$.

The group $H$ is the \'etale fundamental group of a connected variety $X$ of finite type over $\Kbar$ (let $X\to U_{\Kbar}$ be an \'etale cover corresponding the subgroup $H$ of $\pi_1(U_{\Kbar})$).   From \cite{MR0354656}, II.2.3.1,  we find that $H$ is finitely generated as a topological group.   In particular, the abelianization $H^\ab$ of $H$, i.e, the quotient of $H$ by its commutator subgroup, is a finitely generated topological group.    Viewing $H^\ab$ as an additive group, the quotient $H^\ab/Q H^\ab$ is a finitely generated abelian group whose exponent divides $Q$.  Therefore, $H^\ab/Q H^\ab$ has finite order and hence $H$ has only finitely  abelian quotients whose cardinality divides $Q$.

\section{Main reduction}
\label{S:main reduction}

Let $K$ be a number field.  Fix an abelian scheme $\pi\colon A\to U$ of relative dimension $g\geq 1$, where $U$ is a non-empty open subvariety of $\PP^n_K$ for some $n\geq 1$.  Fix notation as in \S\ref{S:monodromy in families} and \S\ref{S:big ell-adic images}.    Take a constant $b_A$ as in Theorem~\ref{T:geometric gp theory}.  

Define the set 
\[
B:=\{ u\in U(K) : \bbar\rho_{A_u,\ell}(\Gal_K)\not\supseteq \calS_{A,\ell}(\FF_\ell)' \text{ for some prime }\ell\geq b_A\}.
\]
The goal of this section is to prove the following proposition which reduces the proof of Theorem~\ref{T:MAIN} to showing that the set $B\subseteq \PP^n(K)$ has density $0$.  

\begin{prop}\label{P:main reduction}
Suppose that $B$ has density $0$.  Then there is a constant $C$ such that $[\rho_A(\pi_1(U)):\rho_{A_u}(\Gal_K)] \leq C$ holds for all $u\in U(K)$ away from a set of density $0$.
\end{prop}

\begin{remark}
Note that $\calS_{A,\ell}(\FF_\ell)'$ is a normal subgroup of $\calG_{A,\ell}(\FF_\ell)$ while $\bbar\rho_{A_u,\ell}(\Gal_K)$ is a subgroup of $\calG_{A,\ell}(\FF_\ell)$ uniquely defined up to conjugation.   It thus makes sense to ask whether the inclusion $\bbar\rho_{A_u,\ell}(\Gal_K)\supseteq \calS_{A,\ell}(\FF_\ell)'$ holds or not. 
\end{remark}

\subsection{Proof of Proposition~\ref{P:main reduction}}

Let $W\subseteq \pi_1(U)$ be the kernel of $\gamma_{A,\ell}$ from \S\ref{SS:neutral component}; it is a normal and open subgroup of $\pi_1(U)$ and is independent of $\ell$ by Lemma~\ref{L:connected independence families}(\ref{L:connected independence families i}).   The group $\rho_A(W)'$ is thus normal in $\rho_A(\pi_1(U))$.   The homomorphism $\gamma_{A,\ell}$ is surjective so the integer $[G_{A,\ell}(\QQ_\ell): G_{A,\ell}^\circ(\QQ_\ell)]$ is independent of $\ell$.  

For a prime $\ell\geq b_A$, we have $\bbar\rho_{A,\ell}(\pi_1(U))\supseteq \calS_{A,\ell}(\FF_\ell)'$.   Hilbert's irreducibility theorem implies that $\bbar\rho_{A_u,\ell}(\Gal_K)\supseteq \calS_{A,\ell}(\FF_\ell)'$ for all $u\in U(K)$ away from a set of density $0$.  Thus there is no harm in replacing $b_A$ by a larger integer.   In particular, we may assume that $b_A>7$ and $b_A>[\pi_1(U):W]$.  

\begin{lemma} \label{L:conn u and W}
Take any $u\in U(K)$ satisfying $G_{A_u,\ell}=G_{A,\ell}$ for all $\ell$.    Then for any integer $m>1$, we have $\rho_{A_u,m}(\Gal_{K_{A_u}^\conn})=\rho_{A_u,m}(\Gal_K) \cap \rho_{A,m}(W)$.
\end{lemma}
\begin{proof}
Fix a prime $\ell|m$. The kernel of the homomorphism $\rho_{A,m}(\pi_1(U))\to G_{A,\ell}(\QQ_\ell)/G_{A,\ell}^\circ(\QQ_\ell)$ obtained by composing the $\ell$-adic projection with the obvious quotient map is equal to $\rho_{A,m}(W)$.  Similarly the kernel of the homomorphism $\rho_{A_u,m}(\Gal_K)\to G_{A_u,\ell}(\QQ_\ell)/G_{A_u,\ell}^\circ(\QQ_\ell)$ obtained by composing the $\ell$-adic projection with the obvious quotient map is equal to $\rho_{A_u,m}(\Gal_{K_{A_u}^\conn})$.  From $G_{A_u,\ell}=G_{A,\ell}$ and the inclusion $\rho_{A_u,m}(\Gal_K) \subseteq \rho_{A,m}(\pi_1(U))$, we deduce that $\rho_{A_u,m}(\Gal_{K_{A_u}^\conn})=\rho_{A_u,m}(\Gal_K) \cap \rho_{A,m}(W)$.
\end{proof}

\begin{lemma} \label{L:most ell-adic images are large}
The set of $u\in U(K)$ that satisfy 
\[
\rho_{A_u,\ell}(\Gal_{K_{A_u}^\conn})' = \calS_{A,\ell}(\ZZ_\ell)'
\]
for all primes $\ell\geq b_A$ has density $1$.  For all $\ell\geq b_A$, we have $\rho_{A,\ell}(W)' =\calS_{A,\ell}(\ZZ_\ell)'$.
\end{lemma}
\begin{proof}
Let $B_1$ be the set of $u\in U(K)$ satisfying the following conditions:
\begin{alphenum}
\item \label{I:conn comm a}
$G_{A_u,\ell}=G_{A,\ell}$ for all primes $\ell$,
\item \label{I:conn comm b}
$\bbar\rho_{A_u,\ell}(\Gal_K) \supseteq \calS_{A,\ell}(\FF_\ell)'$ for all primes $\ell\geq b_A$.
\end{alphenum}
The set of $u\in U(K)$ satisfying (\ref{I:conn comm a}) has density 1 by Proposition~\ref{P:monodromy independence}.  The set of $u\in U(K)$ that satisfy (\ref{I:comm b}) has density 1 since the set $B$ in the statement of Proposition~\ref{P:main reduction} has density 0 by assumption.   Therefore, $B_1$ has density $1$.

Take any $u\in B_1$ and set $L:=K_{A_u}^\conn$.  Take any prime $\ell\geq b_A$.   It suffices to prove that $\rho_{A_u,\ell}(\Gal_{L})'$ and $\rho_{A,\ell}(W)'$ both equal $\calS_{A,\ell}(\ZZ_\ell)'$.

We claim that $\bbar\rho_{A_u,\ell}(\Gal_{L})' \supseteq \calS_{A,\ell}(\FF_\ell)'$.   Since $L/K$ is a Galois extension and $\calS_{A,\ell}(\FF_\ell)'$ is a (normal) subgroup of $\bbar\rho_{A_u,\ell}(\Gal_K)$, the group $H:=\calS_{A,\ell}(\FF_\ell)' \cap \bbar\rho_{A_u,\ell}(\Gal_{L})$ is a normal subgroup of $\calS_{A,\ell}(\FF_\ell)' $ of index at most \[
[L: K] = [G_{A_u,\ell}(\QQ_\ell): G_{A_u,\ell}^\circ(\QQ_\ell)] = [G_{A,\ell}(\QQ_\ell): G_{A,\ell}^\circ(\QQ_\ell)]=[\pi_1(U):W].
\]
By our choice of $b_A$ and $\ell\geq b_A$, we have $[\calS_{A,\ell}(\FF_\ell)':H] \leq [\pi_1(U):W] < b_A\leq \ell$.  Now suppose that $H\neq \calS_{A,\ell}(\FF_\ell)'$.  So there is a simple group $S$  that is a quotient of $\calS_{A,\ell}(\FF_\ell)'$ and satisfies $|S|<\ell$.   Theorem~\ref{T:geometric gp theory}(\ref{T:geometric gp theory iv}) implies that $S$ is of Lie type type in characteristic $\ell$ and hence $\ell$ divides $|S|$.   This contradicts $|S|<\ell$, so we deduce that $H=\calS_{A,\ell}(\FF_\ell)'$. Therefore, $\bbar\rho_{A_u,\ell}(\Gal_{L}) \supseteq H= \calS_{A,\ell}(\FF_\ell)'$.   Since $\calS_{A,\ell}(\FF_\ell)'$ is perfect by Theorem~\ref{T:geometric gp theory}(\ref{T:geometric gp theory iv}), we have $\bbar\rho_{A_u,\ell}(\Gal_{L})' \supseteq \calS_{A,\ell}(\FF_\ell)''=\calS_{A,\ell}(\FF_\ell)'$ which proves the claim.

Since $u\in B_1$, Lemma~\ref{L:conn u and W} implies that $\rho_{A_u,\ell}(\Gal_{L})$ is a subgroup of $ \rho_{A,\ell}(W) \subseteq \calG_{A,\ell}^\circ(\ZZ_\ell)$.   Therefore,
\[
\rho_{A_u,\ell}(\Gal_{L})' \subseteq \rho_{A,\ell}(W)' \subseteq \calG_{A,\ell}^\circ(\ZZ_\ell)' = \calS_{A,\ell}(\ZZ_\ell)',
\]   
where the last equality uses Theorem~\ref{T:geometric gp theory}(\ref{T:geometric gp theory iv}).  So to prove the lemma it thus suffices to show that $\rho_{A_u,\ell}(\Gal_{L})' \supseteq \calS_{A,\ell}(\ZZ_\ell)'$.  The image of $\rho_{A_u,\ell}(\Gal_{L})'$ in $\calS_{A,\ell}(\FF_\ell)$ is $\bbar\rho_{A_u,\ell}(\Gal_{L})'$ and hence contains $\calS_{A,\ell}(\FF_\ell)'$ by our claim.   Theorem~\ref{T:geometric gp theory}(\ref{T:geometric gp theory vi}) implies that $\rho_{A_u,\ell}(\Gal_{L})' \supseteq \calS_{A,\ell}(\ZZ_\ell)'$ as desired
\end{proof}

\begin{lemma}
\label{L:Jordan PL}
Take any prime $p$, any subgroup $G$ of $\GL_{2g}(\FF_p)$, and any composition factor $S$ of $G$.   There is an integer $J$, depending only on $g$, such that $S$ is abelian, $S$ is of Lie type in characteristic $p$, or $|S|<J$.
\end{lemma}
\begin{proof}
This is an immediate consequence of Theorem~0.2 of \cite{larsen-pink-finite_groups}. 
\end{proof}

Let $M$ be the product of all primes $\ell \leq \max\{b_A,J\}$, where $J$ is as in Lemma~\ref{L:Jordan PL}.  Define the group
\[
\calB:=\rho_{A,M}(W)' \times \prod_{\ell \nmid M}\calS_{A,\ell}(\ZZ_\ell)'. 
\]
After the following lemma, we will prove that $\rho_A(W)'$ is equal to $\calB$.

\begin{lemma}  \label{L:applied Goursat}
Let $H$ be a closed subgroup of $\calB$.   Suppose that the projection maps $H\to \rho_{A,M}(W)'$ and $H\to\calS_{A,\ell}(\ZZ_\ell)'$, with $\ell\nmid M$, are surjective.   Then $H=\calB$.
\end{lemma}
\begin{proof}
    After choosing bases of the $\ZZ_\ell$-modules $T_\ell(A)$, one can identify $\calB$ with a closed subgroup of $\GL_{2g}(\Zhat)$.   For each integer $e\geq 2$, let $\calB_e$ be the kernel of the reduction modulo $e$ homomorphism $\calB \subseteq \GL_{2g}(\Zhat)\to \GL_{2g}(\ZZ/e\ZZ)$.   We have  $\calB=\varprojlim_e \calB/\calB_e$, where $e$ is ordered by divisibility.   Since $H$ is a closed subgroup of $\calB$, it suffices to prove that $H \to \calB/\calB_e$ is surjective for all $e\geq 2$.   
    
    Suppose that $H\to \calB/\calB_e$ is not surjective for some $e\geq 2$.  We have an isomorphism $\calB/\calB_e\cong Q_M \times \prod_{\ell\nmid M} Q_\ell$ of groups for which all the following hold:
\begin{itemize}
\item
$Q_M$ and $Q_\ell$ (with $\ell\nmid M$) are finite quotients of $\rho_{A,M}(W)'$ and $\calS_{A,\ell}(\ZZ_\ell)'$, respectively, 
\item
$Q_\ell=1$ for all but finitely many primes $\ell\nmid M$,
\item 
the projection maps $H\to \calB/\calB_e \to Q_M$ and $H\to \calB/\calB_e \to Q_\ell$ (with $\ell\nmid M$) are surjective.
\end{itemize}
The last condition uses our assumptions on $H$ in the statement of the lemma.  \\

Since $H\to \calB/\calB_e$ is not surjective, there is an integer $b>1$ relatively prime to $M$ and a proper subgroup $H_0$ of $Q_M \times \prod_{\ell|b} Q_\ell$ such that the projection map $H_0\to Q_m$ is surjective for all $m\in \{M\}\cup\{\ell: \ell |b\}$.  So there are non-empty and disjoint subsets $I_1$ and $I_2$ of $\{M\}\cup\{\ell: \ell |b\}$ and a proper subgroup $H_1$ of $\prod_{m\in I_1} Q_m \times \prod_{m\in I_2} Q_m$ such that the projection map $H_1 \to \prod_{m\in I_i} Q_m$ is surjective for each $i\in\{1,2\}$.   

Set $G_i:= \prod_{m\in I_i} Q_m$ for $1\leq i \leq 2$.  Let $N$ and $N'$ be the kernel of the projection maps $H_1 \to G_2$ and $H_1 \to G_1$, respectively.  We have $H_1 \subseteq G_1\times G_2$ and hence $N \subseteq G_1\times \{1\}$ and $N' \subseteq \{1\} \times G_2$.  We can thus identify $N$ and $N'$ with normal subgroups of $G_1$ and $G_2$, respectively.  By Goursat's lemma (see \cite{MR0457455}*{Lemma~5.2.1}), the image of $H_1$ in $G_1/N \times G_2/N'$ is the graph of an isomorphism $G_1/N \cong G_2/N'$.

Suppose that $G_1/N$ and $G_2/N'$ are both trivial.   We have $N=G_1$ and $N'=G_2$ and hence $H_1=G_1 \times G_2$.   However, this contradicts that $H_1$ is a proper subgroup of $G_1 \times G_2$.  Therefore, there is a simple group $S$ that is isomorphic to a quotient of both $G_1$ and $G_2$.   In particular, there are distinct $m_1,m_2 \in  \{M\}\cup\{\ell: \ell |b\}$ such that $S$ is isomorphic to a quotient of both $Q_{m_1}$ and $Q_{m_2}$.   We may assume that $m_2=\ell$ is a prime not dividing $M$.

Since $\ell\nmid M$, every simple quotient of $Q_{m_2}=Q_\ell$ is non-abelian and of Lie type in characteristic $\ell$ by Theorem~\ref{T:geometric gp theory}(\ref{T:geometric gp theory iv}).   Therefore, $S$ is a non-abelian simple group of Lie type in characteristic $\ell$.

Suppose that $m_1$ is a prime $p$ that does not $M$.  By the same argument, $S$ is also a non-abelian simple group of Lie type in characteristic $p$.  However, there are no simple groups that are of Lie type in both characteristic $\ell$ and $p$, where $\ell$ and $p$ are distinct primes with $\ell>7$ (we have $\ell>7$ since $\ell\nmid M$).    This contradiction implies that $m_1=M$.  In particular, $S$ is a non-abelian group that is a quotient of $Q_M$ and hence also of $\rho_{A,M}(W)'$.

Define $m:= \prod_{\ell|M} \ell$.   The kernel of the quotient homomorphism 
\[
\rho_{A,M}(W)' \subseteq \Aut_{\ZZ_m}(T_m(A)) \to \Aut_{\ZZ/m\ZZ}(T_m(A)/mT_m(A)) \cong {\prod}_{p | M } \GL_{2g}(\FF_p)
\] 
is a product of pro-$p$ groups with $p | M$; in particular, it is prosolvable.  So there is a prime $p|M$ and a subgroup $G$ of $\GL_{2g}(\FF_p)$ such that $S$ is isomorphic to a composition factor of $G$.     Since $S$ is non-abelian, Lemma~\ref{L:Jordan PL} implies that $S$ is of Lie type in characteristic $p$ or satisfies $|S|<J$.    The group $S$ cannot be of Lie type in characteristic $p$ since again there are no simple groups that are of Lie type in both characteristic $\ell$ and $p$, where $\ell$ and $p$ are distinct primes with $\ell>7$.   Therefore, $|S|<J$.    Since $S$ is of Lie type in characteristic $\ell$, it has an element of order $\ell$ and hence $\ell\leq |S|<J$.   However, this contradicts that $\ell \nmid M$.  This contradiction proves that $H\to \calB/\calB_e$ is surjective for all $e\geq 2$ as desired.
\end{proof}

\begin{lemma} \label{L:commutator subgroup is B}
We have $\rho_{A}(W)'=\calB$.
\end{lemma}
\begin{proof}
Take any prime $\ell\geq b_A$.  By Theorem~\ref{T:geometric gp theory}(\ref{T:geometric gp theory ii}), we have $ \rho_{A,\ell}(\pi_1(U))\supseteq \calS_{A,\ell}(\ZZ_\ell)' $. We thus have inclusions
\begin{align}\label{E:inclusions pre-comm}
\calS_{A,\ell}(\ZZ_\ell)' \subseteq \rho_{A,\ell}(W) \subseteq \calG_{A,\ell}^\circ(\ZZ_\ell).
\end{align}
By Theorem~\ref{T:geometric gp theory}(\ref{T:geometric gp theory iv}), the commutators subgroups of both $\calS_{A,\ell}(\ZZ_\ell)'$ and $\calG_{A,\ell}^\circ(\ZZ_\ell)$ are equal to $\calS_{A,\ell}(\ZZ_\ell)'$.  Taking commutators of the groups in (\ref{E:inclusions pre-comm}), we deduce that $\rho_{A,\ell}(W)'=\calS_{A,\ell}(\ZZ_\ell)'$.

We can identify $\rho_A(W)'$ with a closed subgroup of $\rho_{A,M}(W)' \times \prod_{\ell\nmid M} \rho_{A,\ell}(W)' = \calB$.  Since the projections of $\rho_A(W)'$ to $\rho_{A,M}(W)'$ and $\rho_{A,\ell}(W)'=\calS_{A,\ell}(\ZZ_\ell)'$, with $\ell\nmid M$, are surjective, we deduce that $\rho_A(W)' =\calB$ by Lemma~\ref{L:applied Goursat}.
\end{proof}

\begin{lemma} \label{L:commutator density 1}
The set of $u\in U(K)$ for which $\rho_{A_u}(\Gal_K) \supseteq \rho_A(W)'$ holds has density $1$.
\end{lemma}
\begin{proof}
Take any $u\in U(K)$ that satisfies all the following conditions with $L:=K_{A_u}^\conn$:
\begin{alphenum}
\item \label{I:comm a}
$G_{A_u,\ell}=G_{A,\ell}$ for all $\ell$,
\item \label{I:comm b}
$\rho_{A_u,M}(\Gal_K)=\rho_{A,M}(\pi_1(U))$
\item \label{I:comm c}
$\rho_{A_u,\ell}(\Gal_L)' = \calS_{A,\ell}(\ZZ_\ell)'$ for all $\ell\geq b_A$
\end{alphenum}
By Lemma~\ref{L:conn u and W} with (\ref{I:comm a}) and (\ref{I:comm b}), we have $\rho_{A_u,M}(\Gal_{L})=\rho_{A_u,M}(\Gal_K) \cap \rho_{A,M}(W)=\rho_{A,M}(W)$.  In particular, $\rho_{A_u,M}(\Gal_{L})'=\rho_{A,M}(W)'$.   From this and (\ref{I:comm c}), we deduce that $\rho_{A_u}(\Gal_{L})'$ is a closed subgroup of $\calB=\rho_{A,M}(W)' \times \prod_{\ell \nmid M}\calS_{A,\ell}(\ZZ_\ell)'$ for which the projection maps to $\rho_{A,M}(W)'$ and $\calS_{A,\ell}(\ZZ_\ell)'$, with $\ell\nmid M$, are all surjective.  By Lemmas~\ref{L:applied Goursat} and \ref{L:commutator subgroup is B}, we have $\rho_{A_u}(\Gal_{L})'=\calB=\rho_A(W)'$.

To complete the lemma, it thus suffices to show that the set of $u\in U(K)$ that satisfy each of the conditions (\ref{I:comm a}), (\ref{I:comm b}) and (\ref{I:comm c}) has density 1.   The set of $u\in U(K)$ that satisfy (\ref{I:comm a}) has density 1 by Proposition~\ref{P:monodromy independence}.  The set of $u\in U(K)$ that satisfy (\ref{I:comm b}) has density 1 by Lemma~\ref{L:HIT frattini}(\ref{L:HIT frattini i}). The set of $u\in U(K)$ that satisfy (\ref{I:comm c}) has density 1 by Lemma~\ref{L:most ell-adic images are large}.
\end{proof}

Let 
\[
\beta_A \colon \pi_1(U) \to \rho_A(\pi_1(U))/ \rho_A(W)'
\]
be the surjective homomorphism obtained by composing $\rho_A$ with the obvious quotient map.

\begin{lemma} \label{L:finite geometric monodromy}
The group $\beta_A(\pi_1(U_{\Kbar}))$ is finite.
\end{lemma}
\begin{proof}
We need to prove that there is an open subgroup $H$ of $\pi_1(U_{\Kbar})$ such that $\rho_A(H)$ is contained in $\rho_A(W)'=\calB=\rho_{A,M}(W)' \times \prod_{\ell \nmid M}\calS_{A,\ell}(\ZZ_\ell)'$.     By Proposition~\ref{P:geometric monodromy}, there is an open subgroup $H$ of $\pi_1(U_{\Kbar})$ such that $\rho_{A,\ell}(H)\subseteq \calS_{A,\ell}(\ZZ_\ell)'$ for all $\ell\geq b_A$.    It thus suffices to show that there is an open subgroup $H$ of $\pi_1(U_{\Kbar})$ such that $\rho_{A,M}(H) \subseteq \rho_{A,M}(W)'$.

The group $\rho_{A,M}(W)$ is open in $\prod_{\ell | M} G_{A,\ell}^\circ(\QQ_\ell)$ since $\rho_{A,\ell}(W)$ is open in $G_{A,\ell}^\circ(\QQ_\ell)$ for all $\ell|M$ (note that each $\rho_{A,\ell}(W)$ has an open pro-$\ell$ subgroup).    For each $\ell$, let $S_{A,\ell}$ be the derived subgroup of $G_{A,\ell}^\circ$.  Note that for every open subgroup $H$ of $G_{A,\ell}^\circ(\QQ_\ell)$, the commutator subgroup $H'$ is open in $S_{A,\ell}(\QQ_\ell)$, cf.~\cite{HuiLarsen2015}*{Proposition~3.2}.    Therefore, $\rho_{A,M}(W)'$ is an open subgroup of $\prod_{\ell|M} S_{A,\ell}(\QQ_\ell)$.

By Proposition~\ref{P:geometric monodromy}, there is an open subgroup $H$ of $\pi_1(U_{\Kbar})$ such that $\rho_{A,M}(H) \subseteq \prod_{\ell|M} S_{A,\ell}(\QQ_\ell)$.  Since $\rho_{A,M}(H)$ is compact in  $\prod_{\ell|M} S_{A,\ell}(\QQ_\ell)$ and $\rho_{A,M}(W)'$ is open in  $\prod_{\ell|M} S_{A,\ell}(\QQ_\ell)$, we find that $\rho_{A,M}(H) \cap \rho_{A,M}(W)'$ is a finite index subgroup of $\rho_{A,M}(H)$.   So after replacing $H$ by a suitable open subgroup, we will have $\rho_{A,M}(H)\subseteq \rho_{A,M}(W)'$.
\end{proof}

Let $C$ be the cardinality of $\beta_A(\pi_1(U_{\Kbar}))$; it is finite by Lemma~\ref{L:finite geometric monodromy}.  Take any $u\in U(K)$ for which $\rho_{A_u}(\Gal_K) \supseteq \rho_A(W)'$.  We will now show that $[\rho_A(\pi_1(U)):\rho_{A_u}(\Gal_K)]\leq C$.  Since the set of $u\in U(K)$ satisfying $\rho_{A_u}(\Gal_K) \supseteq \rho_A(W)'$ has density $1$ by Lemma~\ref{L:commutator density 1}, this will complete the proof of the proposition.

Let $\beta_{A,u}\colon \Gal_K \to \rho_A(\pi_1(U))/\rho_A(W)'$ be the homomorphism obtained by specializing $\beta_A$ at $u$.   We have
\[
[\rho_A(\pi_1(U))/\rho_A(W)':\beta_{A,u}(\Gal_K)] \leq C
\]
since the homomorphism 
\[
\pi_1(U)\xrightarrow{\rho_A} (\rho_A(\pi_1(U))/\rho_A(W)')/\beta_A(\pi_1(U_{\Kbar}))
\]
 is surjective and factors through $\Gal_K$ (in particular, its specialization at a point $u\in U(K)$ is independent of the choice $u$).  Since $\rho_{A_u}(\Gal_K)\supseteq \rho_A(W)'$, we have
\begin{align*}
[\rho_A(\pi_1(U)):\rho_{A_u}(\Gal_K)] &= [\rho_A(\pi_1(U))/\rho_A(W)':\rho_{A_u}(\Gal_K)/\rho_A(W)']\\
&=[\rho_A(\pi_1(U))/\rho_A(W)':\beta_{A,u}(\Gal_K)] \\
&\leq C.
\end{align*}
This completes the proof of Proposition~\ref{P:main reduction}.

\begin{remark} \label{R:specific constant}
The above constant $C$ is precisely the one described in \S\ref{SS:the constant C}.  With notation as in \S\ref{SS:the constant C}, the group $M$ is equal to $\rho_A(W)$.  In \S\ref{S:main proof}, we will show that the set $B$ has density $0$ and this will imply, by Proposition~\ref{P:main reduction}, that Theorem~\ref{T:MAIN} holds with this particular constant $C$.
\end{remark}

\section{Explicit Hilbert irreducibility} \label{S:explicit HIT}

Let $U$ be a nonempty open subvariety of $\PP^n_K$ for some integer $n\geq 1$, where $K$ is a number field.  Let 
\[
\rho\colon \pi_1(U)\to G
\]
be a continuous and surjective homomorphism, where $G$ is a finite group.  For each point $u\in U(K)$, we obtain a homomorphism $\rho_u\colon \Gal_K\to G$ by specializing $\rho$ at $u$; it is uniquely defined up to conjugation by an element of $G$. 

 Let $G_g$ be the image of $\pi_1(U_{\Kbar})$ under $\rho$; it is a well-defined normal subgroup of $G$.  Let $L$ be the minimal extension of $K$ in $\Kbar$ for which $G_g$ is the image of $\pi_1(U_L)$.   We have a natural short exact sequence
\[
1\to G_g\to G \xrightarrow{\varphi} \Gal(L/K) \to 1.
\]

Let $\calU$ be the open subscheme of $\PP^n_{\OO_K}$ that is the complement of the Zariski closure of $\PP^n_K-U$ in $\PP^n_{\OO_K}$.   The $\OO_K$-scheme $\calU$ has generic fiber $U$.   Fix a finite set $\scrS$ of non-zero prime ideals of $\OO_K$ such that $\rho$ arises from a continuous homomorphism $\varrho\colon \pi_1(\calU_\OO)\to G$, where $\OO$ is the ring of $\scrS$-integers in $K$.   

\begin{thm} \label{T:HIT} 
With notation as above, fix a Galois extension $F\subseteq L$ of $K$ and a set $C\subseteq G$ that is stable under conjugation by $G$.  Define 
\[
\delta := \max_\kappa  \frac{|C\cap \kappa|}{|G_g|},
\]
where $\kappa$ varies over the $G_g$-cosets of $\varphi^{-1}(\Gal(L/F))$.  Assume that $\delta<1$.
Then for $x\geq 2$,
 \[
 |\{ u \in U(K) : H(u)\leq x,\, \rho_u(\Gal_K)\subseteq C\}| \ll_{U, F,\delta} \,\, x^{(n+1/2)[K:\QQ]} \log x + |\scrS|^{4n+4} + |C|^{2n+2}\cdot |G_g|^{4n+4}.
 \]  
\end{thm}

 Recall that Hilbert's irreducibility theorem implies that $\rho_u(\Gal_K)=G$ for all $u\in U(K)$ away from a set of density $0$.
Corollary~\ref{C:HIT} shows how Theorem~\ref{T:HIT} can be viewed as an explicit version of {Hilbert's irreducibility theorem}. 

\begin{cor} \label{C:HIT}
For $x\geq 2$, we have
\[
 |\{ u \in U(K) : H(u)\leq x,\, \rho_u(\Gal_K)\neq G\}| \ll_{U,L,|G|} \, x^{[K:\QQ](n+1/2)}\log x + |\scrS|^{4n+4}.
 \]
\end{cor}
\begin{proof}
For each $u\in U(K)$,  the quotient map $\rho_u(\Gal_K) \to G/G_g$ is surjective.  So if $\rho_u(\Gal_K) \neq G$, then $\rho_u(\Gal_K)$ is contained in a maximal subgroup $M$ of $G$ for which $M\to G/G_g$ is surjective.  It thus suffices to bound 
$ |\{ u \in U(K) : H(u)\leq x,\, \rho_u(\Gal_K)\subseteq {\bigcup}_{g\in G}\,g M g^{-1}\}|$ for any such  $M$; the corollary will then follow by summing over all $M$ (the number of such maximal subgroups can be bounded in terms of $|G|$).

Take any maximal subgroup $M$ of $G$ for which the quotient map $M\to G/G_g$ is surjective.  Define $C:=\cup_{g\in G} \, gMg^{-1} = \cup_{g\in G_g} gMg^{-1}$, where the last equality uses our assumption that $M\to G/G_g$ is surjective.    We have $C\cap G_g = \cup_{g\in G_g} g(M\cap G_g)g^{-1} $ since $G_g$ is normal in $G$.   The group $M\cap G_g$ is a proper subgroup of $G_g$ since $M\neq G$ and $M\to G/G_g$ is surjective.  Jordan's lemma (\cite{MR1997347}*{Theorem~4'}) implies that $C\cap G_g \neq G_g$.  In particular, $\delta:=|C\cap G_g|/|G_g|\leq 1-1/|G_g|<1$.     Applying Theorem~\ref{T:HIT} with $F=L$, we have
\[
 |\{ u \in U(K) : H(u)\leq x,\, \rho_u(\Gal_K)\subseteq C\}| \ll_{U,L,|G|} \, x^{[K:\QQ](n+1/2)}\log x + |\scrS|^{4n+4}
 \]
as desired.
\end{proof}

\subsection{Equidistribution over finite fields}

Fix a finite field $\FF_q$ of cardinality $q$ and denote its characteristic by $p$.  
In this section, we denote by $U$ a smooth affine variety over $\FF_q$ that is geometrically irreducible and has dimension $d\geq 1$.     Take positive integers $N$, $r$ and $\delta$ such that $U_{\FFbar_q}$ is isomorphic to a closed subscheme of $\AA^N_{\FFbar_q}$ defined by the vanishing of $r$ polynomials of degree at most $\delta$.

Consider a surjective continuous homomorphism
\[
\varrho\colon  \pi_1(U)\to G,
\]
where $G$ is a finite group.  Define $G_g:=\varrho(\pi_1(U_{\FFbar_q}))$; it is a normal subgroup of $G$.    We have a natural short exact sequence
\[
1\to G_g\to G \xrightarrow{\varphi} \Gal(\FF_{q^e}/\FF_q) \to 1
\]
for some $e\geq 1$.    Let $\kappa$ be the $G_g$-coset of $G$ that is the inverse image of the $q$-th power Frobenius automorphism $\Frob_q \in \Gal(\FF_{q^e}/\FF_q)$ under $\varphi$.    If $U(\FF_q)\neq \emptyset$, we can also characterize $\kappa$ as the unique $G_g$-coset of $G$ that contains $\rho(\Frob_u)$ for all $u\in U(\FF_q)$.

\begin{thm} \label{T:tame equidistribution}
Fix notation as above and assume that $p\nmid |G_g|$.   Let $C \subseteq \kappa$ be a set stable under conjugation by $G$.  Then
\[
|\{u \in U(\FF_q): \varrho(\Frob_u) \in C \}| = \frac{|C|}{|G_g|}\, |U(\FF_q)| + O_{N,r,\delta}\big(|C|^{1/2} |G_g| q^{d-1/2}\big).
\]
\end{thm}
\begin{proof}(Sketch)
This is essentially Theorem~1.1 of \cite{MR2240230} due to Kowalski;  the key difference is that we are more explicit with the dependencies in the implicit constant (we also have $|G_g|$ in our error term instead of $|G|$).    We now sketch the minor changes that need to be made in Kowalski's proof.

Choose any prime $\ell\neq p$.  Let $V\to \calU_{\FFbar_q}$ be the finite  \'etale Galois covering with group $G_g$ corresponding to the surjective homomorphism $\pi_1(\calU_{\FFbar_q}) \xrightarrow{\rho} G_g$.  By Propositions~4.5 and 4.4 of \cite{MR2289204} and using $p\nmid |G_g|$, we have
\[
\sigma_c(V,\QQ_\ell):= \sum_{i} \dim H_c^i(V, \QQ_\ell) \leq c(N,r,\delta) \cdot |G_g|,
\]
where $c(N,r,\delta)$ is a constant depending only on $N$, $r$ and $\delta$.  Moreover, \cite{MR2289204} gives an explicit value of $c(N,r,d)$.   Proposition 4.7 of \cite{MR2289204} and its proof imply that 
\begin{align} \label{E:sigma c bound}
\sigma_c(U_{\FFbar_q},\pi(\rho)):= \sum_{i} H_c^i(U_{\FFbar_q}, \pi(\rho)) \leq c(N,r,\delta) \cdot |G_g| \cdot \dim \pi
\end{align}
for any representation $\pi \colon G_g \to \GL_{\dim \pi}(\Qbar_\ell)$, where $\pi(\rho)$ denotes the lisse $\Qbar_\ell$-sheaf corresponding to $\pi\circ \rho\colon\pi_1(U_{\FFbar_q})\to \GL_{\dim \pi}(\Qbar_\ell)$.  Examining the proof of Theorem~1.1 of \cite{MR2240230} with the bound (\ref{E:sigma c bound}), we have
\begin{align*}
\bigg| |\{u \in U(\FF_q) : \rho(\Frob_u) \in C \}| - \frac{|C|}{|G_g|} \,|U(\FF_q)|  \bigg| \leq c(N,r,\delta) \cdot  |C|^{1/2} |G_g|^{3/2} q^{d-1/2}
\end{align*}
which gives the theorem.
\end{proof}

\subsection{Sieving}

Fix a subset $B\subseteq \PP^n(K)$ with $n\geq 1$.  Let $\Sigma$ be a set of non-zero primes ideals $\p$ of $\OO_K$ with positive density.  Let $\scrS$ be a finite subset of $\Sigma$.  Suppose that there are real numbers $0\leq \delta<1$ and $c\geq 1$ such that the image of the reduction modulo $\p$ map $B\to \PP^n(\FF_\p)$ has cardinality at most $\delta N(\p)^n + c N(\p)^{n-1/2}$ for all $\p\in \Sigma-\scrS$.  

The follow proposition uses the large sieve to bound $|\{u\in B : H(u)\leq x\}|$; we will use it later to prove Theorem~\ref{T:HIT}.

\begin{prop} \label{P:large sieve}
Fix notation and assumptions as above.  For $x\geq 2$, we have
\[
|\{u\in B : H(u)\leq x\}| \ll_{K,n,\Sigma} (1-\delta)^{-1} \cdot x^{(n+1/2)[K:\QQ]} \log x + |\scrS|^{4n+4} + ((1-\delta)^{-1} c)^{4n+4}.
\]
\end{prop}
\begin{proof}
For each $a=(a_1,\ldots,a_{n+1})\in \OO_K^{n+1}$, define \[
\norm{a}:=\max_{1\leq i \leq n+1} \max_\sigma |\sigma(a_i)|,
\]
where $\sigma$ runs over the field embeddings $K\hookrightarrow \CC$.   Note that $\norm{\cdot}$ extends uniquely to a norm on $\OO_K^{n+1}\otimes_\ZZ \RR$. 

Let $B'$ be the set of $a\in \OO_K^{n+1}-\{0\}$ for which the image of $a$ in $\PP^n(K)$ lies in $B$.  We first bound the number of $a\in B'$ for which $\norm{a}\leq x$.  For a non-zero prime ideal $\p$ of $\OO_K$, let $B'_\p$ be the the image of $B'$ under the reduction modulo $\p$ map $\OO_K^{n+1}\to \FF_\p^{n+1}$.    Define $\omega_\p:=1- |B'_\p|/N(\p)^{n+1}$.  We may assume $\omega_\p<1$ since otherwise $B=\emptyset$ and the proposition is trivial.

Now suppose $\p\in \Sigma-\scrS$.    Using our assumption on the image of $B$ modulo $\p$ and $c\geq 1$, we find that
\[
|B'_\p|\leq (\delta N(\p)^n + c N(\p)^{n-1/2}) \cdot |\FF_\p^\times| + 1 \leq \delta N(\p)^{n+1} + cN(\p)^{n+1/2}
\]
and hence $|B'_\p|/N(\p)^{n+1} \leq \delta +c/N(\p)^{1/2}$.      If $N(\p) \geq 4c^2/(1-\delta)^2$, then 
\[
{|B'_\p|}/{N(\p)^{n+1}} \leq \delta + c/(4c^2/(1-\delta)^2)^{1/2} = (1+\delta)/2
\]  
and hence $\omega_\p \geq (1-\delta)/2$. 

Since $\Sigma$ has positive density, there is a constant $b\geq 1$ depending only on $\Sigma$ such that 
\[
\#\{ \p \in \Sigma-\scrS : y/2 \leq N(\p) \leq y \} \gg_\Sigma \, y/\log y
\]
holds for all real $y$ satisfying $y\geq b$ and $y\geq |\scrS|^2$.\\

\noindent $\bullet$ First take any $x\geq 2$ satisfying 
\[
x^{[K:\QQ]/2} \geq \max\{b,|\scrS|^2,8c^2/(1-\delta)^2\}.
\] 
The large sieve (\cite{MR1757192}*{\S12.1}) implies that for any $Q\geq 1$, we have
\begin{align} \label{E:large sieve}
|\{a \in B' : \norm{a}\leq  x \}| \ll_{K,n} \frac{\max\{ x^{(n+1)[K:\QQ]}, Q^{2(n+1)}\}}{L(Q)},
\end{align}
where
\[
L(Q):=\sum_{\mathfrak{a}} \prod_{\p| \mathfrak{a}} \frac{\omega_\p}{1-\omega_\p}
\]
and the sum is over the square-free ideals $\mathfrak{a}$ of $\OO_K$ with norm at most $Q$.  We interpret (\ref{E:large sieve}) as being the trivial bound $+\infty$ when $L(Q)=0$.   

Set $Q:=x^{[K:\QQ]/2}$.   Take any prime $\p \in \Sigma-\scrS$ with $Q/2 \leq N(\p) \leq Q$.  We have $N(\p)\geq  \tfrac{1}{2} x^{[K:\QQ]/2} \geq 4c^2/(1-\delta)^2$ and hence $\omega_\p \geq (1-\delta)/2$.  Since $t/(1-t)$ is increasing on the interval $[0,1)$, we have $\omega_\p/(1-\omega_\p) \geq  ((1-\delta)/2)/(1-(1-\delta)/2)=(1-\delta)/(1+\delta)$.   Therefore,
\[
L(Q)\geq \sum_{\substack{\p\in \Sigma-\scrS \\ Q/2\leq N(\p) \leq Q}} \frac{1-\delta}{1+\delta} \geq \frac{1-\delta}{2} \cdot \#\{\p \in \Sigma-\scrS : Q/2\leq N(\p) \leq Q\}.
\]
We have $Q =x^{[K:\QQ]/2} \geq \max\{b,|\scrS|^2\}$ and hence $L(Q)\gg_\Sigma (1-\delta) Q/\log Q \gg_K (1-\delta)\, x^{[K:\QQ]/2}/\log x$.  From (\ref{E:large sieve}), we deduce that 
\[
|\{a \in B' : \norm{a}\leq  x \}| \ll_{K,n,\Sigma,}(1-\delta)^{-1} \cdot x^{(n+1/2)[K:\QQ]} \log x. \\
\]

\noindent $\bullet$  Now suppose that $x\geq 2$ satisfies $x^{[K:\QQ]/2} \leq \max\{b,|\scrS|^2,8c^2/(1-\delta)^2\}$.
Therefore,
\begin{align*}
|\{a \in B' : \norm{a}\leq  x \}| & \leq |\{a\in \OO_K^{n+1} : \norm{a} \leq x\}| \\
& \ll_{K,n} x^{(n+1)[K:\QQ]}\\
& \leq (\max\{b,|\scrS|^2,8c^2/(1-\delta)^2\})^{2n+2}.
\end{align*}

Combining both cases for $x\geq 2$ and using $b\ll_\Sigma 1$ gives
\begin{align} \label{E:integral version}
|\{a \in B' : \norm{a}\leq  x \}| \ll_{K,n,\Sigma} (1-\delta)^{-1} \cdot x^{(n+1/2)[K:\QQ]} \log x + |\scrS|^{4n+4} + ((1-\delta)^{-1} c)^{4n+4}.
\end{align}

By the proposition in \S13.4 of \cite{MR1757192}, there is a constant $c'$, depending only on $K$ and $n$, such that each $u \in \PP^n(K)$ is represented by a tuple $a\in \OO_K^{n+1}$ with $\norm{a} \leq c' H(u)$.  In particular, we have $|\{u\in B : H(u)\leq x\}| \leq |\{a\in B' : \norm{a}\leq c'x\}|$.  The proposition now follows directly from (\ref{E:integral version}) and $c'\ll_{K,n} 1$.
\end{proof}

\subsection{Proof of Theorem~\ref{T:HIT}}
We may assume that the set $C$ is non-empty since otherwise the bound in the theorem is trivial.      

\begin{lemma} \label{L:big mod p monodromy}
There is a finite set $\scrS_1$ of non-zero prime ideals of $\OO_K$, depending only on $U\subseteq\PP^n_K$, such that $\varrho(\pi_1(\calU_{\FFbar_\p}))=G_g$ for all non-zero prime ideals $\p \notin \scrS\cup \scrS_1$ of $\OO_K$ satisfying $\p\nmid |G_g|$.
\end{lemma}
\begin{proof}
Define the closed subvariety $Z:=\PP^n_K-U$ of $\PP^n_K$. Let $\calZ$ be the Zariski closure of $Z$ in $\PP^n_{\OO_K}$; its complement in $\PP^n_{\OO_K}$ is $\calU$.  For a commutative ring $R$, let $\Gr_{R}(1,n)$ be the Grassmannian of lines in $\PP^n_R$.   In \S4H1 of \cite{REU}, a closed subscheme $\calW$ of $\Gr_{\OO_K}(1,n)$ is constructed such that for each $\OO_K$-algebra $R$ and line $\calL \in (\Gr_{\OO_K}(1,n)-\calW)(R)$, the scheme theoretic intersection $\calL \cap \calZ_R$ is finite and \'etale over $\Spec R$.  We have $\calW\neq \Gr_{\OO_K}(1,n)$ by Bertini's theorem.    Let $\scrS_1$ be the (finite) set consisting of all non-zero prime ideals $\p$ of $\OO_K$ for which $\calW_{\FF_\p}\neq \Gr_{\OO_K}(1,n)_{\FF_\p}$.  Note that $\scrS_1$ depends only on $U\subseteq \PP^n_K$.

Now take any non-zero prime ideal $\p \notin \scrS\cup \scrS_1$ of $\OO_K$ satisfying $\p\nmid |G_g|$.   Let $\OO_\p^{\un}$ be the ring of integers in the maximal unramified extension of $K_\p^{\un}$ of $K_\p$ in a fixed algebraic closure $\Kbar_\p$.  The ring $\OO_\p^{\un}$ is a complete discrete valuation ring with residue field $\FFbar_\p$.    Take any line $L\in (\Gr_{\OO_K}(1,n)-\calW)(\FFbar_\p)$.  Since $\Gr_{\OO_K}(1,n)$ is smooth and $\calW$ is a closed subscheme, there is a line $\calL \in (\Gr_{\OO_K}(1,n)-\calW)(\OO_\p^\un)$ whose image in $(\Gr_{\OO_K}(1,n)-\calW)(\FFbar_\p)$ is $L$.      Define the $\OO_\p^\un$-scheme $\calV:=\calU_{\OO_\p^\un} \cap \calL$.   We have $\calV=\calL-\calD$, where $\calD:=\calL \cap \calZ_{\OO_\p^\un}$.   Observe that $\calD$ is finite \'etale over $\Spec \OO_\p^\un$ since $\calL \notin \calW(\OO_\p^\un)$.

We claim that the homomorphism 
\[
\pi_1(\calV_{\Kbar_\p}) =\pi_1(U_{\Kbar_\p} \cap \calL_{\Kbar_\p}) \to \pi_1(U_{\Kbar_\p}) \overset{\varrho}{\to} G
\]
has image $G_g$.  Fix an embedding $\Kbar_\p \subseteq \CC$. To prove the claim there is no harm in replacing $\Kbar_\p$ by the larger algebraically closed field $\CC$.  By Bertini's theorem, the homomorphism $\pi_1(\calU_\CC \cap L) \to \pi_1(U_{\CC}) \overset{\varrho}{\to} G$ has image $G_g$ for a generic line $L \in \Gr_\CC(1,n)(\CC)$.  The above claim follows by (topologically) deforming $\calL$  in $\Gr_\CC(1,n)(\CC)$ to a generic line;  note that small changes in $\calL$ do not change the image of the representation since $\calL$ intersects $\calZ(\CC)$ only at smooth points and transversally at each of these points.

 Choose a point $a_0 \in \calV(\FFbar_\p)$ with a lift $a_1 \in \calV(\OO_\p^{\un})$.  Since $\calD$ is finite \'etale over $\Spec \OO_\p^\un$, the Grothendieck specialization theorem implies that the natural homomorphisms
\[
 \pi_1(\calV_{K_\p^{\un}}, a_1) \to \pi_1(\calV_{\OO_\p^{\un}},a_1) \leftarrow \pi_1(\calV_{\FFbar_\p}, a_0)
\]
induce an isomorphism between the prime to $p=\operatorname{char} \FF_\p$ quotients of $\pi_1(\calV_{\Kbar_\p}, a_1) $ and $\pi_1(\calV_{\FFbar_\p}, a_0)$.  In the present setting, an accessible proof of Grothendieck's theorem can be found in \cite{MR1708609}*{\S4}. Therefore, the homomorphism
\begin{align}  \label{E:V Fbar}
\pi_1(\calV_{\FFbar_\p},a_0) \to \pi_1(\calV_{\OO_\p^{\un}},a_1) &\to \pi_1(\calU,a_1) \xrightarrow{\varrho} G 
\end{align}
has the same image as $\pi_1(\calV_{\Kbar_\p},a_1) \to \pi_1(\calV_{\OO_\p^{\un}},a_1) \to \pi_1(\calU,a_1) \xrightarrow{\varrho} G$ which is $G_g$ by our claim (we have $p\nmid |G_g|$ since $\p\nmid |G_g|$ by assumption).  

We have thus proved that the image of $\pi_1(\calU_{\FFbar_\p} \cap L) \to \pi_1(\calU_{\FFbar_\p})\xrightarrow{\varrho} G$ is $G_g$ for all lines $L\in (\Gr_{\OO_K}(1,n)-\calW)(\FFbar_\p)$.  Since $\p \notin \scrS_1$, $(\Gr_{\OO_K}(1,n)-\calW)_{\FFbar_\p}$ is a non-empty open subvariety of $\Gr_{\OO_K}(1,n)_{\FFbar_\p}$.  By Bertini's theorem, we deduce that $\pi_1(\calU_{\FFbar_\p})\xrightarrow{\varrho} G$ has image $G_g$.
\end{proof}

Let $\Sigma$ be the set of non-zero prime ideals $\p$ of $\OO_K$ that split completely in $F$; it has positive density by the Chebotarev density theorem.  

\begin{lemma} \label{L:mod p bound for B}
For any non-zero prime ideal $\p\in \Sigma-\scrS$ of $\OO_K$ satisfying $\p\nmid |G_g|$, we have 
\[
|\{u \in \PP^n(\FF_\p) : u\notin \calU(\FF_\p) \text{ or } \varrho(\Frob_u) \in C \}| \leq \delta N(\p)^n + O_U(|C|^{1/2}|G_g| N(\p)^{n-1/2}).
\]
\end{lemma}
\begin{proof}
We can view $\AA_{\OO_K}^n=\Spec \OO_K[x_1,\ldots, x_n]$ as an open subscheme of $\PP_{\OO_K}^n$ via the morphism $(a_1,\ldots, a_n)\mapsto [a_1,\ldots, a_n,1]$.   There is a non-zero polynomial $f\in \OO_K[x_1,\ldots, x_n]$ that is squarefree in $K[x_1,\ldots, x_n]$  such that $\calU':= \Spec (\OO_K[x_1,\ldots,x_n][f^{-1}])$ is an open $\OO_K$-subscheme of $\calU$.    There is a finite set $\scrS_2$ of non-zero prime ideals $\OO_K$ such that for all non-zero prime ideals $\p\notin \scrS_2$ of $\OO_K$:
\begin{itemize}
\item
$\calU'_{\FF_\p}$ is an open affine subvariety of $\PP^n_{\FF_\p}$ of dimension $n$ that is geometrically irreducible,
\item
$\calU'_{\FF_\p}$ is isomorphic to the closed subscheme of $\AA_{\FF_\p}^{n+1}=\Spec \FF_\p[x_1,\ldots, x_n,x_{n+1}]$ defined by the equation $\bbar{f}(x_1,\ldots, x_n) \cdot x_{n+1}=1$, where $\bbar{f}$ is obtained from $f$ by reducing its coefficients modulo $\p$.
\end{itemize}
Note that $f$ and $\scrS_2$ are choices that depends only on $U\subseteq \PP^n_K$. 

Now take any prime ideal $\p\in \Sigma-\scrS$ satisfying $\p\nmid |G_g|$. 
Let $\scrS_1$ be a set of prime ideals from Lemma~\ref{L:big mod p monodromy}.    Since $\scrS_1$ and $\scrS_2$ depend only on $U\subseteq \PP^n_K$, we may further assume that $\p\notin \scrS_1\cup \scrS_2$; the lemma holds for the finite number of excluded prime ideals by suitably increasing the implicit constant.  Similarly, we may also assume that $\calU'(\FF_\p)$ is non-empty.

From $\varrho$, we obtain a continuous homomorphism $\varrho_\p \colon \pi_1(\calU_{\FF_\p})\to G$.  We have $\varrho_\p(\calU_{\FFbar_\p})=G_g$ by Lemma~\ref{L:big mod p monodromy} and $\p\nmid |G_g|$.   Since $\calU'_{\FF_\p}$ is a non-empty open subvariety of $\calU_{\FF_\p}$, we can restrict $\varrho_\p$ to obtain a homomorphism $\pi_1(\calU'_{\FFbar_\p})\to G$ that satisfies $\varrho_\p(\pi_1(\calU'_{\FFbar_\p}))=\varrho_\p(\pi_1(\calU_{\FFbar_\p}))=G_g$.
 
There is a unique $G_g$-coset $\kappa$ of $G$ such that $\varrho_\p(\Frob_u) \in \kappa$ for all $u\in \calU'(\FF_\p)$.  By Theorem~\ref{T:tame equidistribution}, applied to the affine variety $\calU'_{\FF_\p}$ and the representation $\rho_\p$ (and using $\p\nmid|G_g|$), we have
\begin{align*}
|\{u \in \calU'(\FF_\p) : \varrho(\Frob_u) \in C\}|
&= |\{u \in \calU'(\FF_\p) : \varrho_\p(\Frob_u) \in C \cap\kappa\}|\\
&= \frac{|C\cap \kappa|}{|G_g|} \, |\calU'(\FF_\p)| + O_U(|C|^{1/2}|G_g| N(\p)^{n-1/2});
\end{align*}
note that the implicit term depends only on $U\subseteq \PP^n_K$ since $\calU_{\FFbar_\p}'$ is isomorphic to a closed subscheme of $\AA^{n+1}_{\FFbar_\p}$ defined by the polynomial equation $\bbar{f}(x_1,\ldots, x_n) \cdot x_{n+1}=1$, where $f$ is a choice depending only on $U$.
 
Take any $u\in \calU(\FF_\p)$.  Since $\p$ splits completely in $F$, we have $(\varphi\circ \varrho_\p)(\Frob_\p) \in \Gal(L/F)$.    Therefore, $\kappa\subseteq \varphi^{-1}(\Gal(L/F))$.    We thus have $|C\cap\kappa|/|G_g|\leq \delta$ by the definition of $\delta$.  Using this and $|\calU'(\FF_\p)| = N(\p)^n +O_U(N(\p)^{n-1/2})$, we find that 
\begin{align*}
&|\{u \in \PP^n(\FF_\p) : u\notin \calU(\FF_\p) \text{ or } \varrho(\Frob_u) \in C \}| \\
\leq&\, |\{u \in \calU'(\FF_\p) : \varrho(\Frob_u) \in C\}| + |\PP^n(\FF_\p)-\calU'(\FF_\p)| \\
\leq & \,\delta N(\p)^n + O_U(|C|^{1/2}|G_g| N(\p)^{n-1/2}),
\end{align*}
where since $C\neq \emptyset$ we can absorb the various error terms.  
\end{proof} 
 
Define the set 
\[
B:=\{ u \in U(K) : \rho_u(\Gal_K) \subseteq C\}.
\]   
For each $\p\in \Sigma-\scrS$, denote by $B_\p$ the image of $B$ under the reduction modulo $\p$ map $U(K)\subseteq \PP^n(K)\to \PP^n(\FF_\p)$.  Let $\scrS'$ be the finite set of primes $\p\in \Sigma$ that lie in $\scrS$  or divide $|G_g|$.   

Take any $\p\in \Sigma-\scrS'$ and $u\in B$.    Denote by $u_\p\in \PP^n(\FF_\p)$ the image of $u$ modulo $\p$.   If $u_\p \in \calU(\FF_\p)$, then $\rho_u(\Frob_\p)=\varrho(\Frob_{u_\p})$. Since $u\in B$, we deduce that $u_\p \notin \calU(\FF_\p)$ or $\varrho(\Frob_{u_\p}) \in C$.   By Lemma~\ref{L:mod p bound for B}, we deduce that
\[
|B_\p| \leq \delta N(\p)^n + O_U(|C|^{1/2}|G_g| N(\p)^{n-1/2})
\]
for all $\p\in \Sigma-\scrS'$.  Take any $x\geq 2$.  By Proposition~\ref{P:large sieve}, we have
\begin{align*}
|\{u\in B : H(u)\leq x\}| &\ll_{U,F} (1-\delta)^{-1} \cdot x^{(n+1/2)[K:\QQ]} \log x + |\scrS'|^{4n+4} + ((1-\delta)^{-1} |C|^{1/2}|G_g|)^{4n+4}\\
&\ll_{\delta} x^{(n+1/2)[K:\QQ]} \log x + |\scrS'|^{4n+4} + |C|^{2n+2}\cdot |G_g|^{4n+4}\\
&\ll_{n} x^{(n+1/2)[K:\QQ]} \log x + |\scrS|^{4n+4} + |C|^{2n+2}\cdot |G_g|^{4n+4},
\end{align*}
where the last inequality uses that $|\scrS'| = |\scrS|+ O( \log |G_g|)$ and $C\neq \emptyset$.

\section{Derangements} \label{S:derangements}

Let $G$ be a linear algebraic group defined over a finite field $\FF_\ell$ for which its neutral component $G^\circ$ is semisimple and adjoint.   Let $S$ be the commutator subgroup of $G^\circ(\FF_\ell)$.     Fix a group $H$ satisfying 
\[
S \subseteq H \subseteq G(\FF_\ell)
\] 
and fix a normal subgroup $H_g$ of $H$.  Define $H_0 := H \cap G^\circ(\FF_\ell)$; it is a normal subgroup of $H$ that contains $S$.  Let $r$ be the rank of $G^\circ$ and define $m=[G(\FF_\ell):G^\circ(\FF_\ell)]$.

\begin{prop} \label{P:coset gp theory}
With notation as above, let $M$ be a subgroup of $H$ for which $M\not\supseteq S$ and for which the natural homomorphisms $M\to H/H_g$ and $M\to H/H_0$ are surjective.   Define the subset 
\[
C := \bigcup_{h\in H} h M h^{-1}
\]    
of $H$.   Then there is a constant $0\leq \delta <1$, depending only on $r$ and $m$, satisfying
\begin{equation} \label{E:Msharp}
\frac{|C \cap \kappa| }{ |H_g| } \leq \delta + O_{r,m}(1/\ell)
\end{equation}
for every $H_g$-coset $\kappa$ in $H_0 \cdot H_g$.
\end{prop}

We will prove Proposition~\ref{P:coset gp theory} by repeated reducing to simpler cases.    Proposition~\ref{P:coset gp theory} will be used to prove a specialized version of Hilbert's irreducibility theorem (Theorem~\ref{T:HIT main}).  The notation $H_g$ is chosen because in our applications, $H$ will be the image of an arithmetic fundamental group and $H_g$ will be the image of its geometric subgroup. 

\subsection{A theorem of Fulman and Guralnick} \label{SS:derangements}

Consider a finite group $H$ acting on a set $\Omega$.  An element $h \in H$ is called a \defi{derangement} on $\Omega$ if it has no fixed points.  For a non-empty subset $B\subseteq H$, let $\delta(B,\Omega)$ be the proportion of elements in $B$ that are derangements on $\Omega$.

The following is a slight variant of a result of Fulman and Guralnick.

\begin{thm}[Fulman--Guralnick] \label{T:Fulman-Guralnick}
Let $G$ be a connected, geometrically simple and adjoint linear algebraic group of rank $r$ defined over a finite field $\FF_q$.   Let $S$ be the commutator subgroup of $G(\FF_q)$ and fix a group $S\subseteq H \subseteq G(\FF_q)$.  Fix a maximal subgroup $M$ of $H$ satisfying $M\not\supseteq S$ and define $\Omega=H/M$ with $H$ acting by left multiplication.  Then for every $S$-coset $\kappa$ in $H$, we have
\[
\delta(\kappa,\Omega) \geq \delta + O_r(1/q)
\]
with $0< \delta\leq1$ a constant that depends only on $r$.
\end{thm}
\begin{proof}
Since $G$ is geometrically simple and adjoint group, by taking $q$ sufficiently large in terms of $r$, we may assume that $S$ is a non-abelian simple group and that $H$ has socle $S$.   Conjugation on $S$ allows us to view $G(\FF_q)$, and hence also $H$,  as a subgroup of the automorphism group of $S$.   Our theorem is then a consequence of Corollary 7.4 in \cite{Fulman:2012} and the remark following it;  note that since $H\subseteq G(\FF_q)$, $H$ lies in the group of inner-diagonal automorphisms of $S$.
\end{proof}

\begin{remark} \label{R:derangement equivalence}
With notation as in Theorem~\ref{T:Fulman-Guralnick}, define the set $C := \bigcup_{h\in H} h M h^{-1}$.     Left multiplication gives a transitive action of $H$ on $\Omega=H/M$.  Note that an element $x\in H$ fixes a coset $hM \in H/M$ if and only if $x \in hMh^{-1}$.   So an element $x \in H$ is a derangement on $\Omega$ if and only if it does {not} lie in $C$.  In particular, for any $S$-coset $\kappa$ in $H$, we have $\delta(\kappa,\Omega) = 1 - {|C\cap \kappa|}/{|S|}$. 

Similarly, we could reformulate Proposition~\ref{P:coset gp theory}
 in terms of derangements.
\end{remark}

\subsection{Proof of Proposition~\ref{P:coset gp theory}}

 Since $G^\circ$ is a connected and adjoint, we have $G^\circ= \prod_{i=1}^n G_i$, where the $G_i$ are connected, adjoint and simple groups defined over $\FF_\ell$.  We have $S= S_1\times \cdots \times S_n$, where $S_i$ is the commutator subgroup of $G_i(\FF_\ell)$.    
 
 By excluding a finite number of primes $\ell$ that depend only on $r$ and $m$, we may assume that all the groups $S_i$ are non-abelian and simple and that $\ell  > m$.  Since $S_i$ contains an element of order $\ell$,  we have $|S_i| \geq\ell > m \geq [H:H_0]$.  Note that the proposition holds for the finite number of excluded primes by increasing the implicit constant.
 
\begin{lemma} \label{L:M S surjectivity}
The natural map $M\cap S\to S/(S\cap H_g)$ is surjective.
\end{lemma} 
\begin{proof}
We claim that $S_i$ is not isomorphic to a composition factor of $M/(M\cap S)$ for any $1\leq i \leq n$.   Take any $1\leq i \leq n$. To prove the claim it suffices to show that neither of the groups $M/(M\cap H_0)$ or $(M\cap H_0)/(M\cap S)$ have a composition factor isomorphic to $S_i$;  note that $H_0$ is a normal subgroup of $H$ that contains $S$.    The group $S_i$ is not a composition factor of $M/(M\cap H_0)$ since $[M: M\cap H_0] \leq [H:H_0] \leq m <|S_i|$.  The group $S_i$ is not a composition factor of $(M\cap H_0)/(M\cap S)$ since we have an injective homomorphism $(M\cap H_0)/(M\cap S) \hookrightarrow H_0/S$ and $H_0/S$ is abelian.  This proves the claim.

Let $B_1,\ldots,B_m$ be the composition factors of $S/(S\cap H_g)\cong (S H_g)/H_g$.  Note that each $B_j$ is isomorphic to some $S_i$.  

Let $\varphi\colon M\to H/H_g$ be the quotient homomorphism.  We have $\varphi(M\cap S)\subseteq (S H_g)/H_g$.  The groups $B_1,\ldots, B_m$ occur, with multiplicity, as composition factors of $\varphi(M)$ since $\varphi$ is surjective by our assumptions on $M$ and $(S H_g)/H_g$ is normal in $H/H_g$.   By the claim, the group $M/(M\cap S)$, and hence also $\varphi(M)/\varphi(M\cap S)$, has no composition factors isomorphic to any $B_i$.  Therefore, the groups $B_1,\ldots, B_m$ occur, with multiplicity, as composition factors of $\varphi(M\cap S)$.  In particular, $|\varphi(M\cap S)| \geq |B_1|\cdots |B_m| =|(S H_g)/H_g|$. Since $\varphi(M\cap S)\subseteq (S H_g)/H_g$, this implies that $\varphi(M\cap S)= (S H_g)/H_g$.  The lemma follows by noting that $(S H_g)/H_g\cong S/(S\cap H_g)$.
\end{proof}

Define the subgroup $M_0:= M\cap H_0$ of $H_0$ and the subset 
\begin{align} \label{E:C0 defn}
C_0:=\bigcup_{h\in H_0} h M_0 h^{-1}
\end{align}
of $H_0$.  We have $M\cap S=M_0\cap S$ since $S$ is a subgroup of $H_0$.  Therefore,  $M_0\not\supseteq S$.  The natural homomorphism $M_0\cap S \to S/(S\cap H_g)$ is surjective by Lemma~\ref{L:M S surjectivity}.  This surjectivity and $M_0\not\supseteq S$  implies that $S\cap H_g\neq 1$.  The following lemma will be used to reduce to a setting where the group $G$ is connected.

\begin{lemma}  \label{L:connected reduction}
Let $\kappa$ be any coset of $H_g$ in $H_0\cdot H_g$.   Then there is a coset $\kappa_0$ of $S\cap H_g$ in $H_0$ such that 
\begin{equation} \label{E:connected redn}
 \frac{|C\cap \kappa|}{|H_g|} \leq 1 - \frac{1}{e} + \frac{1}{e}\cdot  \frac{|C_0 \cap \kappa_0|}{|S\cap H_g|},
\end{equation}
where $e:=[H_g : S\cap H_g]$.
\end{lemma}
\begin{proof}
We have $\kappa \cap H_0\neq \emptyset$ since $\kappa$ is a $H_g$-coset in $H_0\cdot H_g$.  Therefore, there is a $S\cap H_g$-coset $\kappa_0$ of $H_0$ satisfying $\kappa_0\subseteq \kappa$.   Since $\kappa$ is the disjoint union of $e$ different $S\cap H_g$-cosets, one of which is $\kappa_0$, we have
\[
|C \cap \kappa | \leq (e-1)|S\cap H_g|  + |C \cap \kappa_0|=  |H_g|-|S\cap H_g|  + |C \cap \kappa_0|,
\]
where the inequality uses the trivial upper bound for the $S\cap H_g$-cosets that are not $\kappa_0$.   Dividing by $|H_g|$, we deduce that $|C\cap\kappa|/|H_g| \leq 1-1/e+1/e\cdot |C\cap\kappa_0|/|S\cap H_g|$.  So to prove (\ref{E:connected redn}), it suffices to show that $C\cap \kappa_0= C_0\cap \kappa_0$.   

Since $\kappa_0 \subseteq H_0$ and $C_0 \subseteq H_0$, it thus suffices to show that $C \cap H_0 = C_0$.    Using that $M\to H/H_0$ is surjective, we find that $C = \cup_{h\in H_0} hMh^{-1}$.  Therefore, $C\cap H_0= \cup_{h\in H_0} h (M\cap H_0) h^{-1} = C_0$ as desired.
\end{proof}

\begin{prop}  \label{P:claim:gp}
Take any subgroup $M_1$ of $H_0$ satisfying $M_1\not\supseteq S$ for which $M_1\cap S \to S/(S\cap H_g)$ is surjective.  Define the subset $C_1:=\bigcup_{h\in H_0} h M_1 h^{-1}$ of $H_0$.  Then for any coset $\kappa_0$ of $S\cap H_g$ in $H_0$, we have
\[
\frac{|C_1 \cap \kappa_0|}{|S\cap H_g|}  \leq \delta_0 + O_r(1/\ell)
\]
with $0\leq \delta_0 <1$ depending only on $r$.
\end{prop}

Suppose that Proposition~\ref{P:claim:gp} holds.   Take any $H_g$-coset $\kappa$ in $H_0\cdot H_g$.  By Lemma~\ref{L:connected reduction}, there is a 
$(S\cap H_g)$-coset $\kappa_0$ in $H_0$ such that (\ref{E:connected redn}) holds.  Proposition~\ref{P:claim:gp}, with $M_1=M_0$, implies that $|C_0 \cap \kappa_0|/{|S\cap H_g|}  \leq \delta_0 + O_r(1/\ell)$ holds with a constant $0\leq\delta_0<1$ depending only on $r$. From (\ref{E:connected redn}), we deduce that
\[
\frac{|C \cap \kappa|}{|H_g|} \leq \delta + O_r(1/\ell)
\]
with $\delta:=1-1/e+\delta_0/e=1+(-1+\delta_0)/e$.  We have $0\leq \delta<1$ since $0\leq \delta_0<1$.      We have 
\[
e\leq [G(\FF_\ell): S] \leq m\cdot  [G^\circ(\FF_\ell): S] \ll_{r} m.
\]   
So  $\delta<1$ depends only on $r$ and $m$.   This completes the proof of Proposition~\ref{P:coset gp theory} assuming Proposition~\ref{P:claim:gp}.\\

It thus remains to prove Proposition~\ref{P:claim:gp}.  Note that the only role that $H_g$ plays in the proposition is through its subgroup $S\cap H_g$, so without loss generality assume that $H_g$ is a normal subgroup of $S$.  Take any subgroup $M_1\subseteq H_0$ such that $M_1\not\supseteq S$ and such that $M_1\cap S \to S/(S\cap H_g)=S/H_g$ is surjective.  We thus have $H_g\neq 1$.  Define $C_1:=\bigcup_{h\in H_0} h M_1 h^{-1}$.     Since the proposition now only concerns subgroups of $G^\circ(\FF_\ell)$, we may assume without loss of generality that $G$ is connected and hence that $H=H_0$.      

\begin{lemma} \label{L:proj coset}
It suffices to prove Proposition~\ref{P:claim:gp} with the additional assumption that the projection homomorphism $M_1\cap S \hookrightarrow S \to \prod_{j\in J} S_j$ is surjective for all proper subsets $J\subseteq \{1,\ldots,n\}$.
\end{lemma}
\begin{proof}
    Since $M_1\not\supseteq S$, there is a minimal (non-empty) subset $I \subseteq \{1,\ldots, n\}$ for which the projection $M_1\cap S \to \prod_{i\in I} S_i$ is not surjective.   Define the projection
\[
\varphi \colon G(\FF_\ell) \to {\prod}_{i\in I} G_i(\FF_\ell) = \tilde G(\FF_\ell),
\]
where $\tilde{G}:=\prod_{i\in I} G_i$.  
Define the groups $\tilde H:=\varphi(H)$, $\tilde H_g:=\varphi(H_g)$ and $\tilde M_1 := \varphi(M_1)$.     The group $\varphi(S)$ equals $\tilde S:=\prod_{i\in I} S_i$.     By our choice of $I$, we have $\tilde M_1 \not\supseteq \tilde S$.  The natural homomorphism $M_1\cap S \to S/H_g$ is surjective and hence so is $\tilde M_1 \cap \tilde S \to \tilde S/ \tilde H_g$.

Take any coset $\kappa_0$ of $S\cap H_g=H_g$ in $H$.   Define $\tilde\kappa_0:=\varphi(\kappa_0)$; it is a coset of $\tilde H_g$ in $\tilde H$.   The map $\kappa_0 \to \tilde{\kappa}_0$, $x\mapsto \varphi(x)$ is $b$-to-$1$ with $b:=|H_g|/|\tilde H_g|$.    Therefore,
\begin{equation} \label{E:proj reduction}
\frac{|C_1 \cap \kappa_0|}{|H_g|} \leq \frac{|\varphi(C_1) \cap \varphi(\kappa_0)| \cdot b}{|H_g|} = \frac{|\varphi(C_1) \cap \tilde\kappa_0| \cdot b}{ |\tilde H_g|\cdot b} = \frac{|\tilde C_1 \cap \tilde\kappa_0|}{|\tilde H_g|},
\end{equation}
where $\tilde C_1:= \cup_{h \in \tilde H} h \tilde M_1 h^{-1}$.

The inequality (\ref{E:proj reduction}) shows that is suffices to prove Proposition~\ref{P:claim:gp} with $(G,H,H_g,M_1)$ replaced by $( \tilde G,\tilde H,\tilde H_g,\tilde M_1)$; note that we have already verified the required properties and that the rank of $\tilde G$ is at most the rank of $G$.   By the minimality of our choice of $I$, the projection $\tilde M_1 \cap \tilde S \to \prod_{j\in J} S_j$ is surjective for each proper subset $J\subseteq I$.  The lemma is now immediate.
\end{proof}

By Lemma~\ref{L:proj coset}, we may now assume that the projection $M_1\cap S \to \prod_{j\in J} S_j$ is surjective for every proper subset $J\subseteq \{1,\ldots,n\}$.    

\begin{lemma} \label{L:n is 1 or 2}
To prove Proposition~\ref{P:claim:gp}, it suffices to assume that $n=1$ or that $n=2$  and there is a group isomorphism $f\colon S_1\to S_2$ such that $M_1\cap S =\{(s,f(s)): s \in S_1\}$.
\end{lemma}
\begin{proof}
Suppose that $n\geq 3$.  Then the projection $M_1\cap S \to S_i \times S_j$ is surjective for all distinct $i,j\in \{1,\ldots, n\}$.   Since the groups $S_i$ have no nontrivial abelian quotients, Lemma~5.2.2 of \cite{MR0457455} implies that $M_1\cap S = S$.   However, this is a contradiction since $M_1\not\supseteq S$ by assumption.   Therefore, $n\leq 2$.

Suppose that $n=2$.   The projection $M_1\cap S\to S_i$ is surjective for $i\in \{1,2\}$.   Using that $M_1\not\supseteq S$ and that the non-abelian  groups $S_i$ are simple, Goursat's lemma (\cite{MR0457455}*{Lemma~5.1.1}) implies that $M_1\cap S = \{(s,f(s)) : s\in S_1\}$ for some group isomorphism $f\colon S_1\to S_2$.
\end{proof}

By Lemma~\ref{L:n is 1 or 2}, we may assume that $n\leq 2$.  

\noindent $\bullet$ First consider the case $n=1$.  Since $H_g$ is a non-trivial normal subgroup of $S=S_1$ and $S_1$ is simple, we have $H_g=S$.   Take any coset $\kappa_0$ of $S\cap H_g =S$ in $H_0$. 

Since $n=1$, the connected and adjoint group $G$ is simple.  
There is an integer $e\geq 1$ and a connected, {geometrically simple} and adjoint linear algebraic group $\calG$ defined over $\FF_{\ell^e}$ such that $G$ is isomorphic to the Weil restriction $\Res_{\FF_{\ell^e}/\FF_\ell}(\calG)$, cf.~\cite{MR1632779}*{Theorem~26.8}.   Without loss of generality, we may assume that $G=\Res_{\FF_{\ell^e}/\FF_\ell}(\calG)$ and hence $G(\FF_\ell)=\calG(\FF_{\ell^e})$.   In particular, we can view $S$ as the commutator subgroup of $\calG(\FF_{\ell^e})$, and $H_0$ and $M_1$ as subgroups of $\calG(\FF_{\ell^e})$.   

Define $C_1=\bigcup_{h\in H_0} h M_1 h^{-1}$.   Let $H_1$ be the subgroup of $H_0$ generated by $M_1$ and $S$; it is a normal subgroup of $H_0$.  If $\kappa_0 \cap H_1 = \emptyset$, then $C_1\cap \kappa_0 = \emptyset$ and the bound of Proposition~\ref{P:claim:gp} is trivial for the coset $\kappa_0$.   So we may assume that $\kappa_0$ is an $S$-coset in $H_1$.   Since $M_1\to H_1/S$ is surjective, we can replace $M_1$ by a maximal subgroup of $H_1$; it still will satisfy the conditions of Proposition~\ref{P:claim:gp} and the set $C_1$ will only get larger.

Note that the rank of $\calG$ is at most $r$.  By Theorem~\ref{T:Fulman-Guralnick} and Remark~\ref{R:derangement equivalence},  applied with the algebraic group $\calG/\FF_{\ell^e}$, we have
\[
\frac{|C_1 \cap \kappa_0|}{|S\cap H_g|} = \frac{|C_1 \cap \kappa_0|}{|S|}=1-\delta(\kappa_0,H_1/M_1) \leq 1- \delta + O_r(1/\ell)
\]
with a constant $0<\delta\leq 1$ depending only on $r$.   This completes the proof of Proposition~\ref{P:claim:gp} in the case $n=1$.\\

\noindent $\bullet$ Finally, consider the case $n=2$.   Since $H_g\neq 1$ is a normal subgroup of $S$, $H_g$ is equal to $\{1\}\times S_2$,  $S_1\times\{1\}$ or $S_1\times S_2$. The following lemma allows us to make some further reductions.

\begin{lemma} \label{L:final case n 2}
It suffices to prove Proposition~\ref{P:claim:gp} in the case $n=2$ with $G_1=G_2$,  $H_g=S_1\times \{1\}$ and $M_1=\{(g,g) : g \in G_1(\FF_\ell)\}$.
\end{lemma}
\begin{proof}
Using Lemma~\ref{L:n is 1 or 2}, we make an identification $S_1=S_2$ of abstract groups so that $M_1\cap S =\{(s,s): s \in S_1\}$.   Since the groups $G_i$ are adjoint, we find that conjugation gives a faithful action of $G_i(\FF_\ell)$ on $S_1=S_2$.   So we may identify $G_i(\FF_\ell)$ with a subgroup of $\Aut(S_1)$.  

Take any $(g_1,g_2)\in M_1$.   Since $M_1\cap S$ is a normal subgroup of $M_1$, we have $g_1 s g_1^{-1} = g_2 s g_2^{-1}$ for all $s\in S_1=S_2$.    Therefore, $g_1$ and $g_2$ are equal elements of $\Aut(S_1)$.   Therefore, $M_1$ is a subgroup of $\{(g,g): g \in G_1(\FF_\ell)\}$.   To prove the claim, there is no harm in increasing $M_1$ to be equal to $\{(g,g): g \in G_1(\FF_\ell)\}$; it also does not contain $S=S_1\times S_2$.   We may thus assume that $G_1=G_2$.

We have already observed that $H_g\in \{\{1\}\times S_2, S_1\times\{1\}, S_1\times S_2\}$.    By symmetry, we may assume that $H_g$ is $S_1\times\{1\}$ or $S_1\times S_2$.  From our explicit description of $M_1$, the homomorphism $M_1 \cap S \to S/(S_1\times\{1\})$ is surjective.   So we may assume that $H_g=S_1\times\{1\}$; note that the $S_1\times S_2$ cosets can be broken up into $S_1\times\{1\}$-cosets.
\end{proof}

We finally assume that we are in the setting of Lemma~\ref{L:final case n 2}.   Take any coset $\kappa_0$ of $S_1\times\{1\}$ in $G(\FF_\ell)$.    We have $\kappa_0= \alpha S_1 \times \{\beta\}$ for some $\alpha,\beta\in G_1(\FF_\ell)$. Using our explicit description of $M_1$, we have
\[
C_1 \cap \kappa_0 \subseteq \{ (g,\beta) : \text{$g\in G_1(\FF_\ell)$ is conjugate to $\beta$ in $G_1(\FF_\ell)$}\}.
\]
Therefore, $|C_1 \cap \kappa_0| \leq |G_1(\FF_\ell)|/|\calC_\beta|$, where $\calC_\beta$ is the centralizer of $\beta$ in $G_1(\FF_\ell)$.  If $\beta$ is semisimple in $G_1$, then it lies in a maximal torus (of rank $r$) and hence $|\calC_\beta|\gg_r \ell^r$.  If $\beta$ is not semisimple, then it commutes with a non-trivial unipotent element of $G_1(\FF_\ell)$ (whose order is a power of $\ell$).   Therefore,
\[
|C_1 \cap \kappa_0| \leq |G_1(\FF_\ell)|/|\calC_\beta| \ll_r  |G_1(\FF_\ell)|/\ell \ll_r |S_1|/\ell,
\]
where the last inequality uses that $[G_1(\FF_\ell):S_1]$ can be bounded in terms of $r$.   

We deduce that $|C_1 \cap \kappa_0|/{|H_g|}=|C_1 \cap \kappa_0|/{|S_1|} \ll_r 1/\ell$.  This completes our proof of Proposition~\ref{P:claim:gp}.  

\section{Hilbert irreducibility} \label{S:HIT}

  Fix an abelian scheme $A\to U$ of relative dimension $g\geq 1$, where $U$ is a non-empty open subscheme of $\PP^n_K$ with $n\geq 1$ and $K$ a number field.   Take any constant $b_A$ as in Theorem~\ref{T:geometric gp theory}.   For each prime $\ell\geq b_A$, we have $\bbar\rho_{A,\ell}(\pi_1(U)) \supseteq \calS_{A,\ell}(\FF_\ell)'$.    For a prime $\ell\geq b_A$ and real $x$, define the set
\[
B_\ell(x):=\{ u \in U(K) : H(u)\leq x,\,    \bbar\rho_{A_u,\ell}(\Gal_K) \not \supseteq \calS_{A,\ell}(\FF_\ell)' \}.
\]
  
In this section, we will prove the following.

\begin{thm} \label{T:almost there mod ell}
For each prime $\ell\geq b_A$ and  $x\geq 2$, we have 
\[
|B_\ell(x)| \ll_{A} \,\, (\ell+1)^{3g(2g+1)/2} \cdot x^{[K:\QQ](n + 1/2)} \log x + (\ell+1)^{(6n+15/2)g(2g+1)}.
\]
\end{thm}

\begin{remark}
In our application, we will use Theorem~\ref{T:almost there mod ell} when $\ell\leq c (\log x)^\gamma$ for some positive constants $c$ and $\gamma$ depending only on $A$.  For such $\ell$, we obtain a bound $|B_\ell(x)| \ll_A x^{[K:\QQ](n + 1/2)} (\log x)^{\gamma'}$ for a constant $\gamma'$.
\end{remark}

\subsection{A special version of Hilbert irreducibility} \label{SS:special HIT}
We now state a specialized version of Hilbert's irreducibility theorem.  To ease notation and make it suitable for future use, we keep it separate from our abelian variety application.  

Let $K$ be a number field.  Fix a non-empty open subvariety $U$ of $\PP^n_K$ with $n\geq 1$ and a continuous representation 
\[
\rho\colon \pi_1(U) \to G(\FF_\ell),
\]
where $G$ is a linear algebraic group defined over $\FF_\ell$ for which the neutral component $G^\circ$ is reductive.   Let $S$ be the commutator subgroup of $G^\circ(\FF_\ell)$.    Assume further that $\rho$ satisfies $\rho(\pi_1(U)) \supseteq S$.

For each point $u\in U(K)$, we obtain a homomorphism $\rho_u\colon \Gal_K\to G(\FF_\ell)$ by specializing $\rho$ at $u$; it is uniquely defined up to conjugation by an element of $G(\FF_\ell)$.  Hilbert's irreducibility theorem implies that $\rho_u(\Gal_K)\supseteq S$ for all $u\in U(K)$ away from a set of density $0$; Theorem~\ref{T:HIT main} below gives a quantitative version.   

We first define some quantities for which the implicit constant of our theorem depends on.  Let $\calU$ be the open subscheme of $\PP^n_{\OO_K}$ that is the complement of the Zariski closure of $\PP^n_K-U$ in $\PP^n_{\OO_K}$.   The $\OO_K$-scheme $\calU$ has generic fiber $U$.   Fix a finite set $\scrS$ of non-zero prime ideals of $\OO_K$ such that $\rho$ arises from a homomorphism $\pi_1(\calU_\OO)\to G(\FF_\ell)$, where $\OO$ is the ring of $\scrS_\ell$-integers in $K$ and $\scrS_\ell$ is the set of non-zero prime ideals $\p$ of $\OO_K$ that lie in $\scrS$ or divide $\ell$.  

Let $\alpha\colon \pi_1(U)\to G(\FF_\ell)/G^\circ(\FF_\ell)$ be the homomorphism obtained by composing $\rho$ with the obvious quotient map.   Let $F\subseteq \Kbar$ be the minimal extension of $K$ for which $\alpha(\pi_1(U_F))=\alpha(\pi_1(U_{\Kbar}))$.  

Let $G^\ad$ be the quotient of $G$ by the center of $G^\circ$.   The group $(G^\ad)^\circ$ is an adjoint algebraic group over $\FF_\ell$.  Denote the rank and dimension of $(G^\ad)^\circ$ by $r$ and $d$, respectively.  Define the index $m=[G(\FF_\ell):G^\circ(\FF_\ell)]$.

\begin{thm} \label{T:HIT main}
Fix notation and assumptions as above and take any $x\geq 2$.
There is a constant $c$, depending only on $r$ and $m$, such that if $\ell\geq c$, then 
 \begin{align*}
& |\{ u \in U(K) : H(u)\leq x,\, \rho_u(\Gal_K)\not\supseteq S\}| \\&\ll_{U,F,|\scrS|,r,m} \,\, (\ell+1)^{3d/2} \cdot x^{[K:\QQ](n + 1/2)} \log x + (\ell+1)^{(6n+15/2)d}.
 \end{align*}  
\end{thm}

\subsection{Proof of Theorem~\ref{T:HIT main}}
We assume that $\ell \geq c$, where $c$ is a constant that depends only on $r$ and $m$; we will allow ourselves to appropriately increase $c$ throughout the proof while maintaining the dependencies. 

Let $G^\ad$ be the quotient of $G$ by the center of $G^\circ$ and let $\pi\colon G\to G^\ad$ be the quotient map.     The morphism $\pi$ gives rise to homomorphisms $G(\FF_\ell)\to G^\ad(\FF_\ell)$ and $S\to S^\ad$, where $S^\ad$ is the commutator subgroup of $(G^\ad)^\circ(\FF_\ell)$.  Let 
\[
\rho^\ad\colon \pi_1(U)\to G^\ad(\FF_\ell),
\]
be the representation obtained by composing $\rho$ with $\pi$.      

Before proceeding, the following lemma gives an alternate description of $S$ and $S^\ad$ for all sufficiently large $\ell$.

\begin{lemma} \label{L:alternate commutator}
Let $G$ be a connected reductive group over $\FF_\ell$.   Let $H$ be the derived subgroup of $G$ and let $\varphi\colon H^\sc \to H$ be its  simply connected cover.   Then $G(\FF_\ell)'$ is perfect and equal to $\varphi(H^\sc(\FF_\ell))$ for all $\ell \gg_r 1$, where $r$ is the rank of $H$.
\end{lemma}
\begin{proof}
Since $H^\sc$ is simply connected, it is a product of simply connected and simple groups $H_1,\ldots, H_m$.   We know that each group $H_i(\FF_\ell)$ is perfect for $\ell\gg 1$ (moreover, the quotient by its center is a non-abelian simple group).   In particular, we may assume that the group $H^\sc(\FF_\ell)$ is perfect.  Therefore, $\varphi(H^\sc(\FF_\ell))=\varphi(H^\sc(\FF_\ell)')\subseteq H(\FF_\ell)'$.  The group $\varphi(H^\sc(\FF_\ell))$ is perfect since $H^\sc(\FF_\ell)$ is perfect.

Let $Y$ be the kernel of $\varphi$; it is commutative since the isogeny $\varphi$ is central.  The degree of $\varphi$, and hence also the cardinality of $Y(\FFbar_\ell)$, can be bounded in terms of $r$.   Galois cohomology gives an injective homomorphism $H(\FF_\ell)/\varphi(H^\sc(\FF_\ell)) \hookrightarrow H^1(\Gal_{\FF_\ell}, Y(\FFbar_\ell))$ of groups, so $H(\FF_\ell)/\varphi(H^\sc(\FF_\ell))$ is abelian and its cardinality can be bounded in terms of $r$.   In particular, we have $\varphi(H^\sc(\FF_\ell)) \supseteq H(\FF_\ell)'$.  

We thus have $\varphi(H^\sc(\FF_\ell)) = H(\FF_\ell)'$ since we have shown both inclusions and we have seen that $\varphi(H^\sc(\FF_\ell))$ is perfect.   We clearly have $H(\FF_\ell)'\subseteq G(\FF_\ell)'$ so it suffices to show that $G(\FF_\ell)' \subseteq \varphi(H^\sc(\FF_\ell))$ for all $\ell \gg_r 1$.

Let $G(\FF_\ell)^+$ be the (normal) subgroup of $G(\FF_\ell)$ generated by its elements of order $\ell$.    We have $G(\FF_\ell)^+ \subseteq H(\FF_\ell)$ since $G/H$ is a torus over $\FF_\ell$ and hence has no $\FF_\ell$-points of order $\ell$. We have already shown that $[H(\FF_\ell):\varphi(H^\sc(\FF_\ell))]\leq C_r$, where $C_r$ is a constant depending only on $r$.   By taking $\ell>C_r$, we find that $G(\FF_\ell)^+ \subseteq  \varphi(H^\sc(\FF_\ell))$.  The group $G(\FF_\ell)/G(\FF_\ell)^+$ is abelian by \cite{MR3566639}*{Proposition~1.1} and hence $G(\FF_\ell)'\subseteq G(\FF_\ell)^+$.   Therefore, $G(\FF_\ell)'\subseteq G(\FF_\ell)^+\subseteq  \varphi(H^\sc(\FF_\ell))$.
\end{proof}

\begin{lemma} \label{L:S to S'}
If $\ell \gg_r 1$ and $M$ is a subgroup of $G(\FF_\ell)$, then $S$ is a subgroup of $M$ if and only if $S^\ad$ is a subgroup $\pi(M)$.
\end{lemma}
\begin{proof} 
Let $H$ be the derived subgroup of $G^\circ$; it has rank $r$   Let $\varphi\colon H^\sc\to H$ be the simply connected cover of $H$.  Define $\varphi^\ad:=\pi\circ \varphi \colon H^\sc\to (G^\ad)^\circ$; it is the simply connected cover of $(G^\ad)^\circ$.   By assuming $\ell$ is sufficiently large in terms of $r$,  Lemma~\ref{L:alternate commutator} implies that $S=\varphi(H^\sc(\FF_\ell))$, $S^\ad=\varphi^\ad(H^\sc(\FF_\ell))$, and that $S$ and $S^\ad$ are perfect.  If $M$ is a subgroup of $G(\FF_\ell)$ containing $S$, then 
\[
\pi(M) \supseteq \pi(S)=\pi(\varphi(H^\sc(\FF_\ell)))=\varphi^\ad(H^\sc(\FF_\ell))=S^\ad.
\]

Now let $M$ be a subgroup of $G(\FF_\ell)$ that satisfies $\pi(M)\supseteq S^\ad$.   We need to show that $M\supseteq S$.   We have $\pi^{-1}((G^\ad)^\circ) \subseteq G^\circ$, so there is no harm in replacing $M$ by the smaller group $M\cap G^\circ(\FF_\ell)$.  In particular, we may assume that $M\subseteq G^\circ(\FF_\ell)$.    We have $\pi(M') = \pi(M)' \supseteq (S^\ad)'=S^\ad$.   So after replacing $M$ by the smaller group $M'$, we may further assume that $M \subseteq S$.   We thus have $\pi(M)=S^\ad$ since $\pi(S)=S^\ad$.  The homomorphism $S\xrightarrow{\pi} S^\ad$ is surjective and its kernel $Z$ lies in the center of $G(\FF_\ell)$ since $\pi$ is a central isogeny, so we have $S=M\cdot Z$.   Taking commutator subgroups, we find that $S=S' =(M\cdot Z)'=M'$.   Since $M'\subseteq M\subseteq S$, we deduce that $M=S$.
\end{proof}

Take any $u\in U(K)$.   Specializing $\rho^\ad$ at $u$ gives a representation $ \rho^\ad_u\colon \Gal_K \to G^\ad(\FF_\ell)$.  Up to an inner automorphism, $ \rho^\ad_u$ agrees with $\pi\circ \rho_u$.  By suitably increasing the constant $c$, Lemma~\ref{L:S to S'} implies that $\rho_u(\Gal_K) \supseteq S$ if and only $\rho^\ad_u(\Gal_K) \supseteq S^\ad$.    So to prove the theorem, we need to bound the cardinality of the set
\[
\{ u \in U(K) : H(u)\leq x,\,  \rho^\ad_u(\Gal_K)\not\supseteq  S^\ad\}.
\]

We now show that the assumptions of Theorem~\ref{T:HIT main} hold for $\rho^\ad$ and relate its basic invariants to those of $\rho$.  Lemma~\ref{L:S to S'} and the assumption $\rho(\pi_1(U))\supseteq S$ implies that $ \rho^\ad(\pi_1(U))\supseteq  S^\ad$.   By our assumption on $\rho$, the representation $\rho^\ad$ arises from the homomorphism $\pi_1(\calU_\OO)\to G(\FF_\ell)\xrightarrow{\pi} G^\ad(\FF_\ell)$,  where $\OO$ is the ring of $\scrS_\ell$-integers in $K$.   The quantities $r$ and $d$ associated to $\rho$ and $\rho^\ad$ are the same.

There is no harm in replacing $G$ by the algebraic subgroup generated by $G^\circ$ and $\rho(\pi_1(U))$.  In particular, we may assume that $G$ is generated by $G^\circ$ and $G(\FF_\ell)$.    From this we find that the natural map $G(\FF_\ell)/G^\circ(\FF_\ell) \to G^\ad(\FF_\ell)/ (G^\ad)^\circ(\FF_\ell)$ is an isomorphism.   Therefore, $\alpha$, $F$ and $m$ are the same for $\rho$ and $\rho^\ad$.

It is now clear that proving Theorem~\ref{T:HIT main} for $\rho^\ad$ will give the desired bound for $\rho$.\\   

Without loss of generality, we may now assume that the group $G^\circ$ is adjoint.  Define the subgroups
\[
H:=\rho(\pi_1(U)),\quad H_g:=\rho(\pi_1(U_{\Kbar})) \quad \text{ and } \quad H_0 := H \cap G^\circ(\FF_\ell)
\]
of $G(\FF_\ell)$.   Let $\calM$ be the set of subgroups $M$ of $H$ for which $M\not\supseteq S$ and for which 
the quotient maps $M\to H/H_g$ and $M\to H/H_0$ are surjective.

\begin{lemma} \label{L:HIT for a fixed M}
Take any $x\geq 2$.   There is a constant $c$, depending only on $r$ and $m$, such that if $\ell\geq c$, then 
\[
|\{ u \in U(K) : \rho_u(\Gal_K) \in \calM \}|  \ll_{U,F,|\scrS|,r,m} \, (\ell+1)^{3d/2} \cdot x^{[K:\QQ](n + 1/2)} \log x + (\ell+1)^{(6n+15/2)d}.
\]
\end{lemma}
\begin{proof}
Let $L$ be the minimal extension of $K$ in $\Kbar$ for which $H_g$ is the image of $\pi_1(U_L)$ under $\rho$.   We have a natural short exact sequence
\[
1\to H_g\to H \xrightarrow{\varphi} \Gal(L/K) \to 1.
\]
Observe that $F$ is the subfield of $L$ that satisfies $\varphi^{-1}(\Gal(L/F))=\rho(\pi_1(U_F))=H_0\cdot H_g$.

Fix $x\geq 2$.  Let $\widetilde \calM$ be the set of maximal elements of $\calM$ with respect to inclusion.  Take any group $M\in \widetilde\calM$ and define the subset $C:=\bigcup_{h \in H} hMh^{-1}$ of $H$.  By  appropriately increasing the constant $c$, Proposition~\ref{P:coset gp theory} says that there is a constant $0\leq \delta <1$, depending only on $r$ and $m$, such that ${|C \cap \kappa| }/{ |H_g| } \leq \delta$ holds for every coset $\kappa$ of $H_g$ in $H_0 \cdot H_g$.    By Theorem~\ref{T:HIT}, we have \[
 |\{ u \in U(K) : H(u)\leq x,\, \rho_u(\Gal_K)\subseteq C\}| \ll_{U, F,\delta} \,\, x^{(n+1/2)[K:\QQ]} \log x + |\scrS_\ell|^{4n+4} + |C|^{2n+2}\cdot |H_g|^{4n+4}.
 \]  
Since $|\scrS_\ell|\leq |\scrS| + [K:\QQ]$ and $\delta$ depends only on $r$ and $m$, we deduce that
 \begin{align*}
& |\{ u \in U(K) : H(u)\leq x,\, \rho_u(\Gal_K) \text{ is conjugate in $H$ to a subgroup of } M\}| \\
& \ll_{U, F,r, m} \,\, x^{(n+1/2)[K:\QQ]} \log x + |\scrS|^{4n+4} +  |G(\FF_\ell)|^{6n+6}.
 \end{align*}
 By summing over all $M\in \widetilde\calM$ and using that the implicit constant depends on $m$, we have
 \[
 |\{ u \in U(K) : \rho_u(\Gal_K) \in \calM \}| \ll_{U, F,r, m} \, |\widetilde\calM| \cdot \big(x^{(n+1/2)[K:\QQ]} \log x + |\scrS|^{4n+4} +  |G^\circ(\FF_\ell)|^{6n+6}\big).
 \]
We now bound $|\widetilde \calM|$.  Take any $M\in \widetilde \calM$ and define the subgroup $\widetilde H:=M\cdot S$ of $G(\FF_\ell)$.   Observe that $M$ is a maximal subgroup of $\widetilde H$ (if not, then it would give rise to a larger group in the set $\calM$).   By \cite{MR2360145}, the group $\widetilde H$ has at most $O(|\widetilde H|^{3/2})$ maximal subgroups, where the constant is absolute.  Therefore, 
\[
|\widetilde \calM| \ll |G(\FF_\ell)|^{3/2}\cdot |\{\widetilde H : \widetilde H \text{ subgroup of $G(\FF_\ell)$ containing $S$}\}|.
\]
We obtain $|\widetilde \calM| \ll_{r,m} |G(\FF_\ell)|^{3/2}\leq m \cdot |G^\circ(\FF_\ell)|^{3/2}$ by using that the order of the quotient group $G(\FF_\ell)/S$ can be bounded in terms of $r$ and $m$.  Therefore,
 \[
 |\{ u \in U(K) : \rho_u(\Gal_K) \in \calM \}| \ll_{U, F,r,m} \, |G^\circ(\FF_\ell)|^{3/2} \big(x^{(n+1/2)[K:\QQ]} \log x + |\scrS|^{4n+4} +  |G^\circ(\FF_\ell)|^{6n+6}\big).
 \]
 The bound in the lemma follows by noting that $|G^\circ(\FF_\ell)| \leq (\ell+1)^d$, cf.~\cite{MR880952}*{Lemma~3.5}.
\end{proof}

Let $\beta\colon \pi_1(U)\to H/H_0$ be the surjective representation obtained by composing $\rho$ with the obvious quotient map.  As usual, we have a specialization $\beta_u\colon \Gal_k\to H/H_0$ for each $u\in U(k)$.

\begin{lemma} \label{L:HIT beta}
We have
\[
|\{ u \in U(K) : H(u)\leq x,\, \beta_u(\Gal_k) \neq H/H_0 \}|  \ll_{U,F,|\scrS|,m} \,  x^{[k:\QQ](n + 1/2)} \log x + |\scrS|^{4n+4}.
\]
\end{lemma}
\begin{proof}
To ease notation, define $Y:=H/H_0=\beta(\pi_1(U))$ and its normal subgroup $Y_g:=\beta(\pi_1(U_{\Kbar}))$.  The field $F$ is the smallest extension of $K$ in $\Kbar$ for which $\beta(\pi_1(U_{F}))=Y_g$.  The homomorphism $\beta$ arises from a continuous homomorphism $\pi_1(\calU_{\OO}) \to Y$, where $\OO$ is the ring of $\scrS_\ell$-integers (since $\rho$ arises from a representation of $\pi_1(\calU_{\OO})$).  By Corollary~\ref{C:HIT}, we have
\[
 |\{ u \in U(K) : H(u)\leq x,\, \beta_u(\Gal_K)\neq Y\}| \ll_{U,F,|Y|} \, x^{[K:\QQ](n+1/2)}\log x + |\scrS_\ell|^{4n+4}.
 \]
 The lemma follows by noting that $|\scrS_\ell| \leq |\scrS|+[K:\QQ]$ and $|Y|\leq m$.
\end{proof}

Take any $u\in U(K)$ for which $S$ is \emph{not} a subgroup of $\rho_u(\Gal_K)$.   The natural map $\rho_{u}(\Gal_K)\to H/H_g$ is always surjective.   If $\beta_u(\Gal_K)=H/H_0$ (equivalently, if the natural map $\rho_u(\Gal_K)\to H/H_0$ is surjective), then we must have $\rho_u(\Gal_K) \in \calM$.  Therefore, $ |\{ u \in U(K) : H(u)\leq x,\, \rho_u(\Gal_K)\not\supseteq S\}|$ is less than or equal to
\[
 |\{ u \in U(K) : \beta_u(\Gal_K) \neq H/H_0 \}|
 + |\{ u \in U(K) : \rho_u(\Gal_K) \in\calM \}|.
 \]
The theorem is now a direct consequence of Lemmas~\ref{L:HIT for a fixed M} and \ref{L:HIT beta}.

\subsection{Proof of Theorem~\ref{T:almost there mod ell}}

Take any prime $\ell\geq b_A$.   Corollary~\ref{C:HIT} with our assumption $\bbar\rho_{A,\ell}(\pi_1(U)) \supseteq \calS_{A,\ell}(\FF_\ell)'$ implies that $|B_\ell(x)| \ll_{A,\ell} x^{[K:\QQ] (n+1/2)} \log x$, where the implicit constant depends on $\ell$.   So during our proof we may exclude a finite number of primes $\ell$.

Define the linear algebraic group $G:=(\calG_{A,\ell})_{\FF_\ell}$ over $\FF_\ell$.  By Theorem~\ref{T:geometric gp theory}(\ref{T:geometric gp theory i}), $G^\circ=(\calG_{A,\ell}^\circ)_{\FF_\ell}$ is reductive with derived subgroup $(\calS_{A,\ell})_{\FF_\ell}$.   The rank of $G^\circ$ is bounded in terms of $g$ since it is isomorphic to an algebraic subgroup of $\GL_{2g,\FF_\ell}$.

Let $S$ be the commutator subgroup of $G^\circ(\FF_\ell)$.  By first excluding a finite number of primes,  Lemma~\ref{L:alternate commutator} implies that $S=G^\circ(\FF_\ell)'$ equals $\calS_{A,\ell}(\FF_\ell)'$.  We have a representation
\[
\rho:=\bbar\rho_{A,\ell}\colon\pi_1(U)\to G(\FF_\ell).
\]
For each $u\in U(K)$, specializing $\rho$ at $u$ gives a representation $\rho_u\colon \Gal_K\to G(\FF_\ell)$ that is uniquely defined up to an inner automorphism of $G(\FF_\ell)$ and agrees with $\bbar\rho_{A_u,\ell}$.  In particular, we have
\[
B_\ell(x)=\{u\in U(K): H(u)\leq x, \, \rho_u(\Gal_K)\not\supseteq S\}.
\]  
We are thus in the setting of \S\ref{SS:special HIT} and hence we can define $\calU$, $r$, $d$, $m$ and $F$ as in that section.    

\begin{lemma} \label{L:spread out}
There is a finite set $\scrS$ of non-zero prime ideals of $\OO_K$, not depending on $\ell$, such that $\bbar\rho_{A,\ell}$ arises from a homomorphism $\pi_1(\calU_\OO)\to G(\FF_\ell)$, where $\scrS_\ell$ is the set of prime ideals of $\OO_K$ that lies in $\scrS$ or divides $\ell$ and $\OO$ is the ring of $\scrS_\ell$-integers in $K$.
\end{lemma}
\begin{proof}
We first ``spread out'' $A$.   There is an abelian scheme $\calA'\to \calU_{\OO'}$, where $\OO'$ is the ring of $\scrS$-integers in $K$ for some finite set $\scrS$ of nonzero prime ideals of $\OO_K$ such that the fiber over $(\calU_{\OO'})_K=U$ is the abelian scheme $A\to U$.    
 
Let $\OO$ be the ring of $\scrS_\ell$-integers in $K$, where $\scrS_\ell$ is the set of prime ideals of $\OO_K$ that lie in $\scrS$ or divide $\ell$.  Let $\calA$ be the abelian scheme over $\calU_{\OO}$ obtained from $\calA'$ by base change.   The $\ell$-torsion subscheme $\calA[\ell]$ of $\calA$ can be viewed as locally constant sheaf of $\ZZ/\ell\ZZ$-modules on $\calU_\OO$ that is free of rank $2g$. The fiber of $\calA[\ell]$ over $U=(\calU_\OO)_K$ is $A[\ell]$.   Since $\rho=\bbar\rho_{A,\ell}$ is the representation associated to $A[\ell]$, we find that $\rho$ arises via base change from a representation of $\pi_1(\calU_\OO)$. 
\end{proof}

Let $\scrS$ be a set of prime ideals as in Lemma~\ref{L:spread out}.   We may assume that $\scrS$ is chosen so that $|\scrS|$ is minimal and hence $|\scrS|\ll_A 1$.

\begin{lemma} \label{L:Frm}
As the prime $\ell\geq b_A$ varies, there are only finitely many possibilities for $F$, $r$ and $m$.  
\end{lemma}
\begin{proof}
Let $\alpha \colon \pi_1(U)\to G(\FF_\ell)/G^\circ(\FF_\ell)$ be the homomorphism obtained by composing $\rho$ with the obvious quotient map.   The field $F\subseteq \Kbar$ is the minimal extension of $K$ for which $\alpha(\pi_1(U_F))=\alpha(\pi_1(U_{\Kbar}))$.  

Using that $G_{A,\ell}$ is the Zariski closure of a subset of $\calG_{A,\ell}(\ZZ_\ell)$, we find that the natural map
$\calG_{A,\ell}(\ZZ_\ell)/\calG_{A,\ell}^\circ(\ZZ_\ell) \to G_{A,\ell}(\QQ_\ell)/G_{A,\ell}^\circ(\QQ_\ell)$ is an isomorphism.  Using Hensel's lemma, we find that the reduction module $\ell$ homomorphism $\calG_{A,\ell}(\ZZ_\ell)/\calG_{A,\ell}^\circ(\ZZ_\ell) \to G(\FF_\ell)/G^\circ(\FF_\ell)$ is surjective.  Therefore, $\alpha$ can be obtained by composing the homomorphism $\gamma_{A,\ell}$ of \S\ref{SS:neutral component} with a surjective homomorphism $G_{A,\ell}(\QQ_\ell)/G_{A,\ell}^\circ(\QQ_\ell)\to G(\FF_\ell)/G^\circ(\FF_\ell)$.   In particular, $m$ is at most $[G_{A,\ell}(\QQ_\ell):G_{A,\ell}^\circ(\QQ_\ell)]$ which is independent of $\ell$ by Lemma~\ref{L:connected independence families}(\ref{L:connected independence families i}).  The field $F$ is contained in the minimal extension $F'\subseteq \Kbar$ of $K$ for which $\gamma_{A,\ell}(\pi_1(U_{F'}))=\gamma_{A,\ell}(\pi_1(U_{\Kbar}))$.  By Lemma~\ref{L:connected independence families}(\ref{L:connected independence families i}), $F'$ is independent of $\ell$ and hence there are only finitely many possibilities for $F$.     We can bound $r$ in terms of the rank of $G^\circ$ which we have already noted can be bounded in terms of $g$.
\end{proof}

Take any $x\geq 2$.  After first excluding a finite number of primes $\ell$,  Theorem~\ref{T:HIT main} implies that
 \begin{align*}
|B_\ell(x)|&= |\{ u \in U(K) : H(u)\leq x,\, \rho_u(\Gal_K)\not\supseteq S\}| \\&\ll_{U,F,|\scrS|,r,m} \,\, (\ell+1)^{3d/2} \cdot x^{[K:\QQ](n + 1/2)} \log x + (\ell+1)^{(6n+15/2)d}.
\end{align*}
By Lemma~\ref{L:Frm} and $|\scrS|\ll_A 1$, we have
 \begin{align*}
|B_\ell(x)|& \ll_{A} \,\, (\ell+1)^{3d/2} \cdot x^{[K:\QQ](n + 1/2)} \log x + (\ell+1)^{(6n+15/2)d}.
\end{align*}

It remains to bound $d$.  By choosing a polarization of $A$ and combining with the Weil pairing on the $\ell$-torsion of $A$, we find that $G$ is isomorphic to an algebraic subgroup of $\GSp_{2g,\FF_\ell}$ by taking $\ell$ sufficiently large.   Since $d$ is equal to the dimension of the derived subgroup of $G^\circ$ it is at most $\dim \Sp_{2g,\FF_\ell} = g(2g+1)$ and hence
\[
|B_\ell(x)| \ll_{A} \,\, (\ell+1)^{3g(2g+1)/2} \cdot x^{[K:\QQ](n + 1/2)} \log x + (\ell+1)^{(6n+15/2)g(2g+1)}.
\]

\section{Proof of Theorem~\ref{T:MAIN}} \label{S:main proof}

Take any constant $b_A$ as in Theorem~\ref{T:geometric gp theory} and define the set 
\[
B :=\{ u\in U(K):  \bbar\rho_{A_u,\ell}(\Gal_K)\not\supseteq \calS_{A,\ell}(\FF_\ell)' \text{ for some prime }\ell\geq b_A\}.
\]
To prove the theorem, it suffices by Proposition~\ref{P:main reduction} to show that $B$ has density $0$.\\  

Take any real number $x\geq 2$.  We now define some finite sets that we will use to study $B$.   Let $r$ be the common rank of the groups $G_{A,\ell}^\circ$, cf.~Proposition~\ref{P:G specialization}(\ref{P:G specialization ii}).  Fix a $b>0$ for which $\OO_K$ has a prime ideal of norm at most $b\log 2$.
\begin{itemize}
\item
Let $B(x)$ be the set of $u\in B$ for which $H(u)\leq x$. 
\item
For a prime $\ell$, let $B_\ell(x)$ be the set of $u\in U(K)$ with $H(u)\leq x$ satisfying $\bbar\rho_{A_u,\ell}(\Gal_K)\not\supseteq \calS_{A,\ell}(\FF_\ell)'$.
\item
Let $R(x)$ be the set of $u\in U(K)$ with $H(u)\leq x$ such that $G_{A_u,\ell} \neq G_{A,\ell}$ for some prime $\ell$.   
\item
Let $T(x)$ be the set of $u\in U(K)$ with $H(u)\leq x$ such that for any non-zero prime ideal $\p$ of $\OO_K$ satisfying $N(\p)\leq b\,\log x$, the abelian variety $A_u$ has bad reduction at $\p$ or the roots of the polynomial $P_{A_u,\p}$ in $\CC^\times$ generate a group that is not isomorphic to $\ZZ^r$.  
\end{itemize}

\begin{lemma} \label{L:technical combo}
Take any $u\in U(K)$ with $H(u)\leq x$ satisfying  $u\notin R(x)\cup T(x)$.  There are positive constants $\gamma$ and $c$, with $\gamma$ depending only on $g$ and $c$ depending only on $K$ and $g$, such that if $\ell \geq c (\max\{[K:\QQ],h(A_u), \log x\})^\gamma$, then $\bbar\rho_{A_u,\ell}(\Gal_K) \supseteq \calS_{A,\ell}(\FF_\ell)'$.
\end{lemma}
\begin{proof}
Since $u \notin T(x)$, there is a non-zero prime ideal $\q$ of $\OO_K$ satisfying $N(\q)\leq b \,\log x$ for which $A_u$ has good reduction at $\q$ and for which the subgroup $\Phi_{A_u,\q}$ of $\CC^\times$ generated by the roots of $P_{A_u,\q}$ is isomorphic to $\ZZ^r$.  This uses our choice of $b$ and $x\geq 2$.  By Theorem~\ref{T:main new} and Theorem~\ref{T:geometric gp theory}(\ref{T:geometric gp theory ii}), we have $\bbar\rho_{A_u,\ell}(\Gal_K) \supseteq \calS_{A_u,\ell}(\FF_\ell)'$ for all primes 
\begin{align} \label{E:ell large enough}
\ell  \geq c \cdot \max(\{[K:\QQ],h(A_u), N(\q)\})^\gamma, 
\end{align}
where $c$ and $\gamma$ are positive constants that depend only on $g$. Since $u \notin R(x)$, we have $G_{A_u,\ell} = G_{A,\ell}$.  In particular, we have $\calG_{A_u,\ell}=\calG_{A,\ell}$ and $\calS_{A_u,\ell}=\calS_{A,\ell}$.  Therefore,  we have $\bbar\rho_{A_u,\ell}(\Gal_K) \supseteq \calS_{A,\ell}(\FF_\ell)'$ for all primes $\ell$ satisfying (\ref{E:ell large enough}).  Finally, since $N(\q)\leq b\log x$, we can replace $N(\q)$ by $b\log x$ in (\ref{E:ell large enough}) and adjust the constant $c$ to obtain the lemma.
\end{proof}

We now bound the Faltings height $h(A_u)$ in terms of $H(u)$.

\begin{lemma}  \label{L:Faltings height bound}
We have
$\max\{1,h(A_u)\} \ll_A  \log H(u) + 1$
for all $u\in U(K)$.
\end{lemma}
\begin{proof}
We recall some results of Faltings \cite{MR861971}.    Let $\defi{A}_g$ be the coarse moduli space of  the moduli stack $\mathfrak{A}_g$ of principally polarized abelian varieties of relative dimension $g$; it is a variety defined over $\QQ$.   There is an integer $r>0$ for which $(\omega_{\calA/\mathfrak{A}_g})^{\otimes r}$ defines a very ample line bundle on $\defi{A}_g$, where $\calA\to \mathfrak{A}_g$ is the universal abelian variety.    Using this line bundle, we will identify $\defi{A}_g$ with a subvariety of a projective space $\PP^m_\QQ$.   Let $\bbar{\defi{A}}_g$ be the Zariski closure of $\defi{A}_g$ in $\PP_{\ZZ}^m$ and let $\calM$ be the induced line bundle $\mathcal{O}(1)$ on $\bbar{\defi{A}}_g$.     In \cite{MR861971}*{\S3}, Faltings defines a hermitian metric $\norm{\cdot}$ on the line bundle induced by $\calM$ on $(\defi{A}_g)_\CC$; it gives rise to a height function $h\colon \defi{A}_g(\Kbar) \to \RR$.  Choose any hermitian metric $\norm{\cdot}_1$ on the line bundle induced by $\calM$ on $(\bbar{\defi{A}}_g)_\CC$; it gives rise to a height function $h_1\colon \bbar{\defi{A}}_g(\Kbar) \to \RR$.   For our implicit constants below, we note that the choices of $r$, $\calM$ and $\norm{\cdot}_1$ depend only on $g$.

Faltings observes that the metric $\norm{\cdot}$ has logarithmic singularities along $\bbar{\defi{A}}_g-\defi{A}_g$, cf.~\cite{MR861971}*{p.~15}.  This implies that 
\[
|h(x) - h_1(x)| \ll_g \log h_1(x) + 1
\]
for all $x\in \defi{A}_g(\Kbar)$; see the proof of \cite{MR861971}*{Lemma~3} or \cite{MR861975}*{Proposition~8.2}.   Therefore, we have
\[
\max\{1,h(x)\} \ll_g \max\{1,h_1(x)\} \ll_g \log H(x) + 1
\]
for all $x\in \defi{A}_g(\Kbar)$, where $H$ is the usual absolute height on $\PP^m(\Kbar)$.

Consider a semistable abelian variety $A$ defined over $K$ that has a principal polarization $\xi$ (the connected N\'eron model of $A$ is a scheme over the ring of integers of $K$ that is semiabelian).  Denote by $x \in \defi{A}_g(K)$ the point on the moduli space corresponding to the pair $(A,\xi)$.     Then we have
\[
h(A) = r \cdot h(x) +O_g( 1);
\]
this is noted in the proof of \cite{MR861971}*{Theorem~1}.  Combining the bounds above, we have
\[
\max\{1,h(A)\} \ll_g \log H(x) + 1;
\]
note that this remains true without the semistable hypothesis since both sides are stable under replacing $K$ by a finite extension.

We finally consider our abelian scheme $A\to U$.    First suppose that $A\to U$ has a principal polarization $\xi$.   There is thus a morphism $\varphi\colon U\to (\defi{A}_g)_K$ such that the pair $(A_u,\xi_u)$ represents the point $\varphi(u) \in \defi{A}_g(K)$ for each $u\in U(K)$, where $\xi_u$ is the specialization of $\xi$ at $u$.  So from above, we find that
\[
\max\{1,h(A_u)\} \ll_g \log H(\varphi(u)) + 1.
\]
The lemma now follows since $\log H(\varphi(u)) \ll_\varphi \log H(u) + 1$ for all $u\in U(K)$, where the implicit constant depends only on the morphism $\varphi\colon U \to (\defi{A}_g)_K \subseteq \PP^m_K$, cf.~\cite{MR1757192}*{\S2.6}.

It remains to consider a general $A\to U$ that need not have a principal polarization.  Define the abelian scheme $B:=(A\times A^\vee)^4 \to U$, where $A^\vee$ is the dual of $A$.    Using Zarhin's trick \cite{MR1175627}*{Ch.~IV Proposition~3.8}, one finds that the abelian scheme $B\to U$ is principally polarized.   So by the case of the lemma already proved, we have $\max\{1,h(B_u)\} \ll_A  \log H(u) + 1$ for all $u\in U(K)$.   The lemma follows since $h(B_u)=8h(A_u)$ for all $u\in U(K)$, cf.~the remarks after Propositions~3.7 and 3.8 in Ch.~IV of \cite{MR1175627} (recall we are using the stable Faltings height).  
\end{proof}

Take $x\geq 2$.  Take any $u \in U(K)$ satisfying $H(u)\leq x$ and $u\notin R(x)\cup T(x)$.  We have $\max\{1,h(A_u)\} \ll_A \log x$ by Lemma~\ref{L:Faltings height bound}.    So by Lemma~\ref{L:technical combo}, there are positive constants $c$ and $\gamma$ such that $\bbar\rho_{A_u,\ell}(\Gal_K) \supseteq \calS_{A,\ell}(\FF_\ell)'$ holds for all $\ell\geq c(\log x)^\gamma$, where $\gamma$ depends only on $g$ and $c$ depends only on $A$.   Therefore,
\[
B(x) \subseteq R(x) \cup T(x)  \cup \bigcup_{b_A\leq \ell \leq c(\log x)^\gamma} B_\ell(x).
\]
In particular, we have
\begin{equation} \label{E:B(x) bound}
|B(x)| \leq |R(x)| + |T(x)| + \sum_{b_A\leq \ell \leq c(\log x)^\gamma} |B_\ell(x)|.
\end{equation}
We now bound the terms on the right hand side of (\ref{E:B(x) bound}).

\subsection{Bounding the sum of the $|B_\ell(x)|$}

Take $c$ and $\gamma$ as in (\ref{E:B(x) bound}).  For each prime $b_A\leq \ell \leq c(\log x)^\gamma$, Theorem~\ref{T:almost there mod ell} implies that $|B_\ell(x)| \ll_A x^{[K:\QQ](n+1/2)} (\log x)^{\gamma'}$, where $\gamma'$ is a positive constant depending only on $g$.  Therefore,
\begin{align} \label{E:B bound}
\sum_{b_A\leq \ell \leq c(\log x)^\gamma} |B_\ell(x)| \ll_A x^{[K:\QQ](n+1/2)} (\log x)^{\gamma' +\gamma} = o(x^{[K:\QQ](n+1)}).
\end{align}

\subsection{Bounding $|R(x)|$} \label{SS:bounding R(x)}

Let $R$ be the set of $u\in U(K)$ such that $G_{A_u,\ell}\neq G_{A,\ell}$ for some $\ell$.  Note that $R(x)$ is the set of $u\in R$ with $H(u)\leq x$.  The set $R$ has density $0$ by Proposition~\ref{P:monodromy independence} and hence
\begin{align}\label{E:R bound}
|R(x)| = o(x^{[K:\QQ](n+1)}).
\end{align}

\subsection{Bounding $|T(x)|$} \label{SS:bounding T(x)}

We fix a prime $\ell\geq b_A$.  Let $\calU$ be the open subscheme of $\PP^n_{\OO_K}$ that is the complement of the Zariski closure of $\PP^n_K-U$ in $\PP^n_{\OO_K}$.    There is an abelian scheme $\calA\to \calU_{\OO}$, where $\OO$ is the ring of $\scrS$-integers in $K$ for some finite set $\scrS$ of nonzero prime ideals of $\OO_K$, such that the fiber over $(\calU_{\OO})_K=U$ is our abelian scheme $A\to U$.   By increasing $\scrS$, we may further assume that it contains all prime ideals that divide $\ell\cdot |\calG_{A,\ell}(\ZZ/\ell\ZZ)|$ and that $\calU(\FF_\p)$ is non-empty for all $\p\notin \scrS$.   There is no harm in replacing $U$ by a non-empty open subvariety since this only removes a density 0 set of rational points.   So after replacing $U$ and increasing $\scrS$, we may further assume that $\calU_{\FF_\p}$ is affine and geometrically irreducible for all non-zero prime ideals $\p\notin \scrS$ of $\OO_K$.

For each integer $e\geq 1$, the $\ell^e$-torsion subscheme $\calA[\ell^e]$ of $\calA$ can be viewed as locally constant sheaf of $\ZZ/\ell^e\ZZ$-modules on $\calU$ that is free of rank $2g$. The fiber of $\calA[\ell^e]$ over $U=(\calU)_K$ is $A[\ell^e]$.   Since $\bbar\rho_{A,\ell^e}$ is the representation associated to $A[\ell^e]$, we find that it arises via base change from a representation of $\pi_1(\calU)$.   Combining these representations together appropriately, we obtain a representation $\varrho_{\calA,\ell}$ of $\pi_1(\calU)$ such that base change gives rise to our representation $\rho_{A,\ell}\colon \pi_1(U)=\pi_1(\calU_K) \to \GL_{V_\ell(A)}(\QQ_\ell)$.    Since $\varrho_{\calA,\ell}$ and $\rho_{A,\ell}$ have the same image, we have $\varrho_{\calA,\ell}\colon \pi_1(\calU)\to \calG_{A,\ell}(\ZZ_\ell)$. 

For a non-zero prime ideal $\p\notin \scrS$ of $\OO_K$ and  a point $u\in \calU(\FF_\p)$, let $\calA_u$ be the abelian variety over $\FF_\p$ that is the fiber of $\calA$ over $u$.   Since $\p\nmid \ell$, we have $P_{\calA_u}(x)=\det(xI-\varrho_{\calA,\ell}(\Frob_u))$, where $P_{\calA_u}(x)$ is the Frobenius polynomial of $\calA_u$.  Let $\Phi_{\calA_u}$ be the subgroup of $\CC^\times$ generated by the roots of the Frobenius polynomial $P_{\calA_u}(x)$.  

\begin{lemma} \label{L:subvariety Y}
There is a closed subvariety $Y\subsetneq G_{A,\ell}^\circ$, stable under conjugation by $G_{A,\ell}$, such that if $\varrho_{\calA,\ell}(\Frob_u) \in G_{A,\ell}^\circ(\QQ_\ell)- Y(\QQ_\ell)$ for a prime ideal $\p\notin \scrS$ of $\OO_K$ and a point $u\in \calU(\FF_\p)$, then $\Phi_{\calA_u}\cong \ZZ^r$.  
\end{lemma}
\begin{proof}
This essentially follows from Theorem~1.2 \cite{MR1441234}; we give a few extra details since this theorem was only stated for representations of $\Gal_K$.

Take any prime ideal $\p\notin \scrS$ and $u\in \calU(\FF_\p)$.   Let $H_{u,\ell}$ be the Zariski closure of the subgroup of $G_{A,\ell}$ generated by $\varrho_{A,\ell}(\Frob_u)$.   The proof of Lemma~1.3(b) of \cite{MR1441234} shows that there are only finitely many possibilities for $(H_{u,\ell})_{\Qbar_\ell}$ up to conjugation by $\GL_{V_\ell(A)}(\Qbar_\ell)$ as we vary over all $\p$ and $u$; note that the proof only uses that $\varrho_{A,\ell}(\Frob_u)$ is semisimple along with information about the valuations of the roots of $P_{\calA_u}(x)$.   The end of the proof of Theorem~1.2 \cite{MR1441234} then shows how to construct a closed subvariety $Y\subsetneq G_{A,\ell}^\circ$, stable under conjugation by $G_{A,\ell}$, such that if $\varrho_{\calA,\ell}(\Frob_u) \in G_{A,\ell}^\circ(\QQ_\ell)- Y(\QQ_\ell)$ for a prime ideal $\p\notin \scrS$ of $\OO_K$ and a point $u\in \calU(\FF_\p)$, then $H_{u,\ell}$ is a maximal torus of $G_{A,\ell}^\circ$.

Now suppose that $T:=H_{u,\ell}$ is a maximal torus of $G_{A,\ell}^\circ$; it remains to show that $\Phi_{\calA_u}\cong \ZZ^r$.  Let $X(T)$ be the group of characters $T_{\Qbar_\ell}\to \GG_{m,\Qbar_\ell}$; it is a free abelian group whose rank is equal to $\dim T = \rank G_{A,\ell}^\circ=r$.  Define the homomorphism $\varphi\colon X(T) \to \CC^\times$, $\alpha\mapsto \iota(\alpha(\varrho_{A,\ell}(\Frob_u)))$, where $\iota$ is any embedding of $\Qbar_\ell$ into $\CC$.  The homomorphism $\varphi$ is injective since otherwise $H_{u,\ell}\neq T$.  Since $\varrho_{A,\ell}(\Frob_u)$ is semisimple with characteristic polynomial $P_{\calA_u}(x)$, we find that the image of $\varphi$ is generated by the roots of $P_{\calA_u}(x)$.   Therefore, we have isomorphisms $\Phi_{\calA_u} \cong X(T)\cong \ZZ^r$.
\end{proof}

For each non-zero prime ideal $\p\notin\scrS$ of $\OO_K$, let  $D_\p$ be the set of  $u\in \PP^n(\FF_\p)$ for which $u\notin \calU(\FF_\p)$ or for which $\Phi_{\calA_u}\not\cong \ZZ^r$.    Define $\delta_\p:=|D_\p|/|\PP^n(\FF_\p)|$.

\begin{lemma} \label{L:deltap supply}
There is a constant $0\leq \delta<1$ such that $\delta_\p\leq \delta$ holds  for infinitely many prime ideals  $\p\notin \scrS$ of $\OO_K$.
\end{lemma}
\begin{proof}
Fix an integer $e\geq 1$.   Take $Y$ as in Lemma~\ref{L:subvariety Y} and define $\calY = Y(\QQ_\ell) \cap \calG_{A,\ell}(\ZZ_\ell)$; it is stable under conjugation by $\calG_{A,\ell}(\ZZ_\ell)$.  Let $\calY_e$ be the image of $\calY$ in $\calG_{A,\ell}(\ZZ/\ell^e\ZZ)$.    Let $\bbar\varrho_{\calA,\ell^e}\colon \pi_1(\calU)\to \calG_{A,\ell}(\ZZ/\ell^e\ZZ)$ be the representation obtained by composing $\varrho_{\calA,\ell}$ with the reduction map $\calG_{A,\ell}(\ZZ_\ell)\to \calG_{A,\ell}(\ZZ/\ell^e\ZZ)$.

Define the finite group $G:=\bbar\varrho_{A,\ell^e}(\pi_1(\calU))=\bbar\rho_{A,\ell^e}(\pi_1(U)) \subseteq \calG_{A,\ell}(\ZZ/\ell^e\ZZ)$.   Define $G_g:=\bbar\varrho_{\calA,\ell^e}(\pi_1((\calU_\OO)_{\Kbar}))=\bbar\rho_{A,\ell^e}(\pi_1(U_{\Kbar}))$; it is a normal subgroup of $G$.   After possibly increasing the finite set $\scrS$, we may assume that
 \[
 \bbar\varrho_{\calA,\ell^e}(\pi_1(\calU_{\FFbar_\p}))=G_g
 \]
 holds for all non-zero prime ideals $\p\notin \scrS$ of $\OO_K$, see Lemma~\ref{L:big mod p monodromy}.   
 
 Take a non-zero prime ideal $\p\notin \scrS$ of $\OO_K$.  From $\bbar\varrho_{\calA,\ell^e}$, base change to $\FF_\p$ gives a homomorphism we will denote by $\varrho_\p\colon \pi_1(\calU_{\FF_\p}) \to G$; uniquely defined up to conjugation by $G$ and satisfying $\varrho_{\p}(\pi_1(\calU_{\FFbar_\p}))=G_g$.   Let $\kappa_\p$ be the unique $G_g$-coset of $G$ that contains $\varrho_{\p}(\Frob_u)$ for all $u\in \calU(\FF_\p)$.  Note that the set $\calY_e \cap \kappa_\p$ is stable under conjugation by $G$.
 
  If $\varrho_\p(\Frob_u) \in \kappa_\p- (\calY_e\cap \kappa_\p)$ for a point $u\in \calU(\FF_\p)$, then $\varrho_{\calA,\ell}(\Frob_u) \notin Y(\QQ_\ell)$ and hence $\Phi_{\calA_u}\cong \ZZ^r$ by Lemma~\ref{L:subvariety Y}.   Therefore,
 \[
 \{u \in \calU(\FF_\p): \varrho_\p(\Frob_u) \in \kappa_\p- (\calY_e\cap \kappa_\p) \} \subseteq \PP^n(\FF_\p)-D_\p.
 \]
 By Theorem~\ref{T:tame equidistribution}, this implies that 
\[
\Big(1- \frac{| \calY_e\cap \kappa_\p|}{|G_g|}\Big)\, |\calU(\FF_\p)| + O_{A,\ell^e}(N(\p)^{n-1/2}\big) \leq |\PP^n(\FF_\p)| - |D_\p|.
\]
Dividing by $|\PP^n(\FF_\p)|$ and using that $|\PP^n(\FF_\p)-\calU(\FF_\p)|\ll_A N(\p)^{n-1/2}$, we deduce that 
\begin{align} \label{E:deltap bound}
\delta_\p \leq  \frac{|\calY_e \cap \kappa_\p|}{|G_g|}  + O_{A,\ell^e}(N(\p)^{-1/2}).
\end{align}

Now suppose that $ \calY_e \cap G$ is a proper subset of $G$.  Then there is a $G_g$-coset $\kappa_0$ of $G$ such that $\calY_e \cap \kappa_0$ is a proper subset of $\kappa_0$.  Let $C$ be the conjugacy class of $G/G_g$ that contains the image of $\kappa_0$.   Define the homomorphism 
\[
\alpha\colon \pi_1(\calU) \xrightarrow{\bbar\varrho_{\calA,\ell^e}} G\to G/G_g.
\] 
The homomorphism $\alpha$ factors through $\Gal_K$; moreover, for $u\in \calU(\FF_\p)$ with $\p\notin \scrS$, the conjugacy class of $\alpha(\Frob_u)$ is represented by the image of $\kappa_\p$ in $G/G_g$.     By the Chebotarev density theorem, there are infinitely many $\p\notin \scrS$ for which the image of $\kappa_\p$ in $G/G_g$ lies in $C$; now take any such $\p$.  The cosets $\kappa_\p$ and $\kappa_0$ are conjugate in $G$.  Since $\calY_e \cap G$ is stable under conjugation by $G$, we have
\[
\frac{|\calY_e \cap \kappa_\p|}{|G_g|} = \frac{|\calY_e \cap \kappa_0|}{|G_g|} <1,
\]
where the inequality uses our choice of $\kappa_0$.  In particular, ${|\calY_e \cap \kappa_\p|}/{|G_g|}\leq  1-1/|G_g|$.    After first excluding a finite number of $\p$, we deduce that $\delta_\p<1$ by (\ref{E:deltap bound}).

So to prove the lemma, it suffices to show that $\calY_e \cap G$ is a proper subset of $G$.   Since $Y\subsetneq G_{A,\ell}^\circ$, the variety $Y$ has  dimension at most $d-1$, where $d:= \dim G_{A,\ell}^\circ$.  So $\calY$ is a $p$-adic analytic manifold of dimension at most $d-1$ and hence $|\calY_e| \ll_\calY \ell^{e(d-1)}$, cf.~\cite{MR644559}*{The\'eor\`eme~8}.   Since $\ell\geq b_A$, $\calG_{A,\ell}$ is smooth and $[\calG_{A,\ell}(\ZZ_\ell) : \rho_{A,\ell}(\pi_1(U))] \ll_A 1$ by Theorem~\ref{T:geometric gp theory}(\ref{T:geometric gp theory i}) and (\ref{T:geometric gp theory iii}).   Therefore, $[\calG_{A,\ell}(\ZZ/\ell^e\ZZ) : \bbar\rho_{A,\ell^e}(\pi_1(U))]\ll_A 1$ and hence $|G| \gg_A |\calG_{A,\ell}(\ZZ/\ell^e\ZZ)| \gg_A \ell^{ed}$.   We have not imposed any conditions on the integer $e\geq 1$ yet.  So using $|G|\gg_A \ell^{ed}$ and $|\calY_e| \ll_\calY \ell^{e(d-1)}$, we choose $e\geq 1$ large enough so that $|G|>|\calY_e|$ and hence $\calY_e \cap G$ is a proper subset of $G$. 
\end{proof}

By Lemma~\ref{L:deltap supply}, there are infinitely many non-zero prime ideals $\p_1,\p_2,\ldots $ of $\OO_K$ that are not in $\scrS$ and satisfy $\delta_{\p_i}\leq \delta$ for some $0\leq \delta<1$.  Take any integer $m\geq 1$.

Let $D$ be the set of $u\in \PP^n(K)$ for which the image in $\PP^n(\FF_{\p_i})$ under the reduction modulo $\p_i$ lies in $D_{\p_i}$ for all $1\leq i \leq m$.  The subset $D$ of $\PP^n(K)$ has density $\prod_{i=1}^m \delta_{\p_i}$.    Note that if the reduction $\bbar{u}$ of a point $u\in U(K)$ modulo $\p_i$ lies in $\PP^n(\FF_{\p_i})-D_{\p_i}$, then $A_u=\calA_u$ has good reduction at $\p_i$ and $\Phi_{A_u,\p}=\Phi_{\calA_{\bbar{u}}} \cong \ZZ^r$.    Therefore, we have $T(x) \subseteq D$ for all sufficiently large $x$.  Since $D$ has density $\prod_{i=1}^m \delta_{\p_i}$, we deduce that 
\[
\limsup_{x\to +\infty} \frac{|T(x)|}{|\{u\in \PP^n(K): H(u)\leq x\}|} \leq \prod_{i=1}^m \delta_{\p_i} \leq \delta^m.
\]
Since $0\leq \delta<1$ and since $m\geq 1$ was arbitrary, this implies that $\lim_{x\to +\infty} {|T(x)|}/{|\{u\in \PP^n(K): H(u)\leq x\}|}=0$.  Equivalently, we have
\begin{align} \label{E:T bound}
|T(x)| = o(x^{[K:\QQ](n+1)}).
\end{align}

\subsection{End of the proof}

Using (\ref{E:B(x) bound}) with (\ref{E:B bound}), (\ref{E:R bound}) and (\ref{E:T bound}),  we deduce that $|B(x)|=o(x^{[K:\QQ](n+1)})$ and hence $B$ has density 0.  As already noted, the theorem now follows directly from Proposition~\ref{P:main reduction}.

\section{Proof of Theorem~\ref{T:general base}} \label{S:general base proof}

After replacing $X$ by a non-empty open subvariety, and restricting $A$, we may assume that there is an \'etale morphism $\varphi\colon X\to U$, where $U$ is a non-empty open subvariety of $\PP^n_K$ and $n$ is the dimension of $X$.   \\

We first consider the special case where $\varphi\colon X\to U$ is a Galois cover.   Denote the degree of $\varphi$ by $d$.  Define 
\[
B:=\Res_{X/U}(A), 
\]
i.e., the Weil restriction of $A$ along the morphism $\varphi$; it is an abelian scheme of relative dimension $g\cdot d$ over $U$.   Note that for any $U$-scheme $S$, we have $B(S)=A(S\times_U X)$.

Using $\varphi$, we can identify $\pi_1(X)$ with a normal subgroup of $\pi_1(U)$.  Let $G$ be the Galois group of $\varphi$, i.e., the group of  automorphisms $\sigma$ of $X$ such that $\varphi\circ \sigma = \varphi$.    For each $\sigma\in G$, let $A^\sigma$ be the abelian scheme over $X$ obtained by composing $A\to X$ with $\sigma^{-1}$.   Using that $\varphi$ is a Galois cover, we have a natural isomorphism
\begin{align}  \label{E:B U X}
B\times_U X = \prod_{\sigma\in G} A^\sigma
\end{align}
of abelian schemes over $X$ and hence an isomorphism $
\rho_{B}|_{\pi_1(X)}=\prod_{\sigma\in G}\rho_{A^\sigma}$ of representations of $\pi_1(X)$.

For any number field $L/K$ and point $x\in X(L)$, taking the fiber of (\ref{E:B U X}) above $x$ gives a natural isomorphism
\[
B_{\varphi(x)} = \prod_{\sigma\in G} A_{\sigma(x)}
\]
of abelian varieties over $L$; the fiber of $A^\sigma$ over $x$ is $A_{\sigma(x)}$.    Therefore, we have an equality $\rho_{B_{\varphi(x)}} = \prod_{\sigma\in G} \rho_{A_{\sigma(x)}}$ of representations of $\Gal_L$.  By considering specializations, we find that 
\[
[\rho_B(\pi_1(X)): \rho_{B_{\varphi(x)}}(\Gal_L)] = \Big[ (\prod_{\sigma\in G} \rho_{A_{\sigma}})(\pi_1(X)):  (\prod_{\sigma\in G} \rho_{A_{\sigma(x)}})(\Gal_L)\Big].
\] 
Therefore, $[\rho_{A}(\pi_1(X)):\rho_{A_{x}}(\Gal_L)] \leq  [\rho_B(\pi_1(X)): \rho_{B_{\varphi(x)}}(\Gal_L)] \leq [\rho_B(\pi_1(U)): \rho_{B_{\varphi(x)}}(\Gal_L)]$.

By Theorem~\ref{T:MAIN}, there is a constant $C$ such that $[\rho_B(\pi_1(U)): \rho_{B_{u}}(\Gal_K)]\leq C$ holds for infinitely many $u\in U(K)$.  Take any such $u \in U(K)$.   There is a field $L/K$ with $[L:K]\leq d$ and a point $x \in X(L)$ such that $\varphi(x)=u$.  Therefore,
\[
[\rho_{A}(\pi_1(X)):\rho_{A_{x}}(\Gal_L)] \leq [\rho_B(\pi_1(U)): \rho_{B_{\varphi(x)}}(\Gal_L)] \leq [\rho_B(\pi_1(U)): \rho_{B_u}(\Gal_K)]\cdot [L:K] \leq C\cdot d.
\]
This proves that $[\rho_{A}(\pi_1(X)):\rho_{A_{\tilde{x}}}(\Gal_{k(\tilde{x})})]\leq C\cdot d$ and $[k(\tilde{x}):K]\leq d$, where $\tilde x$ is the closed point of $X$ corresponding to $x$ (one can identify $\tilde x$ with the $\Gal_K$-orbit of $x$ in $X(\Kbar)$).  There are infinitely many such closed points $\tilde x$ since we have infinitely many $u\in U(K)$ for which $[\rho_B(\pi_1(U)): \rho_{B_{u}}(\Gal_K)]\leq C$.  This completes the proof in the case where $\varphi$ is a Galois cover.
\\

We now consider the general case.    There is an \'etale morphism $\psi\colon X' \to X$ such that its composition with $\varphi$ gives a Galois cover $\varphi'\colon X'\to U$.  To prove Theorem~\ref{T:general base} there is no harm in replacing $K$ by a finite extension $K'$ and $A$ by its base change over $X_{K'}$.   So without loss of generality, we may assume that $X'$ is a geometrically irreducible variety defined over $K$.  
 
  Let $A'\to X'$ be the base change of $A$ by $\psi$; it is an abelian scheme over $X'$.   Since $\varphi'$ is Galois, the case of Theorem~\ref{T:general base} already proved shows that there are integers $d$ and $C$ such that $[\rho_{A'}(\pi_1(X')): \rho_{A'_{x'}}(\Gal_{k(x')})] \leq C$ holds for infinitely many closed points $x'$ of $X'$ satisfying $[k(x'):K] \leq d$.  Take any such closed point $x'$ of $X'$ and define the closed point $x=\varphi(x')$ of $X$.    Using $\psi$, we can view $k(x')$ as an extension of $k(x)$ of degree at most $\deg \psi$.  In particular, $[k(x):K] \leq [k(x'):K]\leq d$.   We have an isomorphism $A_x$ between $A'_{x'}$ as abelian varieties over $k(x')$.  Therefore, we have 
\[
[\rho_{A}(\pi_1(X')): \rho_{A_x}(\Gal_{k(x')})] = [\rho_{A'}(\pi_1(X')): \rho_{A'_{x'}}(\Gal_{k(x')})].
\]
and hence 
\begin{align*}
[\rho_{A}(\pi_1(X)): \rho_{A_x}(\Gal_{k(x)})]
&\leq  [\rho_{A}(\pi_1(X')): \rho_{A_x}(\Gal_{k(x')})] \cdot \deg \psi \\
&= [\rho_{A'}(\pi_1(X')): \rho_{A'_{x'}}(\Gal_{k(x')})] \cdot \deg \psi \\
&\leq C\cdot \deg \psi.
\end{align*}
Therefore, $[k(x):k]\leq d$ and $[\rho_{A}(\pi_1(X)): \rho_{A_x}(\Gal_{k(x)})]\leq C\cdot  \deg \psi$.   Finally, there are infinitely many such closed points $x$ of $X$ since they arose from infinitely many closed points $x'$ of $X'$.


\end{document}